\documentclass{amsart}
\usepackage{euscript,amsmath, amssymb, amsfonts}
\usepackage{mathrsfs}
\usepackage{mathabx}
\usepackage{stmaryrd}
\usepackage{graphicx}
\usepackage{upgreek}
\usepackage[unicode]{hyperref}

\numberwithin{equation}{section}

\newtheorem{theorem}{Theorem}[section]
\newtheorem{lemma}[theorem]{Lemma}
\newtheorem{proposition}[theorem]{Proposition}
\newtheorem{corollary}[theorem]{Corollary}

\theoremstyle{definition}
\newtheorem{definition}[theorem]{Definition}
\newtheorem{example}[theorem]{Example}

\theoremstyle{remark}
\newtheorem{remark}[theorem]{Remark}

\newcommand{\fm}[2]{#1\!:\!#2}
\newcommand{\R}{{\mathbb R}}
\newcommand{\N}{{\mathbb N}}

\newcommand{\C}{{\mathbb C}}
\newcommand{\K}{{\mathbb K}}
\newcommand{\as}[2][]{\llbracket #2\rrbracket_{#1}}
\newcommand{\q}[1]{\mathscr H[#1]}
\newcommand{\dotleq}{\buildrel \textstyle .\over \leq} 
\newcommand{\mb}[1]{\mathbf{#1}}
\newcommand{\mf}[1]{\mathfrak{#1}}
\newcommand{\ms}[1]{\mathscr{#1}}
\mathchardef\mhyphen="2D
\newcommand{\sett}[3]{\left\{#1\,{\left|\,\vphantom{{#1}{#3}}#2\right.}\right.\\ \left.\vphantom{{#1}{#2}}#3\right\}}
\newcommand{\set}[2]{\left\{#1\,{\left|\,\vphantom{#1}#2\right.}\right\}}
\newcommand{\sm}[3]{({#1})\mhyphen\mathrm{sum}_{#2}{#3}}
\newcommand{\smm}[2]{({#1})\mhyphen\mathrm{sum}\,{#2}}
\newcommand{\ex}[3]{{#2}\mhyphen\mathrm{ex}_{#1}\,{#3}}
\newcommand{\card}[1]{\#\!{#1}}
\newcommand{\im}{\mathop{\mathrm{Im}}}
\newcommand{\bigS}{\mathrm S}

\makeatletter
\DeclareRobustCommand\bigop[2][1]{%
  \mathop{\vphantom{\sum}\mathpalette\bigop@{{#1}{#2}}}\slimits@
}
\newcommand{\bigop@}[2]{\bigop@@#1#2}
\newcommand{\bigop@@}[3]{%
  \vcenter{%
    \sbox\z@{$#1\sum$}%
    \hbox{\resizebox{\ifx#1\displaystyle#2\fi\dimexpr\ht\z@+\dp\z@}{!}{$\m@th#3$}}%
  }%
}
\makeatother

\newcommand{\finsum}{\DOTSB\bigop[.9]{\Upsigma}}

\makeatletter
\newcommand{\customitem}[1]{\item[#1]\phantomsection\def\@currentlabel{#1}}
\makeatother

\begin{document}

\title{Measure theory without infinities}

\author{A.G.~Smirnov}
\address{I.~E.~Tamm Theory Department, P.~N.~Lebedev
	Physical Institute, Leninsky prospect 53, Moscow 119991, Russia}
\email{smirnov@lpi.ru}	

\author{M.S.~Smirnov}
\address{Lomonosov MSU, Faculty of Computational Mathematics and Cybernetics, Moscow, Russia, 119991}
\email{matsmir98@gmail.com}	

\keywords{Group-valued measure, vector measure, positive measure, spaces of measures, extension of contents}
\subjclass[2020]{28A12,28A33, 28B05,28B10}

\begin{abstract}
  The aim of this paper is to develop a framework for measure theory that avoids infinities and allows for the uniform treatment of positive and vector measures. Our approach is based on a modification of the notion of measure, which supplements the usual $\sigma$-additivity requirement with a suitable maximality condition. To each Hausdorff topological vector space $\mf A$ and $\sigma$-ring $\mathcal{Q}$, we associate a vector space $\ms M(\mathcal{Q},\mf A)$ of `infinite' $\mf A$-valued measures corresponding to $\mathcal{Q}$. In particular, the positive elements of $\ms M(\mathcal{Q},\R)$ are naturally identified with the $\sigma$-finite positive measures defined on $\mathcal{Q}$, thus placing positive and signed measures within the same setting. Finally, extension results for group-valued contents due to Sion and Weber are reformulated and refined within this new framework.
\end{abstract}

\maketitle

\section{Introduction}
\label{s_intro}

In measure theory, positive measures are traditionally allowed to take infinite values. Since infinities cannot be subtracted, this practice, which dates back to Carath\'eodory's analysis of outer measures~\cite{Caratheodory1914}, gives rise to a kind of divide between positive measures, on the one hand, and vector measures, on the other. A clear manifestation of this dichotomy is the common tendency to consider positive and vector measures separately. In the present paper, we propose a unifying setting that avoids infinite values and treats positive and vector measures uniformly. 

There are essentially two approaches to constructing spaces of scalar measures. The first is simply to consider the sets $\ms M_0(\mathcal{Q},\R)$ and $\ms M_0(\mathcal{Q},\C)$ of $\R$- and $\C$-valued measures\footnote{For now, if $\mf A$ is a Hausdorff topological Abelian monoid, by an $\mf A$-valued measure we mean a $\sigma$-additive $\mf A$-valued function whose domain is a $\delta$-ring. This definition will be modified later.} on a $\delta$-ring $\mathcal{Q}$. Endowed with pointwise addition and scalar multiplication, these sets become real and complex vector spaces respectively. For example, if $X$ is a Hausdorff locally compact space and $\mathcal{Q}$ is the $\delta$-ring generated by all compact $G_\delta$-subsets of $X$, then $\ms M_0(\mathcal{Q},\R)$ and $\ms M_0(\mathcal{Q},\C)$ can be identified with the spaces of real and complex Baire (or equivalently, Radon) measures on $X$. The second approach is to put infinities into play and consider the set $\ms M_0(\mathcal{Q},\mf R)$, where $\mf R$ is the monoid $[0,\infty]$ equipped with its usual topology. This set, endowed with pointwise addition, is a monoid but not a vector space. It is easy to see that every $\mf R$-valued measure $\mu$ can be uniquely extended to an $\mf R$-valued measure on the $\sigma$-ring $\sigma(D_\mu)$ generated by the domain $D_\mu$ of $\mu$. For this reason, it suffices to consider the spaces $\ms M_0(\mathcal{Q},\mf R)$ when $\mathcal{Q}$ is a $\sigma$-ring. For the same reason, the submonoid of positive elements of $\ms M_0(\mathcal{K},\R)$ is embedded in $\ms M_0(\mathcal{Q},\mf R)$ whenever $\mathcal{K}$ is a $\delta$-ring such that $\sigma(\mathcal{K}) = \mathcal{Q}$.  

Following Halmos~\cite[p.~31]{Halmos}, we say that a map $\mu$ taking values in the extended real line is $\sigma$-finite
if every set in $D_\mu$ can be covered by a countable family of elements of $D^f_\mu = \set{A}{A\in D_\mu\mbox{ and }\mu(A)\in\R}$.  The submonoid of $\ms M_0(\mathcal{Q},\mf R)$ consisting of all its $\sigma$-finite elements will be denoted by $\ms M_\sigma(\mathcal{Q},\mf R)$. A typical example of a non-$\sigma$-finite element of $\ms M_0(\mathcal{Q},\mf R)$ is the map $\mu$ on $\mathcal{Q}$ such that $\mu(\varnothing) = 0$ and $\mu(A) = \infty$ for all nonempty $A\in\mathcal{Q}$. It is intuitively clear that measures of this kind are excessively infinite and are essentially unrelated to $\mathcal{Q}$.\footnote{Non-$\sigma$-finite measures typically arise in frameworks based on $\sigma$-algebras rather than general $\sigma$-rings. For example, the counting measure on an uncountable set $X$ is non-$\sigma$-finite when considered on the power set of $X$ and is $\sigma$-finite when restricted to the $\sigma$-ring of all countable subsets of $X$. Thus, non-$\sigma$-finite measures become unnecessary if we do not insist on working with $\sigma$-algebras and instead allow general $\sigma$-rings, as in~\cite{Halmos}.} Therefore, the elements of $\ms M_\sigma(\mathcal{Q},\mf R)$ can be viewed as the true positive measures on $\mathcal{Q}$. In what follows, we focus primarily on the $\sigma$-finite case.  

The first of the approaches described above was used, notably, in~\cite{BourbakiIntegration}, where the framework of Radon measures was adopted from the outset. Its advantage is the vector space structure, which allows for a transparent treatment of signed and complex measures. On the other hand, the space $\ms M_\sigma(\mathcal{Q},\mf R)$ generally contains more positive measures than any of the spaces $\ms M_0(\mathcal{K},\R)$ with $\sigma(\mathcal{K}) = \mathcal{Q}$ (for example, there are positive $\sigma$-finite Borel measures on $\R$ that are not Radon measures). The second approach is therefore preferable if one is mostly interested in positive measures. In particular, it was adopted in the classical expositions by Saks~\cite{Saks1937} and Halmos~\cite{Halmos}, and in many subsequent textbooks (see, e.g., \cite{Bauer,Bogachev,Cohn,FremlinV1}) that focus primarily on positive measures.\footnote{These authors rarely speak of spaces of measures but rather use expressions such as `a measure $\mu$ on a measurable space $(X,\mathcal{Q})$' or `a measure space $(X,\mathcal{Q},\mu)$'. This difference in language is inessential: in all cases, $\mf R$-valued measures on a fixed $\sigma$-ring $\mathcal{Q}$ are considered. The key point is that infinities are allowed.} In this setting, however, the treatment of signed measures is either restricted to finite measures (i.e., the elements of $\ms M_0(\mathcal{Q},\R)$) or tends to be awkward because of infinities (see the discussion of Halmos's definition of signed measures below). The treatment of complex measures is ordinarily restricted to $\ms M_0(\mathcal{Q},\C)$.

Let $\mathcal{Q}$ be a $\sigma$-ring. In this paper, we construct an ordered vector space $\ms M(\mathcal{Q},\R)$ and a complex vector space $\ms M(\mathcal{Q},\C)$ such that the following conditions are satisfied:
\begin{enumerate}
\item [(1)] $\ms M(\mathcal{Q},\R)$ and $\ms M(\mathcal{Q},\C)$ consist of $\R$- and $\C$-valued measures respectively,
\customitem{(2)}\label{item:2} $\ms M(\mathcal{Q},\R)$ is a real subspace of $\ms M(\mathcal{Q},\C)$ and every $\mu\in \ms M(\mathcal{Q},\C)$ has a unique representation $\mu = \mu'+i\mu''$, where $\mu',\mu''\in \ms M(\mathcal{Q},\R)$,
\customitem{(3)} \label{item:3} $\ms M_\sigma(\mathcal{Q},\mf R)$ is naturally identified with the submonoid of positive elements of $\ms M(\mathcal{Q},\R)$,
\customitem{(4)} \label{item:4} $\ms M_0(\mathcal{K},\R)$ and $\ms M_0(\mathcal{K},\C)$ are naturally embedded in $\ms M(\mathcal{Q},\R)$ and $\ms M(\mathcal{Q},\C)$, respectively, whenever $\mathcal{K}$ is a $\delta$-ring and $\sigma(\mathcal{K}) = \mathcal{Q}$. 
\end{enumerate}
As a result, we unify the two approaches discussed above while avoiding infinite values in the treatment of positive measures. The elements of $\ms M(\mathcal{Q},\R)$ and $\ms M(\mathcal{Q},\C)$ can be viewed, respectively, as `infinite' signed and complex measures corresponding to $\mathcal{Q}$.

In the 1970s and 1980s, a number of works (see, e.g., \cite{FoxMorales1983,Lipecki1983,Sion1973,Weber1976}) studied additive set functions taking values in uniform monoids.\footnote{The authors of these works used the term `semigroup' instead of `monoid'. We prefer the latter because it implies the existence of a neutral element, which is always assumed in measure theory.} One apparent motivation for considering monoids, rather than restricting attention to groups, was to provide a unified treatment of group-valued measures (including vector measures) and $\mf R$-valued measures. Our perspective is different: in the framework developed here, there is no need to consider $\mf R$-valued measures at all. 

The spaces $\ms M(\mathcal{Q},\R)$ and $\ms M(\mathcal{Q},\C)$ will be obtained as special cases of more general spaces $\ms M(\mathcal{Q},\mf A)$, which can be defined for an arbitrary Hausdorff topological Abelian monoid (HTM for short) $\mf A$. We now briefly outline the corresponding construction.

A map $\mu$ is said to be \emph{hosted by} $\mathcal{Q}$ if $\mathcal{Q} = \sigma(D_\mu)$. In this case, $\mathcal{Q}$ is called the \emph{host $\sigma$-ring} of $\mu$. The host $\sigma$-ring plays the role of the fixed $\sigma$-ring of measurable sets in the classical theory, whereas $D_\mu$ records the sets on which $\mu$ is actually defined.
Maps $\mu$ and $\nu$ are called \emph{similar} (denoted by $\mu\asymp\nu$) if they have the same host $\sigma$-ring and coincide on $D_\mu\cap D_\nu$. We now modify the definition of a measure as follows. Given an HTM $\mf A$, we say that $\mu$ is an \emph{$\mf A$-valued premeasure} if it is an $\mf A$-valued measure in the sense used above, i.e., an $\mf A$-valued $\sigma$-additive function defined on a $\delta$-ring. Similarity is an equivalence relation on the class of all $\mf A$-valued premeasures. An $\mf A$-valued premeasure $\mu$ is called an \emph{$\mf A$-valued measure} if it satisfies the following maximality condition:
\begin{enumerate}
\customitem{(M)} \label{(M)} $\mu$ is an extension of every $\mf A$-valued premeasure $\nu$ such that $\nu\asymp\mu$.
\end{enumerate}
It turns out that every $\mf A$-valued premeasure is similar to a unique $\mf A$-valued measure (this is one of the main results of this paper). In other words, every equivalence class with respect to $\asymp$ contains exactly one measure. Given an $\mf A$-valued premeasure $\mu$, the $\mf A$-valued measure $\nu$ such that $\nu\asymp\mu$ is called the \emph{$\mf A$-valued measure associated with $\mu$} and is denoted by $\as[\mf A]\mu$.

For every $\sigma$-ring $\mathcal{Q}$, we define $\ms M(\mathcal{Q},\mf A)$ to be the set of all $\mf A$-valued measures hosted by $\mathcal{Q}$. We endow this set with addition by putting $\mu + \nu = \as[\mf A]{\mu\dotplus\nu}$ for every $\mu,\nu\in\ms M(\mathcal{Q},\mf A)$, where $\mu\dotplus\nu$ is the pointwise sum of $\mu$ and $\nu$ defined on $D_\mu\cap D_\nu$ (it can be shown that $\mu\dotplus\nu$ is hosted by $\mathcal{Q}$ and, hence, $\mu+\nu\in\ms M(\mathcal{Q},\mf A)$). This makes $\ms M(\mathcal{Q},\mf A)$ into an Abelian monoid. If $\mathcal{K}$ is a $\delta$-ring with $\sigma(\mathcal{K}) = \mathcal{Q}$, then the space $\ms M_0(\mathcal{K},\mf A)$ embeds naturally in $\ms M(\mathcal{Q},\mf A)$ via the map $\mu\mapsto \as[\mf A]\mu$.

The spaces $\ms M(\mathcal{Q},\mf A)$ thus constructed have the following key property: if the HTM $\mf A$ is a topological group, then $\ms M(\mathcal{Q},\mf A)$ is a group. It is here that condition~\ref{(M)} comes into play: if we were only interested in a monoid structure, it would suffice to define $\ms M(\mathcal{Q},\mf A)$ as the set of all premeasures hosted by $\mathcal{Q}$, endowed with pointwise addition. Let $\K = \R$ or $\K = \C$ and suppose $\mf A$ is a topological vector space over $\K$. Then we endow $\ms M(\mathcal{Q},\mf A)$ with scalar multiplication by putting $k\mu = \as[\mf A]{k\cdot\mu}$ for every $k\in\K$ and $\mu\in \ms M(\mathcal{Q},\mf A)$, where $k\cdot\mu$ denotes the pointwise scalar multiple of $\mu$ by $k$, defined on $D_\mu$. This makes $\ms M(\mathcal{Q},\mf A)$ into a vector space over $\K$. In particular, $\ms M(\mathcal{Q},\R)$ and $\ms M(\mathcal{Q},\C)$ are real and complex vector spaces respectively.

We endow $\ms M(\mathcal{Q},\R)$ with the order defined by setting $\mu\leq \nu$ if and only if $\mu(A)\leq\nu(A)$ for all $A\in D_\mu\cap D_\nu$. This makes $\ms M(\mathcal{Q},\R)$ into an ordered vector space. If $\mu\in \ms M_\sigma(\mathcal{Q},\mf R)$, then $D^f_\mu$ is a $\delta$-ring hosted by $\mathcal{Q}$ and the restriction $\mf F\mu$ of $\mu$ to $D^f_\mu$ belongs to $\ms M(\mathcal{Q},\R)$ (infinities of $\mu$ enforce condition~\ref{(M)} for $\mf F\mu$). The correspondence $\mu\mapsto\mf F\mu$ induces a natural isomorphism of ordered monoids between $\ms M_\sigma(\mathcal{Q},\mf R)$ and the submonoid of positive elements of $\ms M(\mathcal{Q},\R)$.

Let $\mf R_\uparrow$ and $\mf R_\downarrow$ be, respectively, the monoids $(-\infty,\infty]$ and $[-\infty,\infty)$ endowed with their natural topologies. Halmos~\cite[p.~118]{Halmos} and some other authors (see, e.g., \cite[pp.~121--122]{Cohn}) define signed measures on a $\sigma$-ring $\mathcal{Q}$ as the elements of the set $\ms M_0(\mathcal{Q},\mf R_\uparrow)\cup \ms M_0(\mathcal{Q},\mf R_\downarrow)$, which is not even a monoid. As in the case of $\mf R$, the submonoids of $\ms M_0(\mathcal{Q},\mf R_\uparrow)$ and $\ms M_0(\mathcal{Q},\mf R_\downarrow)$ consisting of their respective $\sigma$-finite elements embed naturally in $\ms M(\mathcal{Q},\R)$ via the map $\mu\mapsto \mf F\mu$. The images of these embeddings are merely submonoids of $\ms M(\mathcal{Q},\R)$ with no particularly nice properties. Halmos's definition can therefore be viewed as an artifact of using infinities.  

A different approach to signed measures in the presence of infinities was proposed by K\"onig~\cite{Konig1999}, whose signed contents and measures are equivalence classes of pairs of positive extended-valued contents and measures. It can be shown that, when $\mathcal{Q}$ is a $\sigma$-algebra, signed measures in König's sense representable by a pair of $\sigma$-finite measures can be naturally identified with elements of $\ms M(\mathcal{Q},\R)$. An advantage of our construction is that it involves ordinary set functions rather than equivalence classes. Moreover, it is by no means restricted to the real case.    

The notion of measure proposed in the present paper is useful not only for constructing spaces of measures, but also as a natural framework for certain results concerning extensions of group-valued $\sigma$-additive contents ($\sigma$-contents for short). In fact, combining the associated-measure operation with a variant of an extension result due to Weber~\cite{Weber1976} immediately yields the existence part of the following statement: if $\mf A$ is a sequentially complete group, then every weakly exhaustive $\mf A$-valued $\sigma$-content admits a unique extension to an $\mf A$-valued measure hosted by the same $\sigma$-ring. Moreover, by employing the technique used to construct associated measures, we obtain a partial characterization of the domain of this extension. We thus arrive at a theorem that covers both Weber's result and extensions to $\sigma$-rings described by Sion~\cite{Sion1969}.  

The present paper develops the framework briefly sketched in Appendix~A of~\cite{Smirnov2015}. Condition~(a) in the 
definition of a measure adopted there essentially corresponds to Proposition~\ref{p_seqcomplSigma} of the present paper. Accordingly, the existence of the associated measure was stated in~\cite[Lemma~A.3]{Smirnov2015} when $\mf A$ is a sequentially complete topological group, with a proof given only in the positive case. The advantage of condition~\ref{(M)} introduced above is that it allows us to prove the existence of the associated measure without any restrictions on $\mf A$.    

The associated-measure operation also provides the basic extension step in our approach to integration theory and the Radon–Nikodym theorem. In the scalar case, given a measure $\mu$ and a measurable function $f$, integration of $f$ over those $A\in D_\mu$ on which $f$ is bounded defines a premeasure; the measure associated with this premeasure is the multiple $f(s)\,d\mu(s)$, whose domain determines the sets over which $f$ is integrable. For example, if $\mathcal B$ is the Borel $\sigma$-ring of $\mathbb R$, then measures in $\ms M(\mathcal B,\mathbb C)$ can be multiplied by \emph{arbitrary} Borel functions, whereas, for Radon measures, the multiplier must be locally integrable with respect to the total variation of the measure. An analogous construction applies to vector measures, with boundedness replaced by relative compactness of $f(A)$. In this case, it naturally leads to an unconditional bilinear integral, in contrast to Dinculeanu's theory~\cite{Dinculeanu}, which is based on absolute integrability. These topics will be considered in a subsequent paper. Other problems that we plan to address include the Jordan decomposition in the space $\ms M(\mathcal{Q},\R)$, the order structure of this space (which is actually a conditionally complete lattice), and the relationship between the associated-measure operation and various completion procedures described in~\cite{SmirnovCompl} (they are expected to essentially commute). 

The paper is organized as follows. In Section~\ref{s2}, we fix our notation and collect necessary facts concerning $\delta$- and $\sigma$-rings of sets. Section~\ref{s_add} provides precise definitions of contents, $\sigma$-contents, and premeasures, alongside their basic properties. The subsequent section is also preliminary in nature and deals with extension results for semi-rings. The latter are used to prove the existence of associated measures in Section~\ref{s_meas}, which is concerned with the notion of measure and is the central part of the present paper. In Section~\ref{s_ass}, we provide an alternative description of the associated measure for \emph{regular} HTMs, which is often more useful in practice than the definition given in Section~\ref{s_meas}. The spaces $\ms M(\mathcal{Q},\mf A)$ are defined and studied in Section~\ref{s_spaces}. In particular, conditions~\ref{item:2} and~\ref{item:4} for the spaces of scalar measures formulated above are established there, whereas the verification of condition~\ref{item:3} is postponed to Section~\ref{s_pos}, which is devoted to positive measures. The last two sections of this paper treat extensions of group-valued $\sigma$-additive functions. Specifically, Section~\ref{s_unique} addresses the uniqueness of such extensions, and Section~\ref{s_ext} establishes the extension theorem for $\sigma$-contents that was outlined above. In Appendix~\ref{app_B}, we give the necessary background material on summation in monoids. In particular, it is used in Appendix~\ref{app_C} to prove Proposition~\ref{l_semi} for arbitrary, not necessarily regular, HTMs. Appendix~\ref{app_counter} contains two counterexamples showing that Theorems~\ref{ll16} and~\ref{t_ext} cannot be strengthened in certain natural directions.    

\section{Preliminaries}
\label{s2}

We let $\N$, $\mathbb Q$, $\R$ and $\C$  denote the set of strictly positive integer numbers, the set of rational numbers,  the set of real numbers, and the set of complex numbers respectively. A set $A$ is called countable if there is a bijection between $A$ and a subset of $\N$ (thus, finite sets are countable). Given a map $f$, we let $D_f$ and $\im f$ denote the domain of definition and the range of $f$ respectively. We write $\fm{f}{I}$ if $f$ is a map and $D_f = I$. A map $f$ is called $A$-valued if $\im f\subset A$. If $A\subset D_f$, then $f|_A$ denotes the restriction of $f$ to $A$. Maps $f$ and $g$ are said to coincide on $A$ if $A\subset D_f\cap D_g$ and $f|_A = g|_A$.

The words `map' and `family' will be used interchangeably. A family $\fm{f}{I}$ is called disjoint if $f(i)\cap f(j) = \varnothing$ for every $i,j\in I$ such that $i\neq j$. We say that a family $\fm{f}{I}$ is a partition of a set $A$ if it is disjoint and $A = \bigcup_{i\in I}f(i)$. If $f$ is a partition of $A$ and $\im f\subset
\mathcal{Q}$, then we say that $f$ is a partition of $A$ into elements of $\mathcal{Q}$. A family $\fm{f}{I}$ is called finite (countable) if $I$ is finite (resp., countable). A family $\fm{f}{I}$ is called empty if $I = \varnothing$ (and, hence, $f = \varnothing$). Otherwise it is called nonempty.

If $\mf f$ is an expression in which the variable $i$ may occur and which is well defined for every $i\in I$, then $\{\mf f\}_{i\in I}$ denotes the map $f$ such that $D_f = I$ and $f(i) = \mf f$ for every $i\in I$. Alternatively, it will be called the map $i\mapsto\mf f$ on $I$.   

A set $\mathcal Q$ is said to be closed under finite (countable) unions if $\bigcup_{i\in I}\mb A(i)\in \mathcal Q$ for every finite (resp., countable) family $\mb A\colon I\to \mathcal Q$. Since $\varnothing$ is the union of the empty family, we have $\varnothing\in\mathcal Q$ whenever $\mathcal Q$ is closed under finite unions. It follows by induction that $\mathcal Q$ is closed under finite unions if and only if $\varnothing\in\mathcal Q$ and $A\cup B\in\mathcal Q$ for every $A,B\in\mathcal Q$.

A set $\mathcal Q$ is said to be closed under finite (countable) intersections if $\bigcap_{i\in I}\mb A(i)\in \mathcal Q$ for every nonempty finite (resp., nonempty countable) family $\mb A\colon I\to \mathcal Q$. We say that $\mathcal Q$ is $\cap$-closed if $A\cap B\in\mathcal Q$ for every $A,B\in\mathcal Q$. It follows by induction that $\mathcal Q$ is $\cap$-closed if and only if $\mathcal Q$ is closed under finite intersections.

A nonempty set $\mathcal Q$ is called a ring if $A\cup B\in\mathcal Q$ and $A\setminus B\in\mathcal Q$ for every $A,B\in\mathcal Q$. If $\mathcal Q$ is a ring, then $\mathcal Q$ is $\cap$-closed and $\varnothing\in\mathcal Q$; hence, every ring is closed under finite unions and intersections.
A ring $\mathcal Q$ is called a $\delta$-ring ($\sigma$-ring) if it is closed under countable intersections (resp., unions).
Every $\sigma$-ring is a $\delta$-ring. For every set $\mathcal K$, there is a unique ring ($\sigma$-ring, $\delta$-ring) $\mathcal Q\supset \mathcal K$ such that $\mathcal Q\subset \mathcal Q'$ for every ring (resp., $\sigma$-ring, $\delta$-ring) $\mathcal Q'\supset\mathcal K$. It is called the ring (resp., $\sigma$-ring, $\delta$-ring) generated by $\mathcal K$ and will be denoted by $\kappa(\mathcal K)$ (resp., $\sigma(\mathcal K)$, $\delta(\mathcal K)$).

In what follows, the word `ring' without qualifiers will always be understood in the sense specified above. Rings in the algebraic sense (i.e., Abelian groups endowed with an associative and distributive multiplication with a unit element) will be referred to as algebraic rings.  

Given a set $\mathcal K$, we let $\sigma_0(\mathcal K)$ denote the set of all countable unions of its elements.

The proof of the next elementary statement can be found in Sec.~I.1.3 of~\cite{Dinculeanu}.

\begin{lemma}\label{ll1}
  Let $\mathcal Q$ be a $\delta$-ring. Then the following statements hold:
  \begin{itemize}
  \item [(i)] Let $B\in\mathcal Q$ and $\mb A\colon I\to \mathcal Q$ be a countable family such that $\mb A(i)\subset B$ for all $i\in I$. Then $\bigcup_{i\in I} \mb A(i) \in \mathcal Q$.
  \item [(ii)] $\sigma(\mathcal Q) = \sigma_0(\mathcal{Q})$.
  \item [(iii)] If $A\in \sigma(\mathcal Q)$ and $A\subset B$ for some $B\in \mathcal Q$, then $A\in \mathcal Q$.
  \end{itemize}
\end{lemma}

\begin{lemma}\label{l_sigma0}
  Let $\mathcal{K}$ be a set and $A\in \sigma(\mathcal{K})$. Then $A\subset B$ for some $B\in \sigma_0(\mathcal{K})$.
\end{lemma}
\begin{proof}
  Let $\mathcal{Q} = \set{A\in \sigma(\mathcal{K})}{A\subset B\mbox{ for some }B\in\sigma_0(\mathcal{K})}$. Clearly, $\mathcal{Q}$ is a $\sigma$-ring and $\mathcal{K}\subset\mathcal{Q}$. Hence, $\mathcal{Q} = \sigma(\mathcal{K})$.
\end{proof}

\begin{lemma}\label{l_AcapB}
  Let $\mathcal{K}$, $\mathcal{Q}$, and $\mathcal{R}$ be such that $A\cap B\in \mathcal{R}$ for every $A\in \mathcal{K}$ and $B\in \mathcal{Q}$.
  \begin{enumerate}
  \item [(i)] $A\cap B\in \sigma(\mathcal{R})$ for every $A\in \sigma(\mathcal{K})$ and $B\in \sigma(\mathcal{Q})$.
  \item [(ii)] Let $A\in \sigma(\mathcal{K})$, $B\in \sigma(\mathcal{Q})$, and $A\subset B$. Then $A\in \sigma(\mathcal{R})$.
  \item [(iii)] $\sigma(\mathcal{K})\cap \sigma(\mathcal{Q})\subset\sigma(\mathcal{R})$.
  \end{enumerate}
\end{lemma}
\begin{proof}
  (i) We first show that
  \begin{equation}
    \label{eq:cap1}
    A\cap B \in \sigma(\mathcal{R}),\quad A\in \sigma(\mathcal{K}),\,B\in \mathcal{Q}.
  \end{equation}
  Let $B\in\mathcal Q$ and $\mathcal K' = \set{A}{A\in\sigma(\mathcal K)\mbox{ and } A\cap B\in \sigma(\mathcal R)}$. Since $\mathcal K\subset \mathcal K'$ and $\mathcal K'$ is a $\sigma$-ring, we have $\mathcal K' = \sigma(\mathcal K)$. This proves~(\ref{eq:cap1}). Let $A\in \sigma(\mathcal{K})$ and $\mathcal{Q}' = \set{B}{B\in\sigma(\mathcal Q)\mbox{ and } A\cap B\in \sigma(\mathcal R)}$. By~(\ref{eq:cap1}), we have $\mathcal{Q}\subset\mathcal{Q}'$. As $\mathcal{Q}'$ is a $\sigma$-ring, it follows that $\mathcal{Q}' = \sigma(\mathcal{Q})$. This means that $A\cap B\in \sigma(\mathcal{R})$ for every $B\in \sigma(\mathcal{Q})$.
  \par\smallskip\noindent
  (ii) Since $A = A\cap B$, it follows from~(i) that $A\in \sigma(\mathcal{R})$.
  \par\smallskip\noindent
  (iii) Let $A\in \sigma(\mathcal{K})\cap \sigma(\mathcal{Q})$. Since $A = A\cap A$, we have $A\in \sigma(\mathcal{R})$ by~(i).
\end{proof}

\begin{lemma}\label{l3sigma}
  Let $\mathcal K$ and $\mathcal Q$ be $\delta$-rings such that $\sigma(\mathcal K) = \sigma(\mathcal Q)$. Then $\sigma(\mathcal K \cap \mathcal Q) = \sigma(\mathcal K)$ and $A\cap B\in \mathcal{K}\cap\mathcal{Q}$ for every $A\in \mathcal{K}$ and $B\in \mathcal{Q}$.
\end{lemma}
\begin{proof}
  Let $\mathcal{R} = \mathcal{K}\cap\mathcal{Q}$. Let $A \in \mathcal K$, $B \in\mathcal Q$, and $C = A\cap B$. Since $C\in \sigma(\mathcal K)$ and $C\subset A$, it follows from Lemma~\ref{ll1}(iii) that $C\in \mathcal{K}$. Similarly, we have $C\in \mathcal{Q}$ and, hence, $C\in \mathcal{R}$. Lemma~\ref{l_AcapB}(iii) now implies that $\sigma(\mathcal{K}) \subset \sigma(\mathcal{R})$. The opposite inclusion is obvious.
\end{proof}

Given a family $\fm{\mb A}{I}$, its disjoint union $\bigsqcup_{i\in I} \mb A(i)$ is defined by the equality\footnote{As usual, we let $\{x,y\}$ denote the unordered pair of elements $x$ and $y$ and put $\{x\} = \{x,x\}$. Thus, $\{x\}$ is the one-point set whose only element is $x$.}
\[
  \bigsqcup_{i\in I} \mb A(i) = \bigcup_{i\in I}\bigcup_{s\in \mb A(i)}\{(i,s)\} = \bigcup_{i\in I}\,\{i\}\times \mb A(i).
\] 

Let $f$ be a map such that $f(x)$ is a map for every $x\in D_f$. We put $f(x|y) = (f(x))(y)$ for every $x\in D_f$ and $y\in D_{f(x)}$.

\begin{lemma}\label{p_disjpart}
  Let $\fm{\mb A}{I}$ be a partition of a set $A$ and let $\fm{\mb J}{I}$ and $\fm{\mb B}{I}$ be such that $\fm{\mb B(i)}{\mb J(i)}$ is a partition of $\mb A(i)$ for every $i\in I$. Let $\fm{\mb C}{\bigsqcup_{i\in I}\mb J(i)}$ be such that $\mb C(i,j) = \mb B(i|j)$ for every $i\in I$  and $j\in \mb J(i)$. Then $\mb C$ is a partition of $A$.
\end{lemma}
\begin{proof}
  Let $K = \bigsqcup_{i\in I}\mb J(i)$. We have
  \begin{multline}\nonumber
    s\in A\Leftrightarrow s\in \mb A(i)\mbox{ for some }i\in I\Leftrightarrow s\in \mb B(i|j)\mbox{ for some }i\in I\mbox{ and }j\in\mb J(i)\\\Leftrightarrow s\in \mb C(k)\mbox{ for some }k\in K.    
  \end{multline}
  Hence, $A = \bigcup_{k\in K}\mb C(k)$ and we only have to show that $\mb C$ is disjoint. Let $k,k'\in K$ be such that $k\neq k'$. There are $i,i'\in I$, $j\in \mb J(i)$, and $j'\in\mb J(i')$ such that $k = (i,j)$ and $k' = (i',j')$. By the definition of $\mb C$, we have
  \begin{equation}\label{Bkk'}
    \mb C(k) = \mb B(i|j),\quad \mb C(k') = \mb B(i'|j').
  \end{equation} 
  If $i\neq i'$, then by~(\ref{Bkk'}) we have $\mb C(k)\subset \mb A(i)$ and $\mb C(k')\subset \mb A(i')$. Since $\mb A$ is disjoint, we conclude that $\mb C(k)\cap\mb C(k') = \varnothing$. If $i = i'$, then $j \neq j'$ because otherwise we would have $k = k'$. As $\mb B(i)$ is disjoint, it follows from~(\ref{Bkk'}) that $\mb C(k)\cap\mb C(k') = \varnothing$. This means that $\mb C$ is disjoint.
\end{proof}

For every $x$ and $y$, we let $[x;y]$ denote the map on $\{1,2\}$ taking $1$ to $x$ and $2$ to $y$. The disjoint union $A\sqcup B$ of sets $A$ and $B$ is defined by the equality
\[
  A\sqcup B = \bigsqcup_{i\in\{1,2\}}[A;B](i) = (\{1\}\times A)\cup (\{2\}\times B). 
\] 

\begin{lemma}\label{p_disjpart2}
  Let $A$ and $B$ be sets such that $A\cap B = \varnothing$. Let $\fm{\mb A}{I}$ and $\fm{\mb B}{J}$ be partitions of $A$ and $B$ respectively. Let $\fm{\mb C}{I\sqcup J}$ be such that $\mb C(1,i) = \mb A(i)$ for every $i\in I$ and $\mb C(2,j) = \mb B(j)$ for every $j\in J$. Then $\mb C$ is a partition of $A\cup B$.
\end{lemma}
\begin{proof}
  Let  $I' = \{1,2\}$, $A' = A\cup B$, $\mb A' = [A;B]$, $\mb J = [I;J]$, and  $\mb B' = [\mb A;\mb B]$. Then $\mb A'$ is a partition of $A'$, $I\sqcup J = \bigsqcup_{i\in I'} \mb J(i)$,  and $\mb C(i,j) = \mb B'(i|j)$ for every $i\in I'$ and $j\in\mb J(i)$. Hence, the statement follows from Lemma~\ref{p_disjpart} applied to $I'$, $A'$, $\mb A'$, and $\mb B'$ in place of $I$, $A$, $\mb A$, and $\mb B$ respectively.
\end{proof}

Throughout the paper, the words `monoid' and `group' always mean `Abelian monoid' and `Abelian group' respectively. Accordingly, we use additive notation and denote the binary operation and the neutral element of a monoid by $+$ and $0$ respectively. 

A monoid equipped with a topology that makes addition continuous is called a topological monoid. A group $\mf A$ equipped with a topology that makes addition and inversion continuous is called a topological group. For every topological group, there is a unique uniformity that induces its topology and makes addition and inversion uniformly continuous. We assume that every topological group is endowed with this uniformity. Thus, we can speak of (sequential) completeness of a topological group.

In the sequel, we use the abbreviations HTM and HTG for Hausdorff topological monoid and Hausdorff topological group respectively. 

Recall that a topological space is called regular if each of its points has a base of neighbourhoods consisting of closed sets. Every topological group is regular.
There are HTMs that are not regular (see Example~\ref{e_nonreg} and Remark~\ref{r_paratop}).

\section{\texorpdfstring{Additive and $\sigma$-additive functions}{Additive and σ-additive functions}}
\label{s_add}

Let $\mf A$ be a monoid. If $f$ is a finite $\mf A$-valued family, then we let $\finsum f$ denote the sum of $f(i)$ over all $i\in D_f$. Let $I$ be a finite set and $f$ be such that $f(i)\in \mf A$ for every $i\in I$. Then we put $\finsum_{i\in I} f(i) = \finsum\{f(i)\}_{i\in I}$. 

\begin{definition}\label{dd4}
  Let $\mf A$ be a monoid. A map $\mu$ is said to be an $\mf A$-valued additive function if $\im\mu\subset \mf A$ and 
  $\mu(A) = \finsum \mu\circ\mb A$ for every $A\in D_\mu$ and every finite partition $\mb A$ of $A$ into elements of $D_\mu$. A map $\mu$ is called an $\mf A$-valued content if $\mu$ is an $\mf A$-valued additive function and $D_\mu$ is a ring.
\end{definition}

Given a set $I$, we let $\mathcal{F}_I$ denote the set of all finite subsets of $I$ ordered by inclusion. For every map $f$, we put $\mathcal{F}[f] = \mathcal{F}_{D_f}$.

Let $\mf A$ be a monoid and $f$ be such that $f(i)\in\mf A$ for every $i\in I$. Let $\Phi\colon \mathcal{F}_I\to\mf A$ be such that $\Phi(J) = \finsum_{i\in J}f(i)$ for every $J\in \mathcal{F}_I$. Since $\mathcal{F}_I$ is a directed set, $\Phi$ is a net in $\mf A$ over $\mathcal{F}_I$. This net will be denoted by $\bigS_{i\in I}f(i)$. If $f$ is an $\mf A$-valued map, then we define the map $\bigS f\colon \mathcal{F}[f]\to\mf A$ by setting $\bigS f = \bigS_{i\in D_f}f(i)$.

Let $\mf A$ be an HTM and $f$ be such that $f(i)\in\mf A$ for every $i\in I$. If the net $\bigS_{i\in I}f(i)$ converges in $\mf A$ over $\mathcal{F}_I$, then we say that $f(i)$ is summable in $\mf A$ over $i\in I$ and write $\sm{\mf A}{i\in I}{f(i)}$. In this case, its limit is called the unconditional sum of $f(i)$ in $\mf A$ over $i\in I$ and is denoted by $\sum_{i\in I}f(i)$. If both $\finsum_{i\in I} f(i)$ and $\sum_{i\in I} f(i)$ are well defined (i.e., if $\mf A$ is an HTM and $I$ is finite), then $\finsum_{i\in I} f(i) = \sum_{i\in I} f(i)$.

Let $\mf A$ be an HTM and $f$ be an $\mf A$-valued map. If $\sm{\mf A}{i\in D_f}{f(i)}$, then we say that $f$ is summable in $\mf A$ and write $\smm{\mf A}{f}$. In this case, we put $\sum f = \sum_{i\in D_f}f(i)$. 

If the monoid $\mf A$ needs to be indicated explicitly, we write it as a superscript and use the notations $\finsum^{\mf A}f$, $\finsum_{i\in I}^{\mf A}f(i)$, $\bigS^{\mf A}f$, $\bigS_{i\in I}^{\mf A}f(i)$, $\sum^{\mf A}f$, and $\sum_{i\in I}^{\mf A}f(i)$.

\begin{definition}\label{dd5}
  Let $\mf A$ be an HTM. A map $\mu$ is said to be an $\mf A$-valued $\sigma$-additive function if $\im\mu\subset \mf A$ and $\mu(A) = \sum \mu\circ\mb A$ for every $A\in D_\mu$ and every countable partition $\mb A$ of $A$ into elements of $D_\mu$. A map $\mu$ is called an $\mf A$-valued $\sigma$-content (premeasure)  if $\mu$ is an $\mf A$-valued $\sigma$-additive function and $D_\mu$ is a ring (resp., a $\delta$-ring). 
\end{definition}

Every $\mf A$-valued $\sigma$-additive function is an $\mf A$-valued additive function.

Let $\mf A$ be a monoid. A monoid $\mf B$ is called a submonoid of $\mf A$ if $\mf B\subset \mf A$, $0_{\mf A}\in \mf B$, and the addition of $\mf B$ is the restriction of that of $\mf A$. If $\mf B$ is a submonoid of $\mf A$ and $f$ is a $\mf B$-valued map such that $D_f$ is finite, then $\finsum^{\mf B} f = \finsum^{\mf A} f$. Hence, we have the next statement.
\begin{proposition}\label{p_subsgradditive}
  Let $\mf B$ be a submonoid of $\mf A$. Then $\mu$ is a $\mf B$-valued additive function if and only if $\mu$ is an $\mf A$-valued additive function and $\im\mu\subset\mf B$.
\end{proposition}

Let $\mf A$ and $\mf B$ be topological monoids. We say that $\mf B$ is a topological submonoid of $\mf A$ if $\mf B$ is a submonoid of $\mf A$ and is a topological subspace of $\mf A$. The next statement follows immediately from Proposition~\ref{p_sumsubsgr}.

\begin{proposition}\label{p_subpremeasure}
Let $\mf A$ be an HTM and $\mf B$ be its topological submonoid. Then $\mu$ is a $\mf B$-valued $\sigma$-additive function if and only if $\mu$ is an $\mf A$-valued $\sigma$-additive function and $\im\mu\subset\mf B$.  
\end{proposition}

Recall that a subset $X$ of a topological space $\mf A$ is called sequentially closed in $\mf A$ if $x\in X$ whenever there is a sequence of elements of $X$ that converges to $x$ in $\mf A$. 

\begin{proposition}\label{p_Im0}
  Let $\mf A$ be an HTM, $X$ be a sequentially closed subset of $\mf A$, and $\mu$ be an $\mf A$-valued $\sigma$-additive function. Let $\mathcal{K}$ be a ring such that $\mathcal{K}\subset D_\mu\subset\sigma_0(\mathcal{K})$ and $\im\mu|_{\mathcal{K}}\subset X$. Then $\im \mu\subset X$. 
\end{proposition}
\begin{proof}
  Let $A\in D_\mu$. Since $\mathcal{K}$ is a ring and $A\in \sigma_0(\mathcal{K})$, there is a countable partition $\mb A\colon I\to \mathcal{K}$ of $A$. Let $K\in \mathcal{F}_I$. Then the set $B = \bigcup_{i\in K}\mb A(i)$ belongs to $\mathcal{K}$ because $\mathcal{K}$ is a ring. By the additivity of $\mu$, this implies that $\finsum_{i\in K}\mu(\mb A(i))$ is equal to $\mu(B)$ and, hence, belongs to $X$. Since $X$ is sequentially closed in $\mf A$, $I$ is countable, and
  \[
    \mu(A) = \sum_{i\in I} \mu(\mb A(i)) = \lim_{K\in \mathcal{F}_I}\finsum_{i\in K}\mu(\mb A(i))
  \]
   by the $\sigma$-additivity of $\mu$, we conclude that $\mu(A)\in X$. 
\end{proof}

\begin{proposition}\label{p_un0}
  Let $\mf A$ be an HTM and $\mu_1$ and $\mu_2$ be $\mf A$-valued $\sigma$-additive functions. Let $\mathcal{K}$ be a ring such that $\mu_1$ and $\mu_2$ coincide on $\mathcal{K}$. Then $\mu_1$ and $\mu_2$ coincide on $D_{\mu_1}\cap D_{\mu_2}\cap\sigma_0(\mathcal{K})$. 
\end{proposition}
\begin{proof}
  Let $A\in D_{\mu_1}\cap D_{\mu_2}\cap\sigma_0(\mathcal{K})$. Since $\mathcal{K}$ is a ring and $A\in \sigma_0(\mathcal{K})$, there is a countable partition $\mb A$ of $A$ into elements of $\mathcal{K}$. As $\mu_1$ and $\mu_2$ coincide on $\mathcal{K}$, we have $\mu_1\circ\mb A = \mu_2\circ\mb A$. By the $\sigma$-additivity of $\mu_1$ and $\mu_2$, we obtain $\mu_1(A) = \sum\mu_1\circ \mb A = \sum\mu_2\circ\mb A = \mu_2(A)$.
\end{proof}

If $\mathcal{K}$ is a $\delta$-ring, then $\sigma_0(\mathcal{K}) = \sigma(\mathcal{K})$ by Lemma~\ref{ll1}(ii). Hence, Propositions~\ref{p_Im0} and~\ref{p_un0} can be modified as follows.

\begin{proposition}\label{p_Im0delta}
  Let $\mf A$, $X$, and $\mu$ be as in Proposition~\textup{\ref{p_Im0}}. Let $\mathcal{K}$ be a $\delta$-ring such that $\mathcal{K}\subset D_\mu\subset\sigma(\mathcal{K})$ and $\im\mu|_{\mathcal{K}}\subset X$. Then $\im \mu\subset X$. 
\end{proposition}

\begin{proposition}\label{p_un0delta}
  Let $\mf A$, $\mu_1$, and $\mu_2$ be as in Proposition~\textup{\ref{p_un0}}. Let $\mathcal{K}$ be a $\delta$-ring such that $\mu_1$ and $\mu_2$ coincide on $\mathcal{K}$. Then $\mu_1$ and $\mu_2$ coincide on $D_{\mu_1}\cap D_{\mu_2}\cap\sigma(\mathcal{K})$. 
\end{proposition}

\section{Semi-rings}
\label{s_semi}

\begin{definition}\label{dd3}
  A set $\mathcal Q$ is called a semi-ring if
  \begin{enumerate}
  \item $A\cap B\in\mathcal Q$ for every $A,B\in\mathcal Q$,
  \item for every $A,B\in\mathcal Q$, there is a finite disjoint family $\mb C\colon I\to \mathcal Q$ such that $A\setminus B=\bigcup_{i\in I} \mb C(i)$.
  \end{enumerate}
\end{definition}

A proof of the next statement can be found, e.g., in~\cite[Sec.~4.3]{KolmogorovFomin1975}. 

\begin{proposition}\label{p_semiring}
  Let $\mathcal Q$ be a semi-ring. Then
  \begin{equation}\label{kappa}
    \kappa(\mathcal Q) = \set{A}{\mbox{there is a finite partition of $A$ into elements of $\mathcal Q$}}.
  \end{equation}
\end{proposition}

\begin{remark}
A notion equivalent to that of Definition~\ref{dd3} already appears in
Kolmogorov's 1930 paper~\cite{Kolmogoroff1930}, where it is called a
$Z$-system (decomposable system). Proposition~\ref{p_semiring} is also
proved there. A more restrictive variant, called a half-ring, was subsequently used in von~Neumann's
lectures~\cite[Sec.~10.1]{NeumannLectures} and was adopted by Halmos~\cite{Halmos} under the name of a semi-ring. The standard modern definition of a semi-ring essentially coincides with Kolmogorov's original definition.
\end{remark}

\begin{remark}
  Definition~\ref{dd3} is slightly different from the standard one (see, e.g., Definition~2 in~\cite[Sec.~4.2]{KolmogorovFomin1975}) because it does not require $\varnothing\in \mathcal{Q}$ and allows $\mb C$ to be the empty family. In particular, $\varnothing$ and all one-element sets are semi-rings according to Definition~\ref{dd3}. We note that $\varnothing$ always belongs to the set in the right-hand side of~(\ref{kappa}) because the empty family is its partition into elements of $\mathcal{Q}$.
  This modification of the standard definition does not essentially affect the proof of Proposition~\ref{p_semiring}.
\end{remark}

\begin{proposition}\label{l_semi}
  Let $\mf A$ be a monoid and $\mu$ be an $\mf A$-valued additive function such that $D_\mu$ is a semi-ring. There is a unique $\mf A$-valued additive function $\nu$ such that $D_\nu= \kappa(D_\mu)$ and $\mu = \nu|_{D_\mu}$. If $\mf A$ is an HTM and $\mu$ is $\sigma$-additive, then so is $\nu$.
\end{proposition}

For scalar-valued $\mu$, a proof of this statement can be found in~\cite[Sec.~10.1]{NeumannLectures} and~\cite[Secs.~26.1 and~26.2]{KolmogorovFomin1975}. 
While the treatment of the general case is essentially the same, the proof involves an interchange of the order of summation, which is, in general, a subtle operation for HTMs (see Propositions~\ref{p_fubiniHTM} and~\ref{p_fubiniRTS} and Example~\ref{e_nonreg}). For this reason, we give a proof of this result in Appendix~\ref{app_C}.

\begin{proposition}\label{ll6a}
  Let $\fm{\Theta}{\Lambda}$ be a family of sets such that $A\cap B\in \Theta(\lambda)\cap\Theta(\lambda')$ for every $\lambda,\lambda'\in \Lambda$, $A\in\Theta(\lambda)$, and $B\in\Theta(\lambda')$. Let $\mathcal Q = \bigcup_{\lambda\in \Lambda} \Theta(\lambda)$.
  \begin{enumerate}
  \item [(i)] Let $\Theta(\lambda)$ be a semi-ring for every $\lambda\in \Lambda$. Then $\mathcal{Q}$ is a semi-ring.
  \item [(ii)] Let $\Theta(\lambda)$ be a $\delta$-ring for every $\lambda\in \Lambda$. Then $\kappa(\mathcal Q)$ is a $\delta$-ring.
  \item [(iii)] Let $\mf A$ be an HTM. A map $\fm{\eta}{\mathcal{Q}}$ is an $\mf A$-valued $\sigma$-additive function if and only if $\eta|_{\Theta(\lambda)}$ is an $\mf A$-valued $\sigma$-additive function for every $\lambda\in \Lambda$. 
  \end{enumerate}
\end{proposition}
\begin{proof}
  (i) Let $A, B\in \mathcal Q$. Then $A\in \Theta(\lambda)$ and $B\in\Theta(\lambda')$ for some $\lambda,\lambda'\in \Lambda$ and, hence, $A\cap B\in \Theta(\lambda)$. As $A\setminus B=A\setminus(A\cap B)$ and $\Theta(\lambda)$ is a semi-ring, there is a finite disjoint family $\mb C\colon I\to \Theta(\lambda)$ such that $A\setminus B = \bigcup_{i\in I}\mb C(i)$. Since $\Theta(\lambda)\subset \mathcal{Q}$, we have $A\cap B\in \mathcal{Q}$ and $\mb C\colon I\to \mathcal{Q}$. Thus, $\mathcal Q$ is a semi-ring.
  \par\medskip\noindent
  (ii) We have to show that $\kappa(\mathcal Q)$ is closed under countable intersections. Let $\mb A\colon I\to \kappa(\mathcal Q)$ be a nonempty countable family and $A=\bigcap_{i\in I}\mb A(i)$. Let  $i_0\in I$. By~(i) and Proposition~\ref{p_semiring}, there is a finite family $\mb B\colon J\to \mathcal Q$ such that $\mb A(i_0) = \bigcup_{j\in J} \mb B(j)$. We therefore obtain $A = \bigcup_{j\in J} \mb C(j)$, where $\fm{\mb C}{J}$ is such that $\mb C(j) = \bigcap_{i\in I} (\mb A(i)\cap \mb B(j))$ for every $j\in J$.
  Let $j\in J$ and $\lambda\in\Lambda$ be such that $\mb B(j)\in\Theta(\lambda)$. Let $i\in I$. By~(i) and Proposition~\ref{p_semiring}, there is a finite family $\mb B'\colon K\to \mathcal Q$ such that $\mb A(i) = \bigcup_{k\in K} \mb B'(k)$. Since $\mb B'(k)\cap\mb B(j)\in \Theta(\lambda)$ for every $k\in K$ and $\Theta(\lambda)$ is a ring, we conclude that  $\mb A(i)\cap\mb B(j)\in\Theta(\lambda)$. As $\Theta(\lambda)$ is closed under countable intersections, this implies that $\mb C(j)\in \Theta(\lambda)$. Thus, $\mb C(j)\in \mathcal{Q}$ for all $j\in J$ and, hence, $A\in \kappa(\mathcal Q)$.
  \par\medskip\noindent
  (iii) We prove the `if' part, the `only if' part being obvious. Clearly, $\im\eta\subset\mf A$. Let $A\in \mathcal{Q}$ and $\mb A\colon I\to \mathcal{Q}$ be a countable partition of $A$. Let $\lambda\in \Lambda$ be such that $A\in \Theta(\lambda)$ and $\eta' = \eta|_{\Theta(\lambda)}$. Since $\mb A(i) = A\cap \mb A(i)$, we have $\mb A(i)\in \Theta(\lambda)$ for every $i\in I$. Hence, $\eta\circ \mb A = \eta'\circ\mb A$ and the $\sigma$-additivity of $\eta'$ implies that $\eta(A) = \eta'(A) = \sum \eta'\circ\mb A = \sum \eta\circ\mb A$. This means that $\eta$ is $\sigma$-additive.
\end{proof}

\section{Measures}
\label{s_meas}

For a map $\mu$, the $\sigma$-ring $\sigma(D_\mu)$ is called its host $\sigma$-ring and is denoted by $\q{\mu}$.

\begin{definition}\label{d_sim}
  We say that maps $\mu$ and $\nu$ are similar and write $\mu\asymp\nu$ if $\q\mu=\q\nu$ and $\mu$ and $\nu$ coincide on $D_\mu\cap D_\nu$. 
\end{definition}

Clearly, similarity is a reflexive and symmetric relation on the class of all maps. The next statement shows that its restriction to the class of all $\mf A$-valued premeasures is an equivalence relation for every HTM $\mf A$.

\begin{proposition}\label{l_trans}
  Let $\mf A$ be an HTM and $\mu$, $\nu$, and $\eta$ be $\mf A$-valued premeasures. If $\mu\asymp\nu$ and $\nu\asymp\eta$, then $\mu\asymp\eta$. 
\end{proposition}
\begin{proof}
  Let $\mathcal{K} = D_\mu\cap D_\nu\cap D_\eta$. Relations $\mu\asymp\nu$ and $\nu\asymp\eta$ ensure that $\q\mu=\q\nu=\q\eta$ and $\mu$ and $\eta$ coincide on $\mathcal{K}$. Applying Lemma~\ref{l3sigma} first to $D_\mu$ and $D_\eta$ and then to $D_\mu\cap D_\eta$ and $D_\nu$, we obtain $\q\mu = \sigma(D_\mu\cap D_\eta) = \sigma(\mathcal{K})$ and, hence, $\mathcal{K}\subset D_\mu\cap D_\eta\subset \sigma(\mathcal{K})$. By Proposition~\ref{p_un0delta}, $\mu$ and $\eta$ coincide on $D_\mu\cap D_\eta$. Thus, $\mu\asymp\eta$.  
\end{proof}

\begin{definition}\label{dd6}
  Let $\mf A$ be an HTM. An $\mf A$-valued premeasure $\mu$ is called an $\mf A$-valued measure if $\mu$ is an extension of every $\mf A$-valued premeasure $\nu$ such that $\nu\asymp\mu$.
\end{definition}

The next simple uniqueness result is sufficient for our treatment of spaces of measures in Sec.~\ref{s_spaces}. In the group-valued case, it can be significantly generalized (see Theorem~\ref{tt1} below).

\begin{theorem}\label{t_unique}
  Let $\mathcal K$ be a $\delta$-ring, $\mf A$ be an HTM, and $\mu_1$ and $\mu_2$ be $\mf A$-valued premeasures such that $\mathcal K\subset D_{\mu_1}\subset \sigma(\mathcal K)$, $\mathcal K\subset D_{\mu_2}\subset\sigma(\mathcal K)$, and $\mu_1|_{\mathcal K}=\mu_2|_{\mathcal K}$. Then $\mu_1 \asymp \mu_2$. If $\mu_2$ is a measure, then it is an extension of $\mu_1$. If $\mu_1$ and $\mu_2$ are both measures, then $\mu_1=\mu_2$.
\end{theorem}
\begin{proof}
  Let $\mu = \mu_1|_{\mathcal{K}}$. Since $\sigma(\mathcal{K}) = \q{\mu_1} = \q{\mu_2}$, we have $\mu_1\asymp\mu$ and $\mu\asymp\mu_2$. By Proposition~\ref{l_trans}, it follows that $\mu_1\asymp\mu_2$. If $\mu_2$ is a measure, then it is an extension of $\mu_1$ by Definition~\ref{dd6}. If $\mu_1$ and $\mu_2$ are both measures, then they are extensions of each other and, hence, $\mu_1=\mu_2$.
\end{proof}

We now show that every similarity class of premeasures contains precisely one measure. 

\begin{theorem}\label{tt2}
  Let $\mf A$ be an HTM and $\mu$ be an $\mf A$-valued premeasure. There exists a unique $\mf A$-valued measure $\nu$ such that $\nu\asymp\mu$. 
\end{theorem}
\begin{proof}
  Let $\Lambda$ be the set of all $\mf A$-valued premeasures that are similar to $\mu$ and let $\mathcal Q = \bigcup_{\lambda\in\Lambda} D_\lambda$. If $\lambda,\lambda'\in\Lambda$, then $\lambda\asymp \lambda'$ by Proposition~\ref{l_trans}. In particular, $\lambda$ and $\lambda'$ coincide on $D_\lambda\cap D_{\lambda'}$ for every $\lambda,\lambda'\in\Lambda$ and, therefore, there is a map $\eta$ such that $D_\eta = \mathcal Q$ and $\lambda = \eta|_{D_\lambda}$ for every $\lambda\in\Lambda$. Let $\lambda,\lambda'\in\Lambda$. Since $\q\lambda = \q{\lambda'}$, Lemma~\ref{l3sigma} implies that $A\cap B\in D_\lambda\cap D_{\lambda'}$ for every $A\in D_\lambda$ and $B\in D_{\lambda'}$. Applying Proposition~\ref{ll6a} to $\Theta = \{D_\lambda\}_{\lambda\in\Lambda}$, we conclude that $\mathcal Q$ is a semi-ring, $\kappa(\mathcal Q)$ is a $\delta$-ring, and $\eta$ is an $\mf A$-valued $\sigma$-additive function. By Proposition~\ref{l_semi}, there exists an $\mf A$-valued premeasure $\nu$ such that $D_\nu = \kappa(\mathcal Q)$ and $\nu|_{\mathcal Q}=\eta$. As $\mu\in\Lambda$ and $\q\lambda = \q\mu$ for all $\lambda\in\Lambda$, we have $D_\mu\subset \mathcal Q\subset \q\mu$ and $\eta|_{D_\mu}=\mu$.  It follows that $D_\mu\subset D_\nu\subset \q\mu$ and $\nu|_{D_\mu} =\mu$, whence $\nu\asymp\mu$. Let $\lambda$ be an $\mf A$-valued premeasure such that $\lambda\asymp \nu$. Then $\lambda\asymp\mu$ by Proposition~\ref{l_trans} and, therefore, $\lambda\in \Lambda$. This implies that $D_\lambda\subset\mathcal{Q}\subset D_\nu$ and, hence, $\nu$ is an extension of $\lambda$. Thus, $\nu$ is a measure. If $\lambda$ is an $\mf A$-valued measure such that $\lambda\asymp \mu$, then $\lambda\asymp\nu$ by Proposition~\ref{l_trans} and, hence, $\lambda = \nu$. This proves the uniqueness of $\nu$.          
\end{proof}

\begin{definition}\label{d_ass}
  Let $\mf A$ be an HTM and $\mu$ be an $\mf A$-valued premeasure. The $\mf A$-valued measure $\nu$ such that $\nu\asymp\mu$ is called the $\mf A$-valued measure associated with $\mu$ and is denoted by $\as[\mf A]\mu$. The domain of $\as[\mf A]\mu$ will be denoted by $\Sigma_{\mf A}(\mu)$.
\end{definition}

Theorem~\ref{tt2} ensures that Definition~\ref{d_ass} is correct.

Let $\mf A$ be an HTM and $\mu$ be an $\mf A$-valued premeasure. It follows from Definitions~\ref{dd6} and~\ref{d_ass} that $\as[\mf A]\mu$ is an extension of $\mu$. By Definitions~\ref{d_sim}  and~\ref{d_ass}, we have
\begin{equation}\label{Qmu}
  \q{\as[\mf A]{\mu}} = \q{\mu}.
\end{equation}

\begin{proposition}\label{p_ass}
  Let $\mf A$ be an HTM, $\mu$ be an $\mf A$-valued premeasure, and $\nu$ be an $\mf A$-valued measure such that $\nu\asymp\mu$. Then $\nu  = \as[\mf A]\mu$.
\end{proposition}
\begin{proof}
  The statement follows immediately from Theorem~\ref{tt2} and Definition~\ref{d_ass}.
\end{proof}

\begin{proposition}\label{p_measequiv}
  Let $\mf A$ be an HTM and $\mu$ be an $\mf A$-valued premeasure. Then
  \[
    \mu \mbox{ is an $\mf A$-valued measure} \Leftrightarrow \mu = \as[\mf A]\mu \Leftrightarrow \Sigma_{\mf A}(\mu)\subset D_\mu.
  \]
\end{proposition}
\begin{proof}
  Let $\nu = \as[\mf A]\mu$. Since $\nu$ is an extension of $\mu$, we have $\Sigma_{\mf A}(\mu)\subset D_\mu$ if and only if $\mu = \nu$. As $\nu$ is a measure, the equality $\mu = \nu$ implies that $\mu$ is a measure. Let $\mu$ be a measure. Since $\mu\asymp\mu$, it follows from Proposition~\ref{p_ass} that $\mu = \nu$. 
\end{proof}

\begin{proposition}\label{p_asssimilar}
  Let $\mf A$ be an HTM and let $\mu$ and $\nu$ be $\mf A$-valued premeasures. Then $\mu\asymp\nu$ if and only if $\as[\mf A]\mu = \as[\mf A]\nu$.
\end{proposition}
\begin{proof}
  Let $\mu' = \as[\mf A]\mu$ and $\nu' =\as[\mf A]\nu$. Definition~\ref{d_ass} implies that $\mu'$ is a measure, $\mu'\asymp\mu$, and $\nu'\asymp\nu$. If $\mu' = \nu'$, then $\mu\asymp\nu$ by Proposition~\ref{l_trans}. Let $\mu\asymp\nu$. Then $\mu'\asymp\nu$ by Proposition~\ref{l_trans} and it follows from Proposition~\ref{p_ass} that $\mu' = \nu'$.
\end{proof}

\begin{proposition}\label{p_Dmusigma}
	Let $\mf A$ be an HTM and $\mu$ be an $\mf A$-valued premeasure such that $D_\mu$ is a $\sigma$-ring. Then $\mu$ is an $\mf A$-valued measure.
\end{proposition}
\begin{proof}
	Let $\nu$ be an $\mf A$-valued premeasure such that $\nu\asymp\mu$. Since $D_\mu = \q\mu$, Definition~\ref{d_sim} implies that $\mu$ is an extension of $\nu$. By Definition~\ref{dd6}, we conclude that $\mu$ is a measure. 
\end{proof}

Let $\mf A$ and $\mf B$ be monoids. A map $f\colon \mf A\to\mf B$ is called a homomorphism from $\mf A$ to $\mf B$ if $f(0_{\mf A}) = 0_{\mf B}$ and $f(x+_{\mf A} y) = f(x)+_{\mf B} f(y)$ for every $x,y\in \mf A$. A map $f$ is called an isomorphism from $\mf A$ to $\mf B$ if it is both a homomorphism from $\mf A$ to $\mf B$ and a bijection between $\mf A$ and $\mf B$. In this case, $f^{-1}$ is an isomorphism from $\mf B$ to $\mf A$.

Let $\mf A$ and $\mf B$ be topological monoids. A map $f$ is called a topological isomorphism from $\mf A$ to $\mf B$ if it is both an isomorphism and a homeomorphism from $\mf A$ to $\mf B$. In this case, $f^{-1}$ is a topological isomorphism from $\mf B$ to $\mf A$.

Let $\mf A$ and $\mf B$ be HTMs, $f\colon \mf A\to \mf B$ be a continuous homomorphism, and $\mu$ be an $\mf A$-valued premeasure. It follows from Proposition~\ref{p_sum}(ii) that $f\circ\mu$ is a $\mf B$-valued premeasure.

\begin{proposition}\label{l_iso}
  Let $\mf A$ and $\mf B$ be HTMs, $f\colon \mf A\to \mf B$ be a topological isomorphism, and $\mu$ be an $\mf A$-valued measure. Then $f\circ\mu$ is a $\mf B$-valued measure.
\end{proposition}
\begin{proof}
  Let $\nu = f\circ\mu$ and $\eta$ be a $\mf B$-valued premeasure such that $\nu\asymp\eta$. Then $f^{-1}\circ\nu\asymp f^{-1}\circ\eta$. Since $f^{-1}\circ\nu = \mu$ and $\mu$ is a measure, it follows that $\mu$ is an extension of $f^{-1}\circ\eta$ and, hence, $\nu$ is an extension of $\eta$. This means that $\nu$ is a $\mf B$-valued measure.
\end{proof}

\begin{corollary}\label{c_iso}
  Let $\mf A$ and $\mf B$ be HTMs, $f\colon \mf A\to \mf B$ be a topological isomorphism, and $\mu$ be an $\mf A$-valued premeasure. Then $f\circ \as[\mf A]{\mu} = \as[\mf B]{f\circ\mu}$.
\end{corollary}
\begin{proof}
  Let $\nu = \as[\mf A]{\mu}$. Since $\mu\asymp\nu$, we have $f\circ\mu\asymp f\circ\nu$. As $f\circ\nu$ is a measure by Proposition~\ref{l_iso}, the statement follows from Proposition~\ref{p_ass}.  
\end{proof}

\begin{proposition}\label{l_subgr}
  Let $\mf A$ be an HTM, $\mf B$ be its topological submonoid, $\mu$ be a $\mf B$-valued premeasure, and $\nu = \as[\mf B]\mu$. Then $\as[\mf A]\mu = \as[\mf A]\nu$ and, in particular, $\as[\mf A]\mu$ is an extension of $\nu$. If $\mf B$ is sequentially closed in $\mf A$, then $\nu=\as[\mf A]\mu$. 
\end{proposition} 
\begin{proof}
  By Proposition~\ref{p_subpremeasure}, $\mu$ and $\nu$ are $\mf A$-valued premeasures. Let $\nu' = \as[\mf A]\mu$. Since $\nu\asymp \mu$ and $\nu'\asymp\mu$, we have 
  \begin{equation}\label{sim}
    \nu\asymp \nu'
  \end{equation}
  by Proposition~\ref{l_trans}. As $\nu'$ is an $\mf A$-valued measure, Proposition~\ref{p_ass} implies that $\nu' = \as[\mf A]\nu$. If $\mf B$ is sequentially closed in $\mf A$, then $\im\nu'\subset\mf B$ by Proposition~\ref{p_Im0delta} and Proposition~\ref{p_subpremeasure} ensures that $\nu'$ is a $\mf B$-valued premeasure. Since $\nu$ is a $\mf B$-valued measure, it follows from~(\ref{sim}) that $\nu$ is an extension of $\nu'$. This means that $\nu=\nu'$. 
\end{proof}

\begin{corollary}\label{l_subgr1}
  Let $\mf A$ be an HTM and $\mf B$ be its sequentially closed topological submonoid. Then $\mu$ is a $\mf B$-valued measure if and only if $\mu$ is an $\mf A$-valued measure and $\im\mu\subset\mf B$.
\end{corollary}
\begin{proof}
  The statement follows immediately from Propositions~\ref{p_subpremeasure}, \ref{p_measequiv} and~\ref{l_subgr}.
\end{proof}

Corollary~\ref{l_subgr1} is no longer valid if $\mf B$ is not assumed to be sequentially closed in $\mf A$ (see Example~\ref{e_sub} below).

\section{Domain of associated measure}
\label{s_ass}

Let $\mf A$ be an HTM. Given an $\mf A$-valued premeasure $\mu$, we define
\begin{multline}\label{SigmaA}
  \mathcal{L}_{\mf A}(\mu) = \sett{A}{A\in \q\mu \mbox{ and the family $\mu\circ \mb A$ is summable in $\mf A$ for every}}{\mbox{ countable disjoint family $\mb A\colon I\to D_\mu$ such that $\mb A(i)\subset A$ for all $i\in I$ }}.
\end{multline}

\begin{proposition}\label{p_Sigdelta}
  Let $\mf A$ be an HTM and $\mu$ be an $\mf A$-valued premeasure. Then $\mathcal{L}_{\mf A}(\mu)$ is a $\delta$-ring. 
\end{proposition}
\begin{proof}
  Let $\mathcal{L} = \mathcal{L}_{\mf A}(\mu)$. It is clear from~(\ref{SigmaA}) that $\mathcal{L}\neq\varnothing$ and that, together with each of its elements, $\mathcal{L}$ contains all its subsets belonging to $\q\mu$. In particular, if $A,B\in \mathcal{L}$, then $A\setminus B\in \mathcal{L}$, and if $\mb A\colon I\to\mathcal{L}$ is a nonempty countable family, then the set $\bigcap_{i\in I} \mb A(i)$ belongs to $\mathcal{L}$. Let $A,B\in\mathcal{L}$ and $\mb A\colon I\to D_\mu$ be a countable disjoint family such that $\mb A(i)\subset A\cup B$ for every $i\in I$. Let $\fm{\mb C}{I}$ and $\fm{\mb C'}{I}$ be such that $\mb C(i) = \mb A(i)\cap A$ and $\mb C'(i)=\mb A(i)\setminus A$ for every $i\in I$. In view of Lemma~\ref{ll1}(iii), $\mb C$ and $\mb C'$ are disjoint families of elements of $D_\mu$ such that $\mb C(i)\subset A$ and $\mb C'(i)\subset B$ for every $i\in I$. Since $A,B\in \mathcal{L}$, the families $\mu\circ\mb C$ and $\mu\circ\mb C'$ are summable in $\mf A$. As $\mu(\mb A(i))=\mu(\mb C(i))+\mu(\mb C'(i))$ for every $i\in I$, it follows from Proposition~\ref{p_sum}(i) that the family $\mu\circ \mb A$ is summable in $\mf A$ and, hence, $A\cup B\in\mathcal{L}$. Thus, $\mathcal{L}$ is a $\delta$-ring.
\end{proof}

The next theorem is the main result of this section.

\begin{theorem}\label{ll16}
  Let $\mf A$ be an HTM and $\mu$ be an $\mf A$-valued premeasure. Then $\Sigma_{\mf A}(\mu) \subset \mathcal{L}_{\mf A}(\mu)$. If $\mf A$ is regular, then $\Sigma_{\mf A}(\mu) = \mathcal{L}_{\mf A}(\mu)$. 
\end{theorem} 

In Appendix~\ref{s_example1}, we construct a nonregular HTM $\mf A$ and an $\mf A$-valued premeasure $\mu$ such that $\Sigma_{\mf A}(\mu) \neq \mathcal{L}_{\mf A}(\mu)$. Thus, the regularity assumption in Theorem~\ref{ll16} cannot be dropped.

The proof of Theorem~\ref{ll16} is based on the following two lemmas.

\begin{lemma}\label{l_SigmaA}
  Let $\mf A$ be an HTM and $\mu$ be an $\mf A$-valued premeasure. Then $\Sigma_{\mf A}(\mu) \subset \mathcal{L}_{\mf A}(\mu)$.
\end{lemma}
\begin{proof}
  Let $\nu = \as[\mf A]\mu$ and $A\in D_\nu$. Let $\mb A\colon I\to D_\mu$ be a countable disjoint family such that $\mb A(i)\subset A$ for all $i\in I$. Since $\nu$ is an extension of $\mu$, we have $\mb A(i)\in D_\nu$ for every $i\in I$ and $\mu\circ \mb A = \nu\circ \mb A$. As the set $\bigcup_{i\in I}\mb A(i)$ belongs to $D_\nu$ by Lemma~\ref{ll1}(i) and $\nu$ is $\sigma$-additive, we conclude that the family $\mu\circ\mb A$ is summable in $\mf A$ and, therefore, $A\in\mathcal{L}_{\mf A}(\mu)$. Thus, $D_\nu\subset \mathcal{L}_{\mf A}(\mu)$.
\end{proof}

\begin{lemma}\label{l_sc}
  Let $\mf A$ be a regular HTM and $\mu$ be an $\mf A$-valued premeasure. There is an $\mf A$-valued premeasure $\nu$ such that $D_\nu = \mathcal{L}_{\mf A}(\mu)$ and $\nu$ is an extension of $\mu$.
\end{lemma}
\begin{proof}
  Let $\mathcal{L} = \mathcal{L}_{\mf A}(\mu)$. Suppose $A\in \mathcal{L}$ and there are two countable partitions $\fm{\mb A}{I}$ and $\fm{\mb A'}{J}$ of $A$ into elements of $D_\mu$. Since $A\in\mathcal{L}$, the families $\mu\circ \mb A$ and $\mu\circ \mb A'$ are summable in $\mf A$. Let $\fm{\mb B}{I\times J}$ be such that $\mb B(i,j)=\mb A(i)\cap \mb A'(j)$ for every $i\in I$ and $j\in J$. As $\mb B$ is a countable partition of $A$ into elements of $D_\mu$, the family $\mu\circ \mb B$ is also summable in $\mf A$. For every $i\in I$, the family $\{\mb B(i,j)\}_{j\in J}$ is a countable partition of $\mb A(i)$ into elements of $D_\mu$. 
  For every $j\in J$, the family $\{\mb B(i,j)\}_{i\in I}$ is a countable partition of $\mb A'(j)$ into elements of $D_\mu$. By the $\sigma$-additivity of $\mu$, it follows that  $\mu(\mb A(i))=\sum_{j\in J}\mu(\mb B(i,j))$ for every $i\in I$ and $\mu(\mb A'(j))=\sum_{i\in I}\mu(\mb B(i,j))$ for every $j\in J$.
  By Proposition~\ref{p_fubiniRTS}, we conclude that 
  \[
    \sum_{i\in I}\mu(\mb A(i)) = \sum_{i\in I}\sum_{j\in J} \mu(\mb B(i,j)) = \sum_{k\in I\times J} \mu(\mb B(k)) = \sum_{j\in J}\sum_{i\in I} \mu(\mb B(i,j)) = \sum_{j\in J}\mu(\mb A'(j)).
  \]
  Since $\mathcal{L}\subset\q\mu$, Lemma~\ref{ll1}(ii) implies that every element of $\mathcal{L}$ has a countable partition into elements of $D_\mu$. Let $\nu$ be a map such that $D_\nu = \mathcal{L}$ and
  \begin{equation}\label{mu}
    \nu(A)=\sum_{i\in I}\mu(\mb A(i))
  \end{equation}
  for every $A\in\mathcal{L}$ and every countable partition $\mb A\colon I\to D_\mu$ of $A$. As shown above, the right-hand side of~(\ref{mu}) does not depend on the choice of partition and, therefore, $\nu$ is well defined. Lemma~\ref{l_SigmaA} implies that $D_\mu \subset D_\nu$ and we obviously have $\nu(A) = \mu(A)$ for every $A\in D_\mu$.

  We now show that $\nu$ is $\sigma$-additive. Suppose $A\in D_\nu$ and $\fm{\mb A}{I}$ is a countable partition of $A$ into elements of $D_\nu$. 
  Let $\fm{\mb A'}{J}$ be a countable partition of $A$ into elements of $D_\mu$. Let $\fm{\mb B}{I\times J}$ be such that $\mb B(i,j)=\mb A(i)\cap \mb A'(j)$ for every $i\in I$ and $j\in J$. In view of Lemma~\ref{ll1}(iii), we have $\mb B(i,j)\in D_\mu$ for every $i\in I$ and $j\in J$. Thus, $\mb B$ is a countable partition of $A$ into elements of $D_\mu$. Hence, $\nu(A)=\sum_{k\in I\times J} \mu(\mb B(k))$. On the other hand, for every $i\in I$, the family $\{\mb B(i,j)\}_{j\in J}$ is a countable partition of $\mb A(i)$ into elements of $D_\mu$ and, therefore, $\nu(\mb A(i))=\sum_{j\in J} \mu(\mb B(i,j))$. By Proposition~\ref{p_fubiniRTS}(i), it follows that
  \[
    \nu(A)= \sum_{i\in I}\sum_{j\in J} \mu(\mb B(i,j)) = \sum_{i\in I}\nu(\mb A(i)).
  \]
  This means that $\nu$ is $\sigma$-additive. As $\mathcal{L}$ is a $\delta$-ring by Proposition~\ref{p_Sigdelta}, $\nu$ is an $\mf A$-valued premeasure.
\end{proof}

\begin{remark}
  The above proof of Lemma~\ref{l_sc} relies on Proposition~\ref{p_fubiniRTS}. As shown by Example~\ref{e_nonreg}, this statement breaks down for nonregular HTMs.
\end{remark}

\begin{remark}
  When $\mf A$ is a complete uniform HTM and $\mathcal{L}_{\mf A}(\mu) = \q\mu$, Lemma~\ref{l_sc} follows from Lemma~2.6 in~\cite{FoxMorales1983}. 
\end{remark}

\begin{proof}[Proof of Theorem~\ref{ll16}]
  Let $\nu = \as[\mf A]\mu$ and $\mathcal{L} = \mathcal{L}_{\mf A}(\mu)$. Lemma~\ref{l_SigmaA} implies that $D_\nu\subset \mathcal{L}$. Suppose $\mf A$ is regular. By Lemma~\ref{l_sc}, there is an $\mf A$-valued premeasure $\nu'$ such that $D_{\nu'} = \mathcal{L}$ and $\nu'$ is an extension of $\mu$. Since $\mathcal{L}\subset \q\mu$, we have $\mu\asymp\nu'$ and, therefore, $\nu\asymp\nu'$ by Proposition~\ref{l_trans}. As $\nu$ is a measure, we conclude that $\mathcal{L}\subset D_\nu$. Thus, $D_\nu = \mathcal{L}$.
\end{proof}

We now illustrate the constructions considered above with the simple example of a weighted counting measure. We use the notations $\mathcal{F}[f]$ and $\bigS f$ introduced in Sec.~\ref{s_add} before Definition~\ref{dd5}.

\begin{proposition}\label{p_Spremeasure}
  Let $\mf A$ be a monoid and $f$ be an $\mf A$-valued map. Then $\bigS f$ is an $\mf A$-valued content. If $\mf A$ is an HTM, then $\bigS f$ is an $\mf A$-valued premeasure. 
\end{proposition}
\begin{proof}
  Let $\mu = \bigS f$. The additivity of $\mu$ follows immediately from Proposition~\ref{p_finsumunion}. Since $\mathcal{F}[f]$ is obviously a $\delta$-ring, we conclude that $\mu$ is an $\mf A$-valued content. Suppose $\mf A$ is an HTM, $K\in \mathcal{F}[f]$, and $\fm{\mb J}{L}$ is a countable partition of $K$. By Proposition~\ref{p_sumfin}, we have $\mu(K) = \sum_{k\in K}f(k)$ and $\mu(\mb J(l)) = \sum_{k\in\mb J(l)}f(k)$ for every $l\in L$. By Proposition~\ref{p_sumunionfin}, it follows that $\mu(K) = \sum_{l\in L}\mu(\mb J(l))$. Thus, $\mu$ is $\sigma$-additive and, hence, is an $\mf A$-valued premeasure.   
\end{proof}

Given a map $f$, we let $\mathcal{C}[f]$ denote the set of all countable subsets of $D_f$. If $\mf A$ is a monoid and $f$ is an $\mf A$-valued map, then we obviously have
\begin{equation}
  \label{eq:sigFI}
  \q{\bigS f} = \mathcal{C}[f].
\end{equation}
Let $\mf A$ be an HTM and $f$ be an $\mf A$-valued map. We define
\begin{equation}
  \label{eq:LAf}
  \mathcal{L}_{\mf A}^f = \set{K}{K\in \mathcal{C}[f]\mbox{ and }\sm{\mf A}{j\in J}{f(j)}\mbox{ for every }J\subset K}.
\end{equation}

\begin{lemma}\label{l_LL}
  Let $\mf A$ be an HTM and $f$ be an $\mf A$-valued map. Then $\mathcal{L}_{\mf A}(\bigS f) = \mathcal{L}_{\mf A}^f$.
\end{lemma}
\begin{proof}
  Let $\mu = \bigS f$. By Proposition~\ref{p_Spremeasure}, $\mu$ is an $\mf A$-valued premeasure.  Let $K\in \mathcal{L}_{\mf A}^f$. Then $K\in\mathcal{C}[f]$ and, hence, $K\in \q\mu$ by~(\ref{eq:sigFI}). Let $\mb J\colon L\to D_\mu$ be a disjoint family such that $\mb J(l)\subset K$ for every $l\in L$. By Proposition~\ref{p_sumfin}, we have $\mu(\mb J(l)) = \sum_{i\in\mb J(l)}f(i)$ for every $l\in L$. Let $J = \bigcup_{l\in L}\mb J(l)$. Since $J\subset K$, we have $\sm{\mf A}{j\in J}{f(j)}$ and Proposition~\ref{p_sumunionfin} implies that $\sm{\mf A}{l\in L}{\mu(\mb J(l))}$. This means that $K\in \mathcal{L}_{\mf A}(\mu)$ and, hence, $\mathcal{L}_{\mf A}^f\subset \mathcal{L}_{\mf A}(\mu)$. Conversely, let $K\in \mathcal{L}_{\mf A}(\mu)$. Then $K\in\q\mu$ and, therefore, $K\in \mathcal{C}[f]$ by~(\ref{eq:sigFI}). Let $J\subset K$ and $\mb A = \{\{j\}\}_{j\in J}$. Then $\mb A$ is a countable disjoint family.  As $\mb A(j)\subset K$, $\mb A(j)\in D_\mu$  and $\mu(\mb A(j)) = f(j)$ for every $j\in J$, it follows from~(\ref{SigmaA}) that $\sm{\mf A}{j\in J}{f(j)}$. Thus, $K\in \mathcal{L}_{\mf A}^f$. This means that $\mathcal{L}_{\mf A}(\mu)\subset\mathcal{L}_{\mf A}^f$  and, hence, $\mathcal{L}_{\mf A}(\mu) = \mathcal{L}_{\mf A}^f$.  
\end{proof}

\begin{proposition}\label{p_assS}
  Let $\mf A$ be an HTM, $f$ be an $\mf A$-valued map, and $\nu = \as[\mf A]{\bigS f}$. Then $\nu(K) = \sum_{k\in K}f(k)$ for every $K\in D_\nu$. If $\mf A$ is regular, then $D_\nu = \mathcal{L}_{\mf A}^f$. If $\mf A$ is a sequentially complete HTG, then $D_\nu = \set{K}{K\in \mathcal{C}[f]\mbox{ and }\sm{\mf A}{k\in K}{f(k)}}$.
\end{proposition}
\begin{proof}
   Let $\mu = \bigS f$. By Proposition~\ref{p_Spremeasure}, $\mu$ is an $\mf A$-valued premeasure and, therefore, $\nu$ is an $\mf A$-valued measure. Let $K\in D_\nu$. Then $K\in \q\mu$ and, hence, $K$ is countable by~(\ref{eq:sigFI}). Since $\nu$ is an extension of $\mu$, we have $\{k\}\in D_\nu$ and $\nu(\{k\}) = f(k)$ for every $k\in K$. As $\{\{k\}\}_{k\in K}$ is a partition of $K$, the $\sigma$-additivity of $\nu$ implies that $\nu(K) = \sum_{k\in K}f(k)$. If $\mf A$ is regular, then Theorem~\ref{ll16} and Lemma~\ref{l_LL} ensure that $D_\nu = \mathcal{L}_{\mf A}^f$. If $\mf A$ is a sequentially complete HTG, then $\mathcal{L}_{\mf A}^f = \set{K}{K\in \mathcal{C}[f]\mbox{ and }\sm{\mf A}{k\in K}{f(k)}}$ by Proposition~\ref{p_seqcompl}.   
\end{proof}

Using Proposition~\ref{p_assS}, we give an example showing that Corollary~\ref{l_subgr1} is not true without the assumption that $\mf B$ be sequentially closed in $\mf A$.

\begin{example}\label{e_sub}
  We construct a $\mathbb Q$-valued measure $\mu$ that is not an $\R$-valued measure. Let $f\colon \N\to \mathbb Q$ be such that $f(n) = 1/n!$ for every $n\in\N$. Let $\eta = \bigS^{\mathbb Q}f $. By Proposition~\ref{p_Spremeasure}, $\eta$ is a $\mathbb Q$-valued premeasure. Let $\mu = \as[\mathbb Q]\eta$ and $\nu = \as[\R]\eta$. Clearly, $\mu$ is a $\mathbb Q$-valued measure. By Proposition~\ref{l_trans} applied to $\mf A = \R$, we have  $\mu\asymp \nu$. Since $\eta = \bigS^\R f$, Proposition~\ref{p_assS} implies that $\N\in D_\nu$ and $\nu(\N) = \sum^\R_{n\in\N}1/n! = e -1$. On the other hand, $\N\notin D_\mu$ because otherwise $f$ would be summable in $\mathbb Q$ by Proposition~\ref{p_assS} and, hence, $e-1$ would be rational by Proposition~\ref{p_sumsubsgr}. Thus, $\mu$ is not an extension of $\nu$ and, therefore, is not an $\R$-valued measure.   
\end{example}

If $\mf A$ is a sequentially complete HTG, then the description of $\Sigma_{\mf A}(\mu)$ given by Theorem~\ref{ll16} can be somewhat simplified.

\begin{proposition}\label{p_seqcomplSigma}
  Let $\mf A$ be a sequentially complete HTG and $\mu$ be an $\mf A$-valued premeasure. Then
  \begin{multline}\label{Sigma_sc}
    \Sigma_{\mf A}(\mu) =  \sett{A}{A\in \q\mu \mbox{ and the family $\mu\circ \mb A$ is summable in $\mf A$}}{\mbox{for every countable partition $\mb A$ of $A$ into elements of $D_\mu$}}.
  \end{multline}
\end{proposition}
\begin{proof}
  Let $\Sigma = \Sigma_{\mf A}(\mu)$ and $\mathcal Q$ be the set in the right-hand side of~(\ref{Sigma_sc}). We obviously have $\Sigma\subset \mathcal Q$. Let $A\in\mathcal Q$ and $\mb A\colon I\to D_\mu$ be a countable disjoint family such that $\mb A(i)\subset A$ for all $i\in I$. 
  Since $A\in \q\mu$, Lemma~\ref{ll1}(ii) implies that there exists a countable partition $\mb B\colon J\to D_\mu$ of $A\setminus \bigcup_{i\in I}\mb A(i)$. 
  Let $\mb C\colon I\sqcup J\to D_\mu$ be such that $\mb C(1,i) = \mb A(i)$ for every $i\in I$ and $\mb C(2,j) = \mb B(j)$ for every $j\in J$. By Lemma~\ref{p_disjpart2}, $\mb C$ is a countable partition of $A$. Since $A\in \mathcal Q$, it follows that the family $\mu\circ \mb C$ is summable in $\mf A$. Let $I' = \{1\}\times I$ and $\mb C' = \mb C|_{I'}$. By Proposition~\ref{p_seqcompl}, the family $\mu\circ \mb C'$ is summable in $\mf A$. Let $\fm{\tau}{I}$ be such that $\tau(i) = (1,i)$ for every $i\in I$. Then $\tau$ is a bijection of $I$ onto $I'$ and $\mb C'\circ\tau = \mb A$. It follows that $\mu\circ\mb A = (\mu\circ \mb C')\circ \tau$ and, hence, $\mu\circ\mb A$ is summable in $\mf A$ by Proposition~\ref{p_sumchange}. This means that $A\in \mathcal{L}_{\mf A}(\mu)$ and, therefore, $A\in \Sigma$ by Theorem~\ref{ll16}. Thus, $\mathcal Q\subset \Sigma$ and, hence, $\mathcal Q = \Sigma$.	
\end{proof}

As shown by the next example, the condition that $\mf A$ be sequentially complete cannot be dropped in Proposition~\ref{p_seqcomplSigma}.

\begin{example}
  Let $f\colon \N\to \mathbb Q$ be such that $f(n) = 2^{-n}$ for every $n\in\N$. We have $\sum^{\mathbb Q}_{n\in \N}f(n) = 1$. Let $\mu = \bigS^{\mathbb Q}f $. By Proposition~\ref{p_Spremeasure}, $\mu$ is a $\mathbb Q$-valued premeasure. If $\mb J\colon I\to D_\mu$ is a partition of $\N$, then $\mu\circ\mb J$ is summable in $\mathbb Q$ by Propositions~\ref{p_sumfin} and~\ref{p_sumunionfin}. We claim that $\N\notin\Sigma_{\mathbb Q}(\mu)$. If $J$ and $J'$ are infinite subsets of $\N$, then $\sum^\R_{n\in J}f(n) = \sum^\R_{n\in J'}f(n)$ if and only if $J = J'$. Since the set of all infinite subsets of $\N$ is uncountable, we conclude that $\sum^\R_{n\in J}f(n)$ is irrational for some $J\subset \N$. By Proposition~\ref{p_sumsubsgr}, we conclude that $\sm{\mathbb Q}{n\in J}f(n)$ does not hold and, hence, $\N\notin\mathcal{L}^f_{\mathbb Q}$.  Our claim therefore follows from~Proposition~\ref{p_assS}.   
\end{example}

\section{Spaces of measures}
\label{s_spaces}

Let $\mf A$ be a monoid and $\mu$ and $\nu$ be $\mf A$-valued maps. The map $\eta\colon D_\mu\cap D_\nu\to\mf A$ such that $\eta(A) = \mu(A)+\nu(A)$ for every $A\in D_\mu\cap D_\nu$ is called the pointwise sum of $\mu$ and $\nu$ in $\mf A$ and is denoted by $\mu\dotplus\nu$. The operation $\dotplus$ is commutative and associative. If $\mf A$ is an HTM and $\mu$ and $\nu$ are $\mf A$-valued  premeasures, then so is $\mu\dotplus\nu$.

Let $\mf A$ be an HTM. For $\mf A$-valued premeasures $\mu$ and $\nu$, we define the $\mf A$-valued measure $\mu\boxplus\nu$ by
\begin{equation}\label{meas_sum}
  \mu\boxplus\nu = \as[\mf A]{\mu\dotplus\nu}.
\end{equation}
Since $\dotplus$ is commutative, we have $\mu\boxplus\nu = \nu\boxplus\mu$.

If the monoid $\mf A$ has to be explicitly indicated, we write $\mu\dotplus_{\mf A}\nu$ and $\mu\boxplus_{\mf A}\nu$ in place of $\mu\dotplus\nu$ and $\mu\boxplus\nu$ respectively.

\begin{proposition}\label{l_boxplus}
  Let $\mf A$ be an HTM and $\mu$ and $\nu$ be $\mf A$-valued premeasures such that $\q{\mu}=\q{\nu}$.
  \begin{enumerate}
  \item[(i)]  $\q{\mu\dotplus\nu}=\q{\mu\boxplus\nu}=\q{\mu}$.
  \item[(ii)] $\mu\boxplus\nu = \as[\mf A]{\mu}\boxplus\nu = \mu\boxplus\as[\mf A]{\nu} = \as[\mf A]{\mu}\boxplus\as[\mf A]{\nu}$.  
  \end{enumerate} 
\end{proposition}
\begin{proof}
  (i) Since $\q{\mu\dotplus\nu} = \sigma(D_\mu\cap D_\nu)$ and $\sigma(D_\mu\cap D_\nu)=\q{\mu}$ by Lemma~\ref{l3sigma}, the statement follows from~(\ref{Qmu}) and~(\ref{meas_sum}). 
  \par\medskip\noindent
  (ii) By~(\ref{Qmu}) and~(i), the premeasures $\mu\dotplus\nu$ and $\as[\mf A]{\mu}\dotplus\nu$ have the same host $\sigma$-ring. Since the latter extends the former, they are similar. By Proposition~\ref{p_asssimilar} and~(\ref{meas_sum}), we conclude that $\mu\boxplus\nu = \as[\mf A]{\mu}\boxplus\nu$. Interchanging $\mu$ and $\nu$ and using the commutativity of $\boxplus$, we obtain $\mu\boxplus\nu = \mu\boxplus\as[\mf A]{\nu}$. Finally, applying the first equality above with $\as[\mf A]{\nu}$ in place of $\nu$ gives $\mu\boxplus\as[\mf A]{\nu} = \as[\mf A]{\mu}\boxplus\as[\mf A]{\nu}$.
\end{proof}  

\begin{proposition}\label{l_comp}
  Let $\mf A$ be an HTM and 
  $\mu$, $\nu$, and $\eta$ be $\mf A$-valued premeasures such that $\q\mu=\q\nu=\q\eta$. Then $(\mu\boxplus\nu)\boxplus\eta=\mu\boxplus(\nu\boxplus\eta)$. 
\end{proposition}
\begin{proof}
  By Proposition~\ref{l_boxplus}(i), $\q{\mu\dotplus\nu} = \q{\eta}$. Hence, by~(\ref{meas_sum}) and Proposition~\ref{l_boxplus}(ii), 
  \[
    (\mu\boxplus\nu)\boxplus\eta = (\mu\dotplus \nu)\boxplus\eta = \as[\mf A]{(\mu\dotplus \nu)\dotplus\eta}.
  \]
  Similarly, $\mu\boxplus(\nu\boxplus\eta) = \as[\mf A]{\mu\dotplus (\nu\dotplus\eta)}$. Since $\dotplus$ is associative, the result follows.  
\end{proof}

Given a set $\mathcal Q$ and a monoid $\mf A$, we let $\mathbf 0_{\mf A}(\mathcal Q)$ denote the constant map on $\mathcal Q$ with value $0_{\mf A}$. If $\mf A$ is an HTM and $\mathcal Q$ is a $\delta$-ring, then $\mathbf 0_{\mf A}(\mathcal Q)$ is an $\mf A$-valued premeasure and 
\begin{equation}\label{ass_zero}
  \as[\mf A]{\mathbf 0_{\mf A}(\mathcal Q)} = \mathbf 0_{\mf A}(\sigma(\mathcal Q)).
\end{equation}
Indeed, $\mathbf 0_{\mf A}(\sigma(\mathcal Q))$ is a measure by Proposition~\ref{p_Dmusigma} and is similar to $\mb 0_{\mf A}(\mathcal Q)$, so the equality follows from Proposition~\ref{p_ass}.

Let $\mf A$ be an HTM and $\mathcal Q$ be a $\sigma$-ring. We let $\ms M(\mathcal Q,\mf A)$ denote the set of all $\mf A$-valued measures $\mu$ such that $\q{\mu} = \mathcal Q$. By Proposition~\ref{l_boxplus}(i), we have $\mu\boxplus\nu\in \ms M(\mathcal Q,\mf A)$ for every $\mu,\nu\in \ms M(\mathcal Q,\mf A)$.
We define addition on $\ms M(\mathcal Q,\mf A)$ by $\mu+\nu=\mu\boxplus\nu$ for every $\mu,\nu\in \ms M(\mathcal Q,\mf A)$.

\begin{proposition}\label{t_semigrspace}
  Let $\mf A$ be an HTM and $\mathcal Q$ be a $\sigma$-ring. Then $\ms M(\mathcal Q,\mf A)$ is a monoid whose neutral element is $\mathbf 0_{\mf A}(\mathcal Q)$.
\end{proposition}
\begin{proof}
  Let $\ms M = \ms M(\mathcal Q,\mf A)$ and $\zeta = \mathbf 0_{\mf A}(\mathcal Q)$. By Proposition~\ref{p_Dmusigma}, we have $\zeta\in\ms M$. Since $\mu\dotplus\zeta =\mu$, it follows from~(\ref{meas_sum}) and Proposition~\ref{p_measequiv} that $\mu\boxplus\zeta = \mu$ for every $\mu\in\ms M$. The addition on $\ms M$ is commutative and, by Proposition~\ref{l_comp}, associative. Thus, $\ms M$ is a monoid whose neutral element is $\zeta$.
\end{proof}

Let $\mf A$ be a group and $\mu$ and $\nu$ be $\mf A$-valued maps. The pointwise minus of $\mu$, denoted by $\dotdiv\mu$, is the map on $D_\mu$ given by $(\dotdiv\mu)(A) = -\mu(A)$ for every $A\in D_\mu$. The pointwise difference $\mu\dotdiv \nu$ of $\mu$ and $\nu$ is defined by $\mu\dotdiv \nu = \mu\dotplus(\dotdiv\nu)$. If $\mf A$ is an HTG and $\mu$ and $\nu$ are $\mf A$-valued premeasures, then so are $\dotdiv\mu$ and $\mu\dotdiv\nu$.

Let $\mf A$ be an HTG. For $\mf A$-valued premeasures $\mu$ and $\nu$, we define the $\mf A$-valued measure $\mu\boxminus\nu$ by
\begin{equation}\label{meas_diff}
  \mu\boxminus\nu = \mu\boxplus(\dotdiv \nu) = \as[\mf A]{\mu\dotdiv\nu}.
\end{equation}
Since $\{-x\}_{x\in \mf A}$ is a topological isomorphism from $\mf A$ to itself, Corollary~\ref{c_iso} implies that
\begin{equation}\label{dotdiv}
  \as[\mf A]{\dotdiv \mu} = \dotdiv\as[\mf A]{\mu}
\end{equation}
for every $\mf A$-valued premeasure $\mu$.

\begin{proposition}\label{l_boxminus}
  Let $\mf A$ be an HTG and $\mu$ and $\nu$ be $\mf A$-valued premeasures such that $\q{\mu}=\q{\nu}$.
  Then $\q{\mu\dotdiv\nu}=\q{\mu\boxminus\nu}=\q{\mu}$ and $\mu\boxminus\nu = \as[\mf A]{\mu}\boxminus\nu = \mu\boxminus\as[\mf A]{\nu} = \as[\mf A]{\mu}\boxminus\as[\mf A]{\nu}$. 
\end{proposition}
\begin{proof}
  Since $D_{\dotdiv\nu} = D_\nu$, we have $\q{\dotdiv\nu} = \q\nu$. Hence, the statement follows from~(\ref{meas_diff}), (\ref{dotdiv}), and Proposition~\ref{l_boxplus}.
\end{proof}

\begin{theorem}\label{l_ab}
  Let $\mf A$ be an HTG and $\mathcal Q$ be a $\sigma$-ring. Then $\ms M(\mathcal Q,\mf A)$ is a group. We have $-\mu = \dotdiv\mu$ for every $\mu\in \ms M(\mathcal Q,\mf A)$. 
\end{theorem}
\begin{proof}
  Let $\ms M = \ms M(\mathcal Q,\mf A)$ and $\zeta = \mathbf 0_{\mf A}(\mathcal Q)$. By Proposition~\ref{t_semigrspace}, $\ms M$ is a monoid whose neutral element is $\zeta$. If $\mu\in\ms M$, then $\dotdiv\mu\in \ms M$ by~(\ref{dotdiv}) and Proposition~\ref{p_measequiv} and it follows from~(\ref{meas_sum}) and~(\ref{ass_zero}) that $\mu\boxplus(\dotdiv\mu) = \zeta$. Thus, $\ms M$ is a group and $-\mu = \dotdiv \mu$ for every $\mu\in \ms M$. 
\end{proof}

\begin{proposition}\label{t_semigrsubspace}
  Let $\mf A$ be an HTM, $\mathcal Q$ be a $\sigma$-ring, and $\mf B$ be a sequentially closed topological submonoid of $\mf A$. Then $\ms M(\mathcal Q,\mf B)$ is a submonoid of $\ms M(\mathcal Q,\mf A)$. If $\mf A$ and $\mf B$ are topological groups, then $\ms M(\mathcal Q,\mf B)$ is a subgroup of $\ms M(\mathcal Q,\mf A)$.
\end{proposition}
\begin{proof}
  Let $\ms M = \ms M(\mathcal Q,\mf A)$ and $\ms M' = \ms M(\mathcal Q,\mf B)$. By Proposition~\ref{t_semigrspace}, $\ms M$ and $\ms M'$ are monoids. By Corollary~\ref{l_subgr1}, we have $\ms M'\subset\ms M$. For every $\mu,\nu\in \ms M'$, Proposition~\ref{l_subgr} and~(\ref{meas_sum}) give $\mu\boxplus_{\mf B} \nu = \mu\boxplus_{\mf A} \nu$, so the addition in $\ms M'$ is induced by that in $\ms M$. Since $0_{\mf B} = 0_{\mf A}$, we have $0_{\ms M} =  0_{\ms M'}$ by Proposition~\ref{t_semigrspace}. Thus, $\ms M'$ is a submonoid of $\ms M$. If $\mf A$ and $\mf B$ are topological groups, then $\ms M$ and $\ms M'$ are groups by Theorem~\ref{l_ab} and, therefore, $\ms M'$ is a subgroup of $\ms M$.
\end{proof}

Suppose a set $\mf A$ is endowed with a left multiplication by elements of a set $\mathbb K$ (i.e., a map $(k,x)\mapsto k x$ from $\mathbb K\times\mf A$ to $\mf A$). Given an $\mf A$-valued map $\mu$ and $k\in\mathbb K$, the map $A\mapsto k\mu(A)$ on $D_\mu$ is called the pointwise multiple of $\mu$ by $k$ and is denoted by $k\cdot \mu$.

We say that  $\mf A$ is a left $\mathbb{K}$-monoid if $\mf A$ is a monoid endowed with a left multiplication by elements of $\mathbb{K}$ such that $k(x+y) = k x + k y$ and $k0 = 0$ for every $k\in \mathbb{K}$ and $x,y\in \mf A$. We say that $\mf A$ is a continuous left $\mathbb{K}$-monoid if $\mf A$ is both a topological monoid and a left $\mathbb{K}$-monoid and $x\mapsto k x$ is a continuous map from $\mf A$ to itself for every $k\in \mathbb{K}$. We use the abbreviation HCLM$_{\mathbb{K}}$ for Hausdorff continuous left $\mathbb{K}$-monoid. 

Let $\mf A$ be an HCLM$_{\mathbb{K}}$, $\mu$ be an $\mf A$-valued premeasure, and $k\in \mathbb{K}$. By Proposition~\ref{p_sum}(ii), $k\cdot\mu$ is an $\mf A$-valued premeasure. We define the $\mf A$-valued measure $k\boxdot\mu$ by
\begin{equation}\label{meas_mult}
  k\boxdot\mu = \as[\mf A]{k\cdot\mu}.
\end{equation} 
Since $D_{k\cdot\mu}=D_\mu$, (\ref{Qmu}) and~(\ref{meas_mult}) imply that
\begin{equation}\label{Qboxdot}
  \q{k\cdot\mu} = \q{k\boxdot\mu} = \q{\mu}.
\end{equation}
By~(\ref{Qmu}) and~(\ref{Qboxdot}), we have 
$\q{k\boxdot\as[\mf A]{\mu}} = \q{\mu}$. Since both $k\boxdot\mu$ and $k\boxdot\as[\mf A]{\mu}$ coincide with $k\cdot\mu$ on $D_\mu$, it follows from Theorem~\ref{t_unique} that
\begin{equation}\label{asboxd}
  k\boxdot\mu = k\boxdot\as[\mf A]{\mu}.
\end{equation} 

\begin{proposition}\label{p_HCLMdistr}
  Let $\mf A$ be an HCLM$_{\mathbb{K}}$, $k\in \mathbb{K}$, and $\mu$ and $\nu$ be $\mf A$-valued premeasures such that $\q\mu=\q\nu$. Then $k\boxdot(\mu\boxplus\nu) = (k\boxdot\mu)\boxplus(k\boxdot\nu)$.
\end{proposition}
\begin{proof}
  By~(\ref{meas_sum}), (\ref{meas_mult}), and~(\ref{asboxd}), we have
  \[
    k\boxdot(\mu\boxplus\nu) = k\boxdot\as[\mf A]{\mu\dotplus\nu} = k\boxdot (\mu\dotplus\nu) = \as[\mf A]{k\cdot(\mu\dotplus\nu)}.  
  \] 
  Since $\q{k\cdot\mu}=\q{k\cdot\nu}$ by~(\ref{Qboxdot}), Proposition~\ref{l_boxplus}(ii), (\ref{meas_sum}), and~(\ref{meas_mult}) imply that
  \[
    (k\boxdot\mu)\boxplus(k\boxdot\nu) = \as[\mf A]{k\cdot\mu}\boxplus\as[\mf A]{k\cdot\nu} = (k\cdot\mu)\boxplus(k\cdot\nu) = \as[\mf A]{k\cdot\mu\dotplus k\cdot\nu}.
  \]
  As $k\cdot(\mu\dotplus\nu) = k\cdot\mu\dotplus k\cdot\nu$, the result follows.
\end{proof}

Let $\mf A$ be an HCLM$_{\mathbb{K}}$ and $\mathcal Q$ be a $\sigma$-ring. It follows from~(\ref{Qboxdot}) that $k\boxdot\mu\in \ms M(\mathcal Q,\mf A)$ for every $k\in\mathbb K$ and $\mu\in \ms M(\mathcal Q,\mf A)$. We define the left multiplication by elements of $\K$ on $\ms M(\mathcal Q,\mf A)$ by $k\mu = k\boxdot\mu$ for every $k\in\mathbb K$ and $\mu\in \ms M(\mathcal Q,\mf A)$.

\begin{proposition}\label{p_HCLM}
  Let $\mf A$ be an HCLM$_{\mathbb{K}}$ and $\mathcal Q$ be a $\sigma$-ring. Then $\ms M(\mathcal Q,\mf A)$ is a left $\mathbb{K}$-monoid.
\end{proposition}
\begin{proof}
  Let $\ms M = \ms M(\mathcal Q,\mf A)$ and $\zeta = \mathbf 0_{\mf A}(\mathcal Q)$. By Proposition~\ref{t_semigrspace}, $\ms M$ is a monoid whose neutral element is $\zeta$. Since $k\cdot\zeta = \zeta$, it follows from~(\ref{meas_mult}) and Proposition~\ref{p_measequiv} that $k\zeta = \zeta$ for every $k\in \mathbb{K}$. By Proposition~\ref{p_HCLMdistr}, we have $k(\mu+\nu) = k\mu + k\nu$ for every $k\in \mathbb{K}$ and $\mu,\nu\in \ms M$. Thus, $\ms M$ is a left $\mathbb{K}$-monoid. 
\end{proof}

Let $\mathbb{K}$ be an algebraic ring and $\mf A$ be both a group and a left $\mathbb{K}$-monoid. If $(k_1 + k_2)x = k_1 x + k_2 x$, $k_1(k_2 x) = (k_1 k_2) x$, and $1x = x$ for every $k_1,k_2\in \mathbb{K}$ and $x\in \mf A$, then $\mf A$ is called a left $\mathbb{K}$-module. If, in addition, $\mathbb{K}$ is a field, then $\mf A$ is called a $\mathbb{K}$-vector space.

\begin{proposition}\label{l_distr}
  Let $\mathbb K$ be an algebraic ring and $\mf A$ be both an HCLM$_{\mathbb{K}}$ and a left $\mathbb K$-module.
  \begin{enumerate}
  \item[(i)]  Let $\mu$ be an $\mf A$-valued measure. If $k\in\mathbb K$ is invertible, then $k\boxdot \mu = k\cdot\mu$.
  \item[(ii)] If $k,l\in\mathbb K$ and $\mu$ is an $\mf A$-valued premeasure, then $(kl)\boxdot\mu=k\boxdot(l\boxdot\mu)$ and $(k+l)\boxdot\mu = (k\boxdot\mu)\boxplus(l\boxdot\mu)$.
  \end{enumerate}
\end{proposition}
\begin{proof}
  (i) If $k\in\mathbb K$ is invertible, then $x\mapsto k x$ is a topological isomorphism from $\mf A$ to itself. Hence, (\ref{meas_mult}), Proposition~\ref{p_measequiv}, and Corollary~\ref{c_iso} imply that $k\boxdot\mu = \as[\mf A]{k\cdot\mu} = k\cdot\as[\mf A]{\mu} = k\cdot\mu$.
  \par\medskip\noindent
  (ii) Let $\eta_1 = (kl)\boxdot\mu$, $\eta_2 = k\boxdot(l\boxdot\mu)$, $\zeta_1 = (k+l)\boxdot\mu$, and $\zeta_2 = (k\boxdot\mu)\boxplus(l\boxdot\mu)$. By~(\ref{meas_mult}), we have $\eta_1 = \as[\mf A]{(kl)\cdot\mu}$ and~(\ref{meas_mult}) and~(\ref{asboxd}) imply that 
  \[
    \eta_2 = k\boxdot\as[\mf A]{l\cdot\mu} = k\boxdot(l\cdot\mu) = \as[\mf A]{k\cdot (l\cdot\mu)}.
  \]
  Since $k\cdot (l\cdot\mu) = (kl)\cdot\mu$, we conclude that $\eta_1=\eta_2$. It follows from~(\ref{meas_mult}) that $\zeta_1 = \as[\mf A]{(k+l)\cdot\mu}$. Since $\q{k\cdot\mu}=\q{l\cdot\mu}$ by~(\ref{Qboxdot}), Proposition~\ref{l_boxplus}(ii) and~(\ref{meas_mult}) imply that
  \[
    \zeta_2 = \as[\mf A]{k\cdot\mu}\boxplus \as[\mf A]{l\cdot\mu} = (k\cdot\mu)\boxplus(l\cdot\mu)=\as[\mf A]{k\cdot\mu\dotplus l\cdot\mu},
  \] 
  whence $\zeta_1=\zeta_2$ because $(k+l)\cdot\mu = k\cdot\mu\dotplus l\cdot\mu$.
\end{proof} 

\begin{theorem}\label{t_module}
  Let $\mathcal{Q}$ be a $\sigma$-ring, $\mathbb K$ be an algebraic ring, and $\mf A$ be both an HTG and a left $\mathbb K$-module. Suppose $x\mapsto k x$ is a continuous map from $\mf A$ to itself for every $k\in \mathbb{K}$.  Then $\ms M(\mathcal Q,\mf A)$ is a left $\mathbb K$-module. If $k\in \mathbb{K}$ is invertible, then $k\mu = k\cdot\mu$ for every $\mu\in \ms M(\mathcal Q,\mf A)$. If $\mf B$ is a sequentially closed submodule of $\mf A$ endowed with the induced topology, then $\ms M(\mathcal Q,\mf B)$ is a submodule of $\ms M(\mathcal Q,\mf A)$.   
\end{theorem}
\begin{proof}
  Let $\ms M = \ms M(\mathcal Q,\mf A)$. By Theorem~\ref{l_ab}, $\ms M$ is a group. Since $\mf A$ is an HCLM$_{\K}$, Proposition~\ref{p_HCLM} implies that $\ms M$ is a left $\mathbb{K}$-monoid. By Proposition~\ref{l_distr}(ii), we have $(k_1 + k_2)\mu = k_1 \mu + k_2 \mu$, $k_1(k_2 \mu) = (k_1 k_2) \mu$ for every $k_1,k_2\in \mathbb{K}$ and $\mu\in \ms M$. If $k\in \mathbb{K}$ is invertible, then $k\mu = k\cdot\mu$ for every $\mu\in \ms M$ by Proposition~\ref{l_distr}(i). In particular, $1\mu = \mu$ for every $\mu\in\ms M$. Thus, $\ms M$ is a left $\mathbb{K}$-module. Suppose $\mf B$ is a sequentially closed submodule of $\mf A$. By Proposition~\ref{t_semigrsubspace}, $\ms M(\mathcal Q,\mf B)$ is a subgroup of $\ms M$. For every $k\in\K$ and $\mu\in \ms M(\mathcal Q,\mf B)$, we have $k\boxdot_{\mf B}\mu = k\boxdot_{\mf A}\mu$ by Proposition~\ref{l_subgr} and~(\ref{meas_mult}). This implies that the left multiplication in $\ms M(\mathcal Q,\mf B)$ is induced by that in $\ms M$, i.e., $\ms M(\mathcal Q,\mf B)$ is a submodule of $\ms M$.
\end{proof}

\begin{corollary}\label{c_vector}
  Let $\mathcal{Q}$ be a $\sigma$-ring, $\K$ be a topological field, and $\mf A$ be a Hausdorff topological $\K$-vector space. Then $\ms M(\mathcal Q,\mf A)$ is a $\mathbb K$-vector space. If $k\in \mathbb{K}$ is nonzero, then $k\mu = k\cdot\mu$ for every $\mu\in \ms M(\mathcal Q,\mf A)$. If $\mf B$ is a sequentially closed topological vector subspace of $\mf A$, then $\ms M(\mathcal Q,\mf B)$ is a $\K$-vector subspace  of $\ms M(\mathcal Q,\mf A)$.
\end{corollary}

Corollary~\ref{c_vector} implies, in particular, that $\ms M(\mathcal Q,\R)$ and $\ms M(\mathcal Q,\C)$ are real and complex vector spaces respectively. 
We now prove condition~\ref{item:2} formulated in Section~\ref{s_intro}.

\begin{proposition}
  $\ms M(\mathcal{Q},\R)$ is a real subspace of $\ms M(\mathcal{Q},\C)$. Every $\mu\in \ms M(\mathcal{Q},\C)$ has a unique representation $\mu = \mu'+i\mu''$, where $\mu',\mu''\in \ms M(\mathcal{Q},\R)$.
\end{proposition}
\begin{proof}
  By Corollary~\ref{c_vector}, $\ms M(\mathcal{Q},\R)$ is a real subspace of $\ms M(\mathcal{Q},\C)$. Let $\mu\in \ms M(\mathcal{Q},\C)$. Let $\nu' = \Re\circ\mu$ and $\nu'' = \Im\circ\mu$, where $\Re$ and $\Im$ are the maps on $\C$ taking each complex number to its real and imaginary parts respectively. Since $\Re$ and $\Im$ are continuous group homomorphisms from $\C$ to itself, $\nu'$ and $\nu''$ are $\C$-valued premeasures. Let $\mu' = \as[\C]{\nu'}$ and $\mu'' = \as[\C]{\nu''}$. As $\q{\nu'} = \q{\nu''} = \mathcal{Q}$, equality~(\ref{Qmu}) implies that $\mu',\mu''\in\ms M(\mathcal{Q},\C)$. Since $\as[\C]{i\cdot\nu''} = i\mu''$ by~(\ref{meas_mult}) and~(\ref{asboxd}) and $\mu = \nu'\dotplus(i\cdot\nu'')$, it follows from~(\ref{meas_sum}), (\ref{Qboxdot}), and Proposition~\ref{l_boxplus}(ii) that $\mu = \mu'+i\mu''$. By Propositions~\ref{p_subpremeasure} and~\ref{l_subgr}, we have $\mu',\mu''\in\ms M(\mathcal{Q},\R)$. Let $\eta',\eta''\in \ms M(\mathcal{Q},\R)$ be such that $\mu = \eta'+i\eta''$. Let $\zeta' = \mu' - \eta'$ and $\zeta'' = \mu'' - \eta''$. Then $\zeta' = -i\zeta''$ and, hence, $\zeta' = -i\cdot\zeta''$ by Corollary~\ref{c_vector}. As $\zeta'$ and $\zeta''$ are $\R$-valued maps, it follows that both $\zeta'$ and $\zeta''$ coincide with $\mb 0_{\C}(\mathcal{Q})$ on $D_{\zeta'}$. By Theorem~\ref{t_unique}, we conclude that $\zeta' = \zeta'' = 0$, i.e., $\mu' = \eta'$ and $\mu'' =\eta''$.
\end{proof}

Let $\mf A$ be a partially ordered set and $\mu$ and $\nu$ be $\mf A$-valued maps. We say that $\mu$ is pointwise less than or equal to $\nu$ (notation $\mu\dotleq\nu$) if $\mu(A)\leq\nu(A)$ for every $A\in D_\mu\cap D_\nu$. 

We call $\mf A$ an ordered monoid if $\mf A$ is both a monoid and a partially ordered set and, moreover, the inequality $x\leq y$ implies that $x+z\leq y+z$ for all $x,y,z\in\mf A$. 

Let $\mf A$ be both a topological monoid and an ordered monoid. We say that $\mf A$ is an ordered topological monoid if the set $\set{(x,y)}{(x,y)\in\mf A\times \mf A\mbox{ and } x\leq y}$ is sequentially closed in $\mf A\times\mf A$. We use the abbreviation HOTM for Hausdorff ordered topological monoid. 

\begin{proposition}\label{l_dotleq0}
  Let $\mf A$ be an HOTM, $\mu$ and $\nu$ be $\mf A$-valued $\sigma$-additive functions, and $\mathcal K$ be a $\delta$-ring such that $\mathcal K\subset D_\mu\cap D_\nu\subset\sigma(\mathcal K)$. Let $\mu(A)\leq \nu(A)$ for every $A\in \mathcal K$. Then $\mu\dotleq\nu$.
\end{proposition} 
\begin{proof}
  Let $X = \{(x,y)\in\mf A\times \mf A: x\leq y\}$ and $\fm{\eta}{D_\mu\cap D_\nu}$ be such that $\eta(A) = (\mu(A),\nu(A))$ for every $A\in D_\mu\cap D_\nu$. Then $\eta$ is an $(\mf A\times\mf A)$-valued $\sigma$-additive function and $\eta(A)\in X$ for every $A\in\mathcal K$. Since $X$ is a sequentially closed subset of $\mf A\times\mf A$, Proposition~\ref{p_Im0delta} implies that $\im\eta\subset X$. This means that $\mu\dotleq\nu$.
\end{proof}
    
\begin{proposition}\label{l_dotleq}
  Let $\mf A$ be an HOTM and $\mu$ and $\nu$ be $\mf A$-valued premeasures such that $\q{\mu}=\q{\nu}$. 
  \begin{enumerate}
  \item[(i)]  $\mu\dotleq\nu$ if and only if $\as[\mf A]{\mu}\dotleq\as[\mf A]{\nu}$.
  \item[(ii)] Suppose $\mu$ and $\nu$ are measures. If $\mu\dotleq\nu$ and $\nu\dotleq\mu$, then $\mu=\nu$.
  \end{enumerate}
\end{proposition} 
\begin{proof}
  (i) Let $\tilde\mu = \as[\mf A]{\mu}$ and $\tilde\nu = \as[\mf A]{\nu}$. If $\tilde\mu\dotleq\tilde\nu$, then we obviously have $\mu\dotleq\nu$. Conversely, let $\mu\dotleq\nu$ and $\mathcal K = D_\mu\cap D_\nu$. By~(\ref{Qmu}), we have $\q{\tilde\mu}=\q{\tilde\nu}=\q{\mu}$. Lemma~\ref{l3sigma} implies that $\sigma(\mathcal K) = \q{\mu}$ and, hence, $D_{\tilde\mu}\cap D_{\tilde\nu}\subset \sigma(\mathcal K)$. Since $\tilde\mu(A)\leq \tilde\nu(A)$ for every $A\in \mathcal K$, it follows from Proposition~\ref{l_dotleq0} that $\tilde\mu\dotleq\tilde\nu$.   
  \par\medskip\noindent
  (ii) If $\mu\dotleq\nu$ and $\nu\dotleq\mu$, then $\mu$ coincides with $\nu$ on $D_\mu\cap D_\nu$. Since $\sigma(D_\mu\cap D_\nu)=\q\mu$ by Lemma~\ref{l3sigma}, it follows from Theorem~\ref{t_unique} that $\mu=\nu$. 
\end{proof}

\begin{proposition}\label{l_dotleq1}
  Let $\mf A$ be an HOTM and $\mu$, $\nu$, and $\eta$ be $\mf A$-valued premeasures such that $\q\mu=\q\nu=\q\eta$. If $\mu\dotleq\nu$ and $\nu\dotleq\eta$, then $\mu\dotleq\eta$. If $\mu\dotleq\nu$, then $\mu\boxplus\eta\dotleq \nu\boxplus\eta$.
\end{proposition}
\begin{proof}
  Let $\mu\dotleq\nu$ and $\nu\dotleq\eta$ and let $\mathcal K = D_\mu\cap D_\nu\cap D_\eta$. By Lemma~\ref{l3sigma}, we have $\sigma(\mathcal K)=\q\mu\supset D_\mu\cap D_\eta$. Since $\mu(A)\leq\eta(A)$ for every $A\in \mathcal K$, we conclude that $\mu\dotleq\eta$ by Proposition~\ref{l_dotleq0}. If $\mu\dotleq\nu$, then $\mu\dotplus\eta\dotleq \nu\dotplus\eta$. As $\q{\mu\dotplus\eta}=\q{\nu\dotplus\eta}$ by Proposition~\ref{l_boxplus}(i), it follows from~(\ref{meas_sum}) and Proposition~\ref{l_dotleq}(i) that $\mu\boxplus\eta\dotleq \nu\boxplus\eta$.
\end{proof}

Let $\mathcal Q$ be a $\sigma$-ring and $\mf A$ be an HOTM. We endow the set $\ms M(\mathcal Q,\mf A)$ with the relation $\leq$ by requiring that $\mu\leq \nu$ holds for $\mu,\nu\in\ms M(\mathcal Q,\mf A)$ if and only if $\mu\dotleq \nu$.

\begin{proposition}\label{l_ordtab}
  Let $\mathcal Q$ be a $\sigma$-ring and $\mf A$ be an HOTM. Then $\ms M(\mathcal Q,\mf A)$ is an ordered monoid. 
\end{proposition}   
\begin{proof}
  Let $\ms M = \ms M(\mathcal Q,\mf A)$. By Proposition~\ref{t_semigrspace}, $\ms M$ is a monoid and Propositions~\ref{l_dotleq} and~\ref{l_dotleq1} imply that $\ms M$ is a partially ordered set. By Proposition~\ref{l_dotleq1}, we have $\mu+\eta\leq \nu+\eta$ for every $\mu,\nu,\eta\in\ms M$ such that $\mu\leq\nu$. Thus, $\ms M$ is an ordered monoid. 
\end{proof}

We call $\mf A$ an ordered vector space if $\mf A$ is both an ordered monoid and an $\R$-vector space and, moreover,  $k x\geq 0$ for all $x\in\mf A$ and $k\in\mathbb R$ such that $x\geq 0$ and $k\geq0$.  

\begin{proposition}\label{l_ordmult}
  Let $\mf A$ be both an HCLM$_\R$ and an ordered vector space. Let $\mu$ and $\nu$ be $\mf A$-valued premeasures such that $\mu\dotleq \nu$ and $\q\mu=\q\nu$. Then $k\boxdot\mu\dotleq k\boxdot\nu$ for every real $k\geq 0$.
\end{proposition}
\begin{proof}
  Let $k\geq 0$. Then $k\cdot\mu\dotleq k\cdot\nu$. Since $\q{k\cdot\mu}=\q{k\cdot\nu}$ by~(\ref{Qboxdot}), it follows from~(\ref{meas_mult}) and Proposition~\ref{l_dotleq}(i) that $k\boxdot\mu\dotleq k\boxdot\nu$.  
\end{proof}

\begin{theorem}\label{p_orderedspace}
  Let $\mf A$ be both an HTG and an ordered vector space. Suppose the set $\set{x}{x\in\mf A\mbox{ and }x\geq 0}$ is sequentially closed in $\mf A$ and $x\mapsto k x$ is a continuous map from $\mf A$ to itself for every $k\in \R$. Let $\mathcal Q$ be a $\sigma$-ring. Then $\ms M(\mathcal Q,\mf A)$ is an ordered vector space.   
\end{theorem}
\begin{proof}
  The set $\set{(x,y)}{(x,y)\in\mf A\times \mf A\mbox{ and } x\leq y}$ is sequentially closed in $\mf A\times\mf A$ because it is the inverse image of the set $\set{x}{x\in\mf A\mbox{ and }x\geq 0}$ under the continuous map $(x,y)\mapsto y-x$ from $\mf A\times\mf A$ to $\mf A$. Thus, $\mf A$ is an HOTM. Let $\ms M = \ms M(\mathcal Q,\mf A)$. By Proposition~\ref{l_ordtab}, $\ms M$ is an ordered monoid and Theorem~\ref{t_module} implies that $\ms M$ is an $\R$-vector space. Since $\mf A$ is an HCLM$_\R$, Proposition~\ref{l_ordmult} ensures that $k\mu\geq 0$ for every $\mu\in \ms M$ and $k\in\R$ such that $\mu\geq 0$ and $k\geq 0$. Thus, $\ms M$ is an ordered vector space.
\end{proof}

Let $\mathcal{K}$ be a $\delta$-ring and $\mf A$ be an HTM. We define
\[
  \ms M_0(\mathcal{K},\mf A) = \set{\mu}{\mu\mbox{ is an $\mf A$-valued premeasure and }D_\mu = \mathcal{K}}.
\]
We make $\ms M_0(\mathcal{K},\mf A)$ into a monoid by endowing it with the pointwise addition. If $\mf A$ is an HTG, then $\ms M_0(\mathcal{K},\mf A)$ is a group. If $\mf A$ is an HCLM$_\K$, then we make $\ms M_0(\mathcal{K},\mf A)$ into a left $\K$-monoid by setting $k\mu = k\cdot \mu$ for every $k\in\K$ and $\mu\in \ms M_0(\mathcal{K},\mf A)$. If $\mf A$ is both an HCLM$_\K$ and a left $\K$-module ($\K$-vector space), then $\ms M_0(\mathcal{K},\mf A)$ is a left $\K$-module (resp., $\K$-vector space). In particular, if $\mf A$ is a Hausdorff topological vector space over a topological field $\K$, then $\ms M_0(\mathcal{K},\mf A)$ is a $\K$-vector space. If $\mf A$ is an HOTM, then we make $\ms M_0(\mathcal{K},\mf A)$ into an ordered monoid by requiring that $\mu\leq \nu$ if and only if $\mu\dotleq \nu$ for every $\mu,\nu\in \ms M_0(\mathcal{K},\mf A)$.

\begin{proposition}\label{p_M0}
  Let $\mf A$ be an HTM, $\mathcal{K}$ be a $\delta$-ring, $\mathcal{Q} = \sigma(\mathcal{K})$, $\ms M_0 = \ms M_0(\mathcal{K},\mf A)$, and $\fm{f}{\ms M_0}$ be such that $f(\mu) = \as[\mf A]\mu$ for every $\mu\in\ms M_0$. Then $f$ is an injective homomorphism from $\ms M_0$ to $\ms M(\mathcal{Q},\mf A)$ and $\im f = \set{\mu}{\mu\in\ms M(\mathcal{Q},\mf A)\mbox { and }\mathcal{K}\subset D_\mu}$. If $\mf A$ is an HCLM$_\K$, then $f(k\mu) = k f(\mu)$ for every $k\in\K$ and $\mu\in\ms M_0$. If $\mf A$ is an HOTM, then $f(\mu)\leq f(\nu)$ if and only if $\mu\leq \nu$ for every $\mu,\nu\in\ms M_0$.
\end{proposition}
\begin{proof}
  Let $\ms M = \ms M(\mathcal{Q},\mf A)$. By~(\ref{Qmu}), we have $\q{f(\mu)} = \mathcal{Q}$ and, hence, $f(\mu)\in\ms M$ for every $\mu\in\ms M_0$. By~(\ref{ass_zero}), we have $f(0_{\ms M_0}) = 0_{\ms M}$, while~(\ref{meas_sum}) and Proposition~\ref{l_boxplus}(ii) imply that $f(\mu + \nu) = \mu\boxplus\nu$ for every $\mu,\nu\in \ms M_0$. Thus, $f$ is a homomorphism from $\ms M_0$ to $\ms M$. If $\mu,\nu\in\ms M_0$ are such that $f(\mu) = f(\nu)$, then $\mu = f(\mu)|_{\mathcal{K}} = f(\nu)|_{\mathcal{K}} =\nu$. Thus, $f$ is injective. Let $\mu\in\ms M$ be such that $\mathcal{K}\subset D_\mu$. Let $\nu = \mu|_{\mathcal{K}}$. Then $\nu\in \ms M_0$ and $\mu$ and $f(\nu)$ coincide on $\mathcal{K}$. By Theorem~\ref{t_unique}, it follows that $\mu = f(\nu)$. This proves that $\im f = \set{\mu}{\mu\in\ms M\mbox{ and }\mathcal{K}\subset D_\mu}$. If $\mf A$ is an HCLM$_\K$, then it follows from~(\ref{meas_mult}) and~(\ref{asboxd}) that $f(k\mu) = k\boxdot\mu = k\boxdot\as[\mf A]\mu = k f(\mu)$ for every $k\in\K$ and $\mu\in\ms M_0$. If $\mf A$ is an HOTM and $\mu,\nu\in \ms M_0$, then the relations $\mu\leq \nu$ and $f(\mu)\leq f(\nu)$ are equivalent by Proposition~\ref{l_dotleq}(i).
\end{proof}

Applying Proposition~\ref{p_M0} to $\R$ and $\C$ in place of $\mf A$, we immediately recover condition~\ref{item:4} for scalar measures formulated in Section~\ref{s_intro}.

\begin{corollary}
  Let $\mf A$ be an HTM and $\mathcal{Q}$ be a $\sigma$-ring. Then $\ms M_0(\mathcal{Q},\mf A)$ is a submonoid of $\ms M(\mathcal{Q},\mf A)$. If $\mf A$ is both an HCLM$_{\mathbb{K}}$ and a left $\mathbb K$-module, then $\ms M_0(\mathcal{Q},\mf A)$ is a left $\K$-submodule of $\ms M(\mathcal{Q},\mf A)$. If $\mf A$ is an HOTM, then the order on $\ms M_0(\mathcal{Q},\mf A)$ is induced by that on $\ms M(\mathcal{Q},\mf A)$.
\end{corollary}
\begin{proof}
  By Proposition~\ref{p_Dmusigma}, every element of $\ms M_0(\mathcal{Q},\mf A)$ is a measure and, hence, the homomorphism $f$ in Proposition~\ref{p_M0} is the identity map of $\ms M_0(\mathcal{Q},\mf A)$. The statement therefore follows from Proposition~\ref{p_M0}. 
\end{proof}

Let $\mf A$ be an HTM and $\mathcal{Q}$ be a $\sigma$-ring. We let $\Delta_{\mathcal{Q}}$ denote the set of all $\delta$-rings $\mathcal{K}$ such that $\sigma(\mathcal{K}) = \mathcal{Q}$, ordered by reverse inclusion.  By Lemma~\ref{l3sigma}, we have $\mathcal{K}\cap\mathcal{K}'\in \Delta_{\mathcal{Q}}$ for every $\mathcal{K},\mathcal{K}'\in \Delta_{\mathcal{Q}}$. Hence, $\Delta_{\mathcal{Q}}$ is a directed set. If $\mathcal{K},\mathcal{K}'\in \Delta_{\mathcal{Q}}$ and $\mathcal{K}\subset\mathcal{K}'$, then there is a natural homomorphism $\ms M_0(\mathcal{K}',\mf A)\to \ms M_0(\mathcal{K},\mf A)$ given by restriction to $\mathcal{K}$. The monoids $\ms M_0(\mathcal{K},\mf A)$ and the restriction maps constitute an inductive system over $\Delta_{\mathcal{Q}}$. Given $\mathcal{K}\in \Delta_{\mathcal{Q}}$, we let $\rho^{\mathcal{Q}}_{\mathcal{K}}$ denote the canonical map from $\ms M_0(\mathcal{K},\mf A)$ to $\varinjlim_{\mathcal{K}\in\Delta_{\mathcal{Q}}}\ms M_0(\mathcal{K},\mf A)$ (note that inductive limits always exist in the category of Abelian monoids; see~\cite[Sec.~I.10.3]{BourbakiAlgebre}).   

\begin{proposition}
  Let $\mf A$ be an HTM and $\mathcal{Q}$ be a $\sigma$-ring. There is a unique isomorphism $l\colon \varinjlim_{\mathcal{K}\in\Delta_{\mathcal{Q}}}\ms M_0(\mathcal{K},\mf A)\to \ms M(\mathcal{Q},\mf A)$ such that $l(\rho^{\mathcal{Q}}_{\mathcal{K}}(\mu)) = \as[\mf A]\mu$ for every $\mathcal{K}\in \Delta_{\mathcal{Q}}$ and $\mu\in\ms M_0(\mathcal{K},\mf A)$.
\end{proposition}
\begin{proof}
  For brevity, we set $\Delta = \Delta_{\mathcal{Q}}$, $\ms M = \ms M(\mathcal{Q},\mf A)$, and $\ms M' = \varinjlim_{\mathcal{K}\in\Delta_{\mathcal{Q}}}\ms M_0(\mathcal{K},\mf A)$. Let $\fm{\theta}{\Delta}$ be such that $\theta(\mathcal{K})$ is the map $\mu\mapsto \as[\mf A]\mu$ on $\ms M_0(\mathcal{K},\mf A)$ for every $\mathcal{K}\in \Delta$. By Proposition~\ref{p_M0}, $\theta(\mathcal{K})$ is an injective homomorphism from $\ms M_0(\mathcal{K},\mf A)$ to $\ms M(\mathcal{Q},\mf A)$ for every $\mathcal{K}\in \Delta$. Let $\mathcal{K},\mathcal{K}'\in\Delta$ be such that $\mathcal{K}\subset\mathcal{K}'$ and let $\varphi\colon \ms M_0(\mathcal{K}',\mf A)\to\ms M_0(\mathcal{K},\mf A)$ be the restriction map. If $\mu\in \ms M_0(\mathcal{K}',\mf A)$, then $\mu\asymp \varphi(\mu)$ and, hence, $\as[\mf A]\mu = \as[\mf A]{\varphi(\mu)}$ by Proposition~\ref{p_asssimilar}. This means that $\theta(\mathcal{K})\circ\varphi = \theta(\mathcal{K}')$. By the universal property of inductive limits, there is a unique homomorphism $l\colon \ms M'\to \ms M$ such that
  \begin{equation}
    \label{circ}
    l\circ\rho^{\mathcal{Q}}_{\mathcal{K}} = \theta(\mathcal{K}),\quad \mathcal{K}\in \Delta.
  \end{equation}
  Since the underlying set of an inductive limit of monoids is the inductive limit of the underlying sets (see~\cite[Sec.~I.10.3]{BourbakiAlgebre}), it follows from the injectivity of $\theta(\mathcal{K})$ and Proposition~6 in~\cite[Sec.~III.7.6]{BourbakiEnsembles} that $l$ is injective. Let $\mu\in \ms M$. Then $D_\mu\in \Delta$ and $\mu\in \ms M_0(D_\mu,\mf A)$. As $\as[\mf A]\mu = \mu$ by Proposition~\ref{p_measequiv}, it follows from~(\ref{circ}) that $\mu = l(t)$, where $t = \rho^{\mathcal{Q}}_{D_\mu}(\mu)$. This means that $l$ is surjective and, hence, is an isomorphism. 
\end{proof}

\section{Positive measures and infinities}
\label{s_pos}

We say that $\mu$ is a positive additive function ($\sigma$-additive function, content, $\sigma$-content, premeasure, measure) if $\mu$ is an $\R$-valued additive function (resp., $\sigma$-additive function, content,  $\sigma$-content, premeasure, measure) and $\mu(A)\geq 0$ for every $A\in D_\mu$. 

\begin{proposition}\label{p_PosSigma}
  Let $\mu$ be a positive premeasure. Then 
  \begin{multline}\label{pmSigma}
    \Sigma_\R(\mu) = \sett{A}{\mbox{there is a countable partition $\mb A$ of $A$ into}}{\mbox{elements of $D_\mu$  
      such that the family $\mu\circ \mb A$ is summable in $\R$}}. 
  \end{multline}
\end{proposition}
\begin{proof}
  Let $\mathcal K$ denote the set in the right-hand side of~(\ref{pmSigma}). Let $A\in \Sigma_\R(\mu)$ and $\nu = \as[\R]\mu$. By~(\ref{Qmu}), we have $A\in\q\mu$ and Lemma~\ref{ll1}(ii) implies that there is a partition $\mb A$ of $A$ into elements of $D_\mu$. Since $\mu\circ\mb A = \nu\circ\mb A$, it follows from the $\sigma$-additivity of $\nu$ that $\mu\circ \mb A$ is summable in $\R$ and, hence, $A\in\mathcal K$. Thus, $\Sigma_\R(\mu)\subset \mathcal K$. Let $A\in \mathcal K$ and $\mb A\colon I\to D_\mu$ be a countable partition of $A$ such that $\mu\circ\mb A$ is summable in $\R$. Let $\mb B\colon J\to D_\mu$ be another countable partition of $A$. Let $K = I\times J$ and $\fm{\mb C}{K}$ be such that $\mb C(i,j) = \mb A(i)\cap \mb B(j)$ for every $i\in I$ and $j\in J$. Since $\mu(\mb A(i)) = \sum_{j\in J}\mu(\mb C(i,j))$ for every $i\in I$, the positivity of $\mu$ implies that the family $\{\mu(\mb C(k))\}_{k\in K}$ is summable in $\R$. As $\mu(\mb B(j)) = \sum_{i \in I} \mu(\mb C(i,j))$ for every $j\in J$ by the $\sigma$-additivity of $\mu$, it follows from Proposition~\ref{p_fubiniRTS}(ii) that $\mu\circ\mb B$ is summable in $\R$. By Proposition~\ref{p_seqcomplSigma}, we conclude that $A \in \Sigma_\R(\mu)$ and, hence, $\mathcal K\subset \Sigma_\R(\mu)$.
\end{proof}

\begin{corollary}\label{cSumPosMeas}
  Let $\mu_1$ and $\mu_2$ be positive measures. Then $\mu_1 \dotplus\mu_2$ is a positive measure.
\end{corollary}
\begin{proof}
  Clearly, $\mu = \mu_1\dotplus\mu_2$ is a positive premeasure. Let $A\in \Sigma_\R(\mu)$. By Proposition~\ref{p_PosSigma}, there is a countable partition $\mb A\colon I\to D_\mu$ of $A$ such that $\mu\circ \mb A$ is summable in $\R$. Since $D_\mu = D_{\mu_1}\cap D_{\mu_2}$ and $\mu_1$ and $\mu_2$ are positive, we have $\mb A(i)\in D_{\mu_1}$ and $\mu_1(\mb A(i))\leq \mu(\mb A(i))$ for every $i\in I$. This implies that $\mu_1\circ \mb A$ is summable in $\R$ and, hence, $A\in \Sigma_\R(\mu_1)$ by Proposition~\ref{p_PosSigma}. As $\mu_1$ is a measure, Proposition~\ref{p_measequiv} implies that $A\in D_{\mu_1}$. In the same way, we obtain $A\in D_{\mu_2}$ and, therefore, $A\in D_\mu$. This means that $\Sigma_\R(\mu)\subset D_\mu$. By Proposition~\ref{p_measequiv}, we conclude that $\mu$ is an $\R$-valued measure and, hence, a positive measure. 
\end{proof}

\begin{corollary}\label{cIneqDomainPosMeas}
	Let $\mu$ be a positive measure and $\nu$ be an $\R$-valued premeasure such that $\q\mu = \q\nu$. If $\mu \dotleq \nu$, then $\nu$ is a positive premeasure and $D_\nu \subset D_\mu$.
\end{corollary}
\begin{proof}
  Let $\mathcal{Q} = \q\mu$ and $\eta = \mb 0_\R(\mathcal{Q})$. Since $\eta\dotleq\mu$, it follows from Proposition~\ref{l_dotleq1} that $\eta\dotleq\nu$ and, hence, $\nu$ is a positive premeasure. Let $A\in D_\nu$. By Lemma~\ref{l3sigma}, we have $\mathcal{Q} = \sigma(D_\mu\cap D_\nu)$. As $A\in\mathcal{Q}$, Lemma~\ref{ll1}(ii) implies that there is a countable partition $\fm{\mb A}{I}$ of $A$ into elements of $D_\mu\cap D_\nu$. Since $\nu\circ \mb A$ is summable and $\mu(\mb A(i))\leq \nu(\mb A(i))$ for every $i\in I$, we conclude that $\mu\circ\mb A$ is summable and, hence, $A\in \Sigma_\R(\mu)$ by Proposition~\ref{p_PosSigma}. As $\mu$ is a measure, it follows from Proposition~\ref{p_measequiv} that $A\in D_\mu$. This means that $D_\nu\subset D_\mu$.
\end{proof}

We let $\R_+$ denote the closed positive half-axis $[0,\infty)$ endowed with the addition, order, and topology induced from $\R$. Clearly, $\R_+$ is a regular HOTM and a closed topological submonoid of $\R$.  

Let $\infty$ be an element such that $\infty\notin \R$. We let $\mf R$ denote the set $\R_+\cup\{\infty\}$. The set $\mf R$ is endowed with the unique topology such that $\mf R$ becomes a compact Hausdorff space and $\R_+$ becomes its topological subspace. The addition on $\R_+$ is extended to $\mf R$ by setting $x + \infty = \infty + x = \infty$ for every $x\in\mf R$.  The order on $\R_+$ is extended to $\mf R$ by setting $x\leq\infty$ for every $x\in\mf R$. This makes $\mf R$ into a regular HOTM. Every subset of $\mf R$ has a supremum in $\mf R$.

Since $\R_+$ is a topological submonoid of both $\R$ and $\mf R$, it follows from Propositions~\ref{p_subsgradditive} and~\ref{p_subpremeasure} that
\begin{multline}
  \label{eq:posR+premeas}
  \mbox{$\mu$ is a positive additive ($\sigma$-additive) function} \\ \mbox{$\Leftrightarrow$ $\mu$ is an $\R_+$-valued additive (resp., $\sigma$-additive) function} \\\mbox{$\Leftrightarrow$ $\mu$ is an $\mf R$-valued additive (resp., $\sigma$-additive) function and $\im \mu\subset \R_+$}. 
\end{multline}

As $\R_+$ is closed in $\R$, Corollary~\ref{l_subgr1} implies that
\begin{equation}
  \label{eq:posR+meas}
  \mbox{$\mu$ is a positive measure $\Leftrightarrow$ $\mu$ is an $\R_+$-valued measure}.
\end{equation}

\begin{lemma}\label{l_Rsummable}
  Let $f$ be an $\mf R$-valued map. Then $f$ is summable in $\mf R$.
\end{lemma}
\begin{proof}
  Let $\theta = \bigS^{\mf R} f$ and $a$ be the supremum of $\im \theta$ in $\mf R$. It suffices to show that the net $\theta$ converges to $a$ in $\mf R$. If $a = 0$, then $\finsum_{k\in K} f(k) = 0$ for every $K\in\mathcal F[f]$ and, hence, our claim is true. Let $a>0$ and $U$ be a neighbourhood of $a$ in $\mf R$. Then there is $b \in\mf R$ such that $b < a$ and $(b,a]\subset U$. Let $K\in \mathcal F[f]$ be such that $b < \finsum_{k\in K} f(k)$. Let $K'\in \mathcal F[f]$ be such that $K\subset K'$ and let $s = \finsum_{k\in K'}f(k)$. Then $s\geq \finsum_{k\in K}f(k)$ and, hence, $b < s$. Since $s\leq a$, we have $s\in U$. Thus, $\theta$ converges to $a$ in $\mf R$. 
\end{proof}

\begin{proposition}\label{p_Rassdom}
  Let $\eta$ be an $\mf R$-valued premeasure. Then $\Sigma_{\mf R}(\eta) = \q\eta$.
\end{proposition}
\begin{proof}
  It follows from~(\ref{SigmaA}) and Lemma~\ref{l_Rsummable} that $\mathcal L_{\mf R}(\eta) = \q\eta$. The statement therefore follows from Theorem~\ref{ll16}.  
\end{proof}

\begin{corollary}\label{c_Rmeasdom}
  Let $\eta$ be an $\mf R$-valued measure. Then $D_\eta$ is a $\sigma$-ring.
\end{corollary}
\begin{proof}
  Propositions~\ref{p_measequiv} and~\ref{p_Rassdom} imply that $\q\eta\subset D_\eta$ and, hence, $D_\eta = \q\eta$.
\end{proof}

Corollary~\ref{c_Rmeasdom} implies that $\ms M(\mathcal{Q},\mf R) = \ms M_0(\mathcal{Q},\mf R)$ for every $\sigma$-ring $\mathcal{Q}$.

\begin{proposition}\label{p_SumRMeas}
  Let $\eta_1$ and $\eta_2$ be $\mf R$-valued measures. Then so is $\eta_1 \dotplus\eta_2$.
\end{proposition}
\begin{proof}
  Let $\eta = \eta_1\dotplus\eta_2$. Since $D_\eta = D_{\eta_1}\cap D_{\eta_2}$, Corollary~\ref{c_Rmeasdom} implies that $D_\eta$ is a $\sigma$-ring. The statement therefore follows from Proposition~\ref{p_Dmusigma}.
\end{proof}

Given a map $\eta$, we put $D^f_\eta = \set{A}{A\in D_\eta\mbox{ and }\eta(A)\in\R}$. The restriction of $\eta$ to $D^f_\eta$ is called the finite part of $\eta$ and is denoted by $\mf F\eta$.

\begin{lemma}\label{l_Dmuf}
  Let $\eta$ be an $\mf R$-valued content. Let $A\in D^f_\eta$ and $B\in D_\eta$ be such that $B\subset A$. Then $B\in D^f_\eta$.
\end{lemma}
\begin{proof}
  Since $\eta(A) = \eta(B) + \eta(A\setminus B)$, we have $\eta(B)\leq\eta(A)<\infty$.
\end{proof}

\begin{proposition}\label{p_Fmupremeas}
  Let $\eta$ be an $\mf R$-valued additive function ($\sigma$-additive function, content, $\sigma$-content, premeasure). Then $\mf F\eta$ is a positive additive function (resp., $\sigma$-additive function, content, $\sigma$-content, premeasure).
\end{proposition}
\begin{proof}
  Let $\mu = \mf F\eta$. If $\eta$ is an $\mf R$-valued additive ($\sigma$-additive) function, then $\mu$ is obviously an $\mf R$-valued additive (resp., $\sigma$-additive) function and it follows from~(\ref{eq:posR+premeas}) that $\mu$ is a positive additive (resp., $\sigma$-additive) function. Let $\eta$ be an $\mf R$-valued content and $\mathcal{Q} = D^f_\eta$. Let $A,B\in \mathcal{Q}$. By Lemma~\ref{l_Dmuf}, we have $A\setminus B\in \mathcal{Q}$. Since $\eta(A\cup B) = \eta(A\setminus B) + \eta(B)$, it follows that $\eta(A\cup B) < \infty$, i.e., $A\cup B\in \mathcal{Q}$. As $\eta(\varnothing) = 0$, we have $\varnothing\in\mathcal{Q}$. This means that $\mathcal{Q}$ is a ring and, hence, $\mu$ is a positive content. Let $\eta$ be an $\mf R$-valued $\sigma$-content. Then $\mu$ is both a positive $\sigma$-additive function and a positive content and, therefore, is a positive $\sigma$-content. Let $\eta$ be an $\mf R$-valued premeasure. Then $\mathcal{Q}$ is closed under countable intersections by Lemma~\ref{l_Dmuf}. As $\mu$ is a positive $\sigma$-content, it follows that $\mu$ is a positive premeasure.
\end{proof}

\begin{proposition}\label{p_Fmu}
  Let $\eta$ be an $\mf R$-valued measure. Then $\mf F\eta$ is a positive measure.
\end{proposition}
\begin{proof}
  Let $\mu = \mf F\eta$. By Proposition~\ref{p_Fmupremeas}, $\mu$ is a positive premeasure. Let $A\in \Sigma_{\R}(\mu)$. By Proposition~\ref{p_PosSigma}, there is a countable partition $\mb A\colon I\to D_\mu$ of $A$ such that $\mu\circ \mb A$ is summable in $\R$. As $\R_+$ is closed in $\R$, Proposition~\ref{p_seqclosedsum} implies that $\mu\circ \mb A$ is summable in $\R_+$. Since $\eta\circ\mb A = \mu\circ\mb A$, it follows from Proposition~\ref{p_sumsubsgr} that $\sum^{\mf R}\eta\circ\mb A = \sum^{\R_+}\mu\circ\mb A$ and, hence, $\sum^{\mf R}\eta\circ\mb A$ belongs to $\R_+$. As $D_\mu\subset D_\eta$ and $D_\eta$ is a $\sigma$-ring by Corollary~\ref{c_Rmeasdom}, we have $\q\mu\subset D_\eta$. Hence,  $A\in D_\eta$, and the $\sigma$-additivity of $\eta$ implies that $\eta(A)\in\R_+$, i.e., $A\in D_\mu$. Thus, $\Sigma_{\R}(\mu)\subset D_\mu$. By Proposition~\ref{p_measequiv}, we conclude that $\mu$ is an $\R$-valued measure and, hence, is a positive measure.
\end{proof}

\begin{definition}\label{d_barsf}
  A map $\eta$ is called $\sigma$-finite if every element of $D_\eta$ is a subset of some element of $\sigma_0(D^f_\eta)$. 
\end{definition}

\begin{proposition}\label{p_barsigma}
  Let $\eta$ be an $\mf R$-valued content and $\mu = \mf F\eta$. Then
  \[
    \eta \mbox{ is $\sigma$-finite}\Leftrightarrow D_\eta\subset\sigma_0(D_\mu)\Leftrightarrow \q{\mu} = \q\eta.
  \]
\end{proposition} 
\begin{proof}
  Suppose $\eta$ is $\sigma$-finite.  Let $A\in D_\eta$. Then $A\subset \bigcup_{i\in I} \mb A(i)$ for some countable family $\mb A\colon I\to D^f_\eta$. By Lemma~\ref{l_Dmuf}, we have $A\cap\mb A(i)\in D^f_\eta$ for every $i\in I$. Since $A = \bigcup_{i\in I}(A\cap\mb A(i))$, we conclude that $A\in \sigma_0(D_\mu)$. Thus, $D_\eta\subset\sigma_0(D_\mu)$. Suppose $D_\eta\subset\sigma_0(D_\mu)$. Since $\sigma_0(D_\mu)\subset\q\mu$, we have $D_\eta\subset\q\mu\subset\q\eta$ and, hence, $\q\mu = \q\eta$. Finally, if $\q{\mu} = \q\eta$, then $\eta$ is $\sigma$-finite by Lemma~\ref{l_sigma0}. 
\end{proof}

\begin{theorem}\label{t_Fass}
  {$ $}
  \begin{enumerate}
  \item [(i)] Let $\mu$ be a positive measure and $\eta = \as[\mf R]\mu$. Then $\eta$ is a $\sigma$-finite $\mf R$-valued measure, $\eta$ is an extension of $\mu$, $D_\eta = \q\eta = \q\mu$, and $\mu = \mf F\eta$.
  \item [(ii)] Let $\eta$ be a $\sigma$-finite $\mf R$-valued measure and $\mu = \mf F\eta$. Then $\mu$ is a positive measure and $\eta = \as[\mf R]\mu$.
  \end{enumerate}
\end{theorem}
\begin{proof}
  (i) By~(\ref{eq:posR+premeas}), $\mu$ is an $\mf R$-valued premeasure and, therefore, $\eta$ is an $\mf R$-valued measure and an extension of $\mu$. Let $\nu = \mf F\eta$. Then $\nu$ is an extension of $\mu$. This implies that $D_\mu\subset D_\nu\subset D_\eta$. By~(\ref{Qmu}), it follows that
  \begin{equation}
    \label{eq:qDq}
    \q\mu = \q\nu = \q\eta
  \end{equation}
  and, therefore, $\mu\asymp\nu$.  Since $\nu$ is a positive premeasure by Proposition~\ref{p_Fmupremeas}, Definition~\ref{dd6} ensures that $\mu$ is an extension of $\nu$ and, hence, $\mu = \nu$. By Proposition~\ref{p_barsigma} and the right equality in~(\ref{eq:qDq}), we conclude that $\eta$ is $\sigma$-finite. Finally, it follows from Proposition~\ref{p_Rassdom} that $D_\eta = \q\mu$.
  \par\medskip\noindent
  (ii) By Proposition~\ref{p_Fmu}, $\mu$ is a positive measure and Proposition~\ref{p_barsigma} implies that $\q\mu = \q\eta$. As $\eta$ is an extension of $\mu$, we have $\mu\asymp\eta$. Since $\mu$ is an $\mf R$-valued premeasure by~(\ref{eq:posR+premeas}), it follows from Proposition~\ref{p_ass} that $\eta = \as[\mf R]\mu$.   
\end{proof}

\begin{proposition}\label{p_asdotplus}
  Let $\mu_1$ and $\mu_2$ be positive premeasures such that $\q{\mu_1} = \q{\mu_2}$, $\eta_1 = \as[\mf R]{\mu_1}$, and $\eta_2 = \as[\mf R]{\mu_2}$.
  \begin{enumerate}
  \item [(i)] $\as[\mf R]{\mu_1\dotplus_\R \mu_2} = \eta_1 \dotplus_{\mf R} \eta_2$. 
  \item [(ii)] $\mu_1 \dotleq_\R \mu_2$ if and only if  $\eta_1\dotleq_{\mf R}\eta_2$.
  \end{enumerate}
\end{proposition}
\begin{proof}
  By~(\ref{eq:posR+premeas}), $\mu_1$ and $\mu_2$ are $\mf R$-valued premeasures and, hence, $\eta_1$ and $\eta_2$ are $\mf R$-valued measures. 
  \par\smallskip\noindent
  (i)  By Lemma~\ref{l3sigma} and Proposition~\ref{p_Rassdom}, we have
  \begin{equation}
    \label{eq:Dmu'}
    D_{\eta_1} = D_{\eta_2} = \sigma(D_{\mu_1}\cap D_{\mu_2}) = \q{\mu_1}.
  \end{equation}
  Let $\mu = \mu_1\dotplus_\R \mu_2$ and $\eta = \eta_1\dotplus_{\mf R} \eta_2$. By Proposition~\ref{p_SumRMeas}, $\eta$ is an $\mf R$-valued measure. By~(\ref{eq:Dmu'}), we have $\q\mu = D_\eta$ and, hence, $\q\mu = \q\eta$. Clearly, $\mu = \mu_1\dotplus_{\mf R} \mu_2$ and, therefore, $\mu$ is an $\mf R$-valued premeasure. Since $\eta_1$ and $\eta_2$ are extensions of $\mu_1$ and $\mu_2$ respectively, $\eta$ is an extension of $\mu$. It follows that $\mu\asymp\eta$ and, hence, $\eta = \as[\mf R]\mu$ by Proposition~\ref{p_ass}.
  \par\medskip\noindent
  (ii) We have  $\mu_1 \dotleq_\R \mu_2$ if and only if $\mu_1 \dotleq_{\mf R} \mu_2$. The statement therefore follows from Proposition~\ref{l_dotleq}(i).
\end{proof}

\begin{proposition}\label{p_Fdotplus}
  Let $\eta_1$ and $\eta_2$ be $\mf R$-valued maps, $\mu_1 = \mf F\eta_1$, and $\mu_2 = \mf F\eta_2$.
  \begin{enumerate}
  \item [(i)] $\mf F(\eta_1\dotplus_{\mf R}\eta_2) = \mu_1\dotplus_\R\mu_2$.
  \item [(ii)] If $\eta_1\dotleq_{\mf R}\eta_2$, then $\mu_1\dotleq_\R \mu_2$.
  \end{enumerate}
\end{proposition}
\begin{proof}
  (i)   Let $\eta = \eta_1\dotplus_{\mf R}\eta_2$ and $\mu = \mu_1\dotplus_\R \mu_2$. Then $\mu = \mu_1\dotplus_{\mf R}\mu_2$ and, hence, $\eta$ is an extension of $\mu$. This implies that $\mf F\eta$ is an extension of $\mu$. Let $A\in D^f_\eta$. Since $\eta_1(A) + \eta_2(A) = \eta(A) < \infty$, we have $\eta_1(A)<\infty$ and $\eta_2(A)<\infty$. It follows that $A\in D_\mu$. Thus, $D^f_\eta\subset D_\mu$ and, therefore, $\mf F\eta = \mu$.
  \par\smallskip\noindent
  (ii) Let $\eta_1\dotleq_{\mf R}\eta_2$ and $A\in D_{\mu_1}\cap D_{\mu_2}$. Then $\eta_1(A)\leq_{\mf R} \eta_2(A)$. As $\mu_1(A) = \eta_1(A)$, $\mu_2(A) = \eta_2(A)$, and $\mu_1(A),\mu_2(A)\in\R$, we have $\mu_1(A)\leq_\R \mu_2(A)$, i.e., $\mu_1\dotleq_\R \mu_2$. 
\end{proof}

\begin{corollary}\label{c_sumsigma}
  Let $\eta_1$ and $\eta_2$ be $\sigma$-finite $\mf R$-valued premeasures such that $\q{\eta_1} = \q{\eta_2}$. Then $\eta_1\dotplus\eta_2$ is $\sigma$-finite.
\end{corollary}
\begin{proof}
  Let $\eta = \eta_1\dotplus\eta_2$. By Proposition~\ref{l_boxplus}(i), we have $\q\eta = \q{\eta_1} = \q{\eta_2}$ and Proposition~\ref{p_barsigma} implies that $\sigma(D^f_{\eta_1}) = \sigma(D^f_{\eta_2}) = \q\eta$. It follows from Proposition~\ref{p_Fdotplus}(i) that $D^f_\eta = D^f_{\eta_1}\cap D^f_{\eta_2}$ and, hence, $\sigma(D^f_\eta) = \q\eta$ by Lemma~\ref{l3sigma}. By Proposition~\ref{p_barsigma}, we conclude that $\eta$ is $\sigma$-finite. 
\end{proof}

Let $\mathcal Q$ be a $\sigma$-ring. It follows from Proposition~\ref{l_ordtab} that $\ms M(\mathcal Q,\R_+)$ and $\ms M(\mathcal Q,\mf R)$ are ordered monoids and Theorem~\ref{p_orderedspace} implies that $\ms M(\mathcal Q,\R)$ is an ordered vector space. Since $\R_+$ is closed in $\R$, Proposition~\ref{t_semigrsubspace} ensures that $\ms M(\mathcal Q,\R_+)$ is a submonoid of $\ms M(\mathcal Q,\R)$ and it follows from~(\ref{eq:posR+meas}) that
\[
  \ms M(\mathcal Q,\R_+) = \set{\mu}{\mu\in \ms M(\mathcal Q,\R)\mbox{ and }\im\mu\subset\R_+}.
\]
The order on $\ms M(\mathcal Q,\R_+)$ coincides with the order induced from $\ms M(\mathcal Q,\R)$. By Corollary~\ref{cSumPosMeas}, the sum of $\mu_1$ and $\mu_2$ in $\ms M(\mathcal Q,\R_+)$ is equal to $\mu_1\dotplus_\R\mu_2$ for every $\mu_1,\mu_2\in \ms M(\mathcal Q,\R_+)$. Corollary~\ref{c_Rmeasdom} implies that $D_\eta = \mathcal Q$ for every $\eta\in \ms M(\mathcal Q,\mf R)$. By Proposition~\ref{p_SumRMeas}, the sum of $\eta_1$ and $\eta_2$ in $\ms M(\mathcal Q,\mf R)$ is equal to $\eta_1\dotplus_{\mf R}\eta_2$ for every $\eta_1,\eta_2\in \ms M(\mathcal Q,\mf R)$.

Given a $\sigma$-ring $\mathcal Q$, we define
\[
  \ms M_\sigma(\mathcal Q,\mf R) =\set{\eta}{\eta\in \ms M(\mathcal Q,\mf R)\mbox{ and }\eta\mbox{ is $\sigma$-finite}}.
\]
By Corollary~\ref{c_sumsigma}, $\ms M_\sigma(\mathcal Q,\mf R)$ is a submonoid of $\ms M(\mathcal Q,\mf R)$. We endow $\ms M_\sigma(\mathcal Q,\mf R)$ with the order induced from $\ms M(\mathcal Q,\mf R)$. This makes $\ms M_\sigma(\mathcal Q,\mf R)$ into an ordered monoid.

The next theorem implies, in particular, condition~\ref{item:3} for scalar measures formulated in Section~\ref{s_intro}.

\begin{theorem}\label{t_infinite}
  Let $\mathcal Q$ be a $\sigma$-ring and $\fm{f}{\ms M_\sigma(\mathcal Q,\mf R)}$ be such that $f(\eta) = \mf F\eta$ for every $\eta\in \ms M_\sigma(\mathcal Q,\mf R)$. Then $f$ is an isomorphism of ordered monoids from $\ms M_\sigma(\mathcal Q,\mf R)$ to $\ms M(\mathcal Q,\R_+)$. We have $f^{-1}(\mu) = \as[\mf R]\mu$ for every $\mu\in \ms M(\mathcal Q,\R_+)$.
\end{theorem}
\begin{proof}
  Let $\ms M = \ms M_\sigma(\mathcal Q,\mf R)$ and $\ms M' = \ms M(\mathcal Q,\R_+)$. Propositions~\ref{p_Fmu} and~\ref{p_barsigma} imply that $f$ is a map from $\ms M$ to $\ms M'$. Let $\eta_1,\eta_2\in \ms M$. It follows from Proposition~\ref{p_Fdotplus}(i) that
  \[
    f(\eta_1 +_{\ms M}\eta_2) = \mf F(\eta_1\dotplus_{\mf R}\eta_2) = \mf F\eta_1\dotplus_\R\mf F\eta_2 = f(\eta_1) +_{\ms M'}f(\eta_2).
  \]
  As $0_{\ms M} = \mathbf 0_{\mf R}(\mathcal Q) = \mathbf 0_{\R}(\mathcal Q) = 0_{\ms M'}$, we have $f(0_{\ms M}) = 0_{\ms M'}$. Thus, $f$ is a homomorphism from $\ms M$ to $\ms M'$. If $\eta_1,\eta_2\in\ms M$ are such that $f(\eta_1) = f(\eta_2)$, then we have $\eta_1 = \as[\mf R]{f(\eta_1)} = \as[\mf R]{f(\eta_2)} = \eta_2$ by Theorem~\ref{t_Fass}(ii). Hence, $f$ is injective. Let $\mu\in \ms M'$. By Theorem~\ref{t_Fass}(i), we have $\as[\mf R]\mu\in\ms M$ and $\mu = f(\as[\mf R]\mu)$. This means that $f$ is an isomorphism from $\ms M$ to $\ms M'$ and $f^{-1}(\mu) = \as[\mf R]\mu$ for every $\mu\in \ms M'$. By Proposition~\ref{p_asdotplus}(ii), we have $\mu_1\leq_{\ms M'}\mu_2$ if and only if $f^{-1}(\mu_1)\leq_{\ms M} f^{-1}(\mu_2)$ for every $\mu_1,\mu_2\in\ms M'$. This implies that $f(\eta_1)\leq_{\ms M'} f(\eta_2)$ if and only if $\eta_1\leq_{\ms M}\eta_2$ for every $\eta_1,\eta_2\in\ms M$.     
\end{proof}

The same argument, with obvious modifications, shows that the $\sigma$-finite submonoids of $\ms M_0(\mathcal Q,\mf R_\uparrow)$ and $\ms M_0(\mathcal Q,\mf R_\downarrow)$ embed in $\ms M(\mathcal Q,\R)$ via the map $\mu\mapsto\mf F\mu$, as stated in Section~\ref{s_intro}. We omit the details here. 

\section{Uniqueness}
\label{s_unique}

A family $\fm{\mb A}{\N}$ is called an increasing (decreasing) sequence of sets if $\mb A(n)\subset\mb A(n+1)$ (resp., $\mb A(n)\supset\mb A(n+1)$) for every $n\in\N$.

\begin{definition}\label{d_alpha}
  We say that a nonempty set $\mathcal Q$ is an $\alpha$-class if it satisfies the following conditions:
  \begin{enumerate}
  \item[(1)] If $A,B\in \mathcal Q$ and $A\cap B=\varnothing$, then $A\cup B\in \mathcal Q$.
  \item[(2)] If $A,B\in \mathcal Q$ and $B\subset A$, then $A\setminus B\in \mathcal Q$.
  \item[(3)] If $\mb A\colon\N\to\mathcal{Q}$ is a decreasing sequence of sets, then $\bigcap_{n\in\N} \mb A(n)\in \mathcal Q$.
  \end{enumerate}
\end{definition}

\begin{lemma}\label{l_alpha}
  Let $\mathcal K$ be a $\cap$-closed set and $\mathcal Q$ be an $\alpha$-class such that $\mathcal K\subset\mathcal Q\subset \sigma(\mathcal K)$. Then $\mathcal Q$ is a $\delta$-ring.
\end{lemma}
\begin{proof}
  See~\cite[Lemma~9.6]{Smirnov2015}.
\end{proof}

\begin{remark}
  Suppose condition~(3) in Definition~\ref{d_alpha} is replaced with the condition
  \begin{enumerate}
  \item [(3$'$)] If $\mb A\colon\N\to\mathcal{Q}$ is an increasing sequence of sets, then $\bigcup_{n\in\N} \mb A(n)\in \mathcal Q$.
  \end{enumerate}
  If, in addition, $\mathcal{Q}$ has a greatest element with respect to inclusion, then we obtain the definition of a $\lambda$-system introduced by Dynkin~\cite[Sec.~1.1]{Dynkin1959}. He proved~\cite[Lemma~1.1]{Dynkin1959} that $\sigma(\mathcal{K})\subset\mathcal{Q}$ whenever $\mathcal{Q}$ is a $\lambda$-system and $\mathcal{K}$ is a $\cap$-closed subset of $\mathcal{Q}$.  This result follows at once from Lemma~\ref{l_alpha}, which can be viewed as its generalization. Indeed, every $\lambda$-system is an $\alpha$-class and, hence, $\mathcal{R} = \mathcal{Q}\cap\sigma(\mathcal{K})$ is a $\delta$-ring by Lemma~\ref{l_alpha}. Moreover, $\mathcal R$ is closed under unions of increasing sequences and, therefore, is a $\sigma$-ring. Thus, $\mathcal{R} = \sigma(\mathcal{K})$.
\end{remark}

\begin{definition}\label{dd2}
  A set $\mathcal Q$ is called a semi-monotone class if the following conditions are satisfied:
  \begin{enumerate}
  \item Let $\mb A\colon\N\to\mathcal{Q}$ be an increasing sequence of sets and $A = \bigcup_{n\in\N} \mb A(n)$. If $A$ is a subset of some element of $\mathcal Q$, then $A\in \mathcal Q$. 
  \item If $\mb A\colon\N\to\mathcal{Q}$ is a decreasing sequence of sets, then  $\bigcap_{n\in \N} \mb A(n)\in\mathcal Q$.
  \end{enumerate}
\end{definition}

The intersection of a nonempty set of semi-monotone classes is a semi-monotone class. Hence, for every set $\mathcal K$, there is a unique semi-monotone class $\mathcal Q\supset \mathcal K$ such that $\mathcal Q\subset \mathcal Q'$ for every semi-monotone class $\mathcal Q'\supset \mathcal K$. This $\mathcal Q$ is called the semi-monotone class generated by $\mathcal K$. We note that $\varnothing$ is a semi-monotone class and every $\alpha$-class (in particular, every $\delta$-ring) is a semi-monotone class.

\begin{lemma}\label{ll4}
  Let $\mathcal K$ be a ring. Then the semi-monotone class generated by $\mathcal K$ coincides with $\delta(\mathcal K)$.
\end{lemma}
\begin{proof}
  Let $\mathcal Q$ be the semi-monotone class generated by $\mathcal K$ and put $\mathcal R=\delta(\mathcal K)$. Since $\mathcal R$ is semi-monotone, we have $\mathcal Q\subset\mathcal R$. We note that every element of $\mathcal R$ is a subset of some element of $\mathcal K$. Indeed, the set of all $A\in\mathcal R$ having this property is a $\delta$-ring of which $\mathcal K$ is a subset. It follows that $\mathcal Q$ is monotone with respect to $\mathcal R$ in the sense of Dinculeanu~\cite[Sec.~I.1.5, Definition~7]{Dinculeanu}. Hence, by Proposition~16 in Sec.~I.1.5 of~\cite{Dinculeanu}, the monotone class with respect to $\mathcal R$ generated by $\mathcal K$ is $\mathcal R$ and, therefore, $\mathcal R\subset\mathcal Q$. Thus, $\mathcal Q=\mathcal R$.
\end{proof}

\begin{remark}
  The well-known monotone class theorem (see, e.g., Theorem~B in Sec.~I.6 of~\cite{Halmos}) can be easily derived from Lemma~\ref{ll4}, which can be viewed as its generalization. Indeed, let $\mathcal{K}$ be a ring and let $\mathcal{Q}$ and $\mathcal{R}$ be, respectively, the semi-monotone class and the monotone class generated by $\mathcal{K}$. Then $\mathcal{Q}$ is a $\delta$-ring by Lemma~\ref{ll4} and $\mathcal{Q}\subset \mathcal{R}\subset\sigma(\mathcal{K})$. Since $\sigma(\mathcal{Q}) = \sigma(\mathcal{K})$, it follows from Lemma~\ref{ll1}(ii) and the definition of a monotone class that $\mathcal{R} = \sigma(\mathcal{K})$.
\end{remark}

A set $\mathcal{Q}$ is said to be closed under finite disjoint unions if $\bigcup_{i\in I}\mb A(i)\in \mathcal{Q}$ for every finite disjoint family $\mb A\colon I\to\mathcal{Q}$. It follows by induction that $\mathcal{Q}$ is closed under finite disjoint unions if and only if $\varnothing\in \mathcal{Q}$ and $A\cup B\in\mathcal{Q}$ for every disjoint $A,B\in \mathcal{Q}$. Every $\alpha$-class is closed under finite disjoint unions.

\begin{lemma}\label{l_semimonotone}
  Let $\mathcal{K}$ be a semi-ring and $\mathcal{Q}$ be a semi-monotone class such that $\mathcal{K}\subset\mathcal{Q}\subset\sigma(\mathcal{K})$. Let $\mathcal{Q}$ be closed under finite disjoint unions. Then $\mathcal{Q}$ is a $\delta$-ring.
\end{lemma}
\begin{proof}
  It suffices to show that $\mathcal Q$ is a ring. Since $\mathcal{Q}$ is closed under finite disjoint unions, we only need to verify that $A\setminus B\in \mathcal{Q}$ for every $A,B\in \mathcal{Q}$. Let $\mathcal{R} = \delta(\mathcal{K})$. As $\kappa(\mathcal{K})\subset\mathcal{Q}$ by Proposition~\ref{p_semiring}, it follows from Lemma~\ref{ll4} that $\mathcal{R}\subset\mathcal{Q}$. Let $A,B\in \mathcal{Q}$. By Lemma~\ref{ll1}(ii), there is an increasing sequence of sets $\mb A\colon \N\to \mathcal{R}$ such that $A = \bigcup_{n\in\N}\mb A(n)$. Let $\fm{\mb C}{\N}$ be such that $\mb C(n) = \mb A(n)\setminus B$ for every $n\in\N$. Then $\mb C(n)\in \mathcal{R}$ for every $n\in\N$ by Lemma~\ref{ll1}(iii) and $A\setminus B = \bigcup_{n\in\N}\mb C(n)$. Since $\mb C$ is increasing and $A\setminus B\subset A$, we conclude that $A\setminus B\in \mathcal{Q}$. 
\end{proof}

If $\mathcal{K}$ is a semi-ring, then Lemma~\ref{l_alpha} obviously follows from Lemma~\ref{l_semimonotone}.

 \begin{lemma}\label{ll7}
   Let $\mf A$ be an HTM and $\mu$ be an $\mf A$-valued $\sigma$-content.
  \begin{enumerate}             
  \item[(i)] If $\mb A\colon \N\to D_\mu$ is an increasing sequence of sets such that the set $A=\bigcup_{n\in\N}\mb A(n)$ belongs to $D_\mu$, then $\mu(\mb A(n))\to \mu(A)$ in $\mf A$ as $n\to \infty$. 
  \item[(ii)] Suppose $\mf A$ is an HTG. If $\mb B\colon\N\to D_\mu$ is a decreasing sequence of sets such that $B=\bigcap_{n\in\N} \mb B(n)$ belongs to $D_\mu$, then $\mu(\mb B(n))\to \mu(B)$ in $\mf A$ as $n\to \infty$.
  \end{enumerate}
\end{lemma}                     
\begin{proof}  
  (i) Let $\fm{\mb C}{\N}$ be such that $\mb C(1)= \mb A(1)$ and $\mb C(n) = \mb A(n)\setminus \mb A(n-1)$ for $n>1$. Since $D_\mu$ is a ring, $\mb C$ is a countable partition of $A$ into elements of $D_\mu$. By the $\sigma$-additivity of $\mu$, we have $\mu(A) = \sum_{n\in\N}\mu(\mb C(n))$. It remains to note that $\mb A(n) = \bigcup_{l=1}^n \mb C(l)$ and, hence, $\mu(\mb A(n)) = \finsum_{l=1}^n \mu(\mb C(l))$ for every $n\in\N$.
  \par\medskip\noindent (ii) Let $A = \mb B(1)\setminus B$ and $\fm{\mb A}{\N}$ be such that $\mb A(n) = \mb B(1)\setminus \mb B(n)$ for every $n\in\N$. Then $A\in D_\mu$ and $\mb A(n)\in D_\mu$ for all $n\in\N$. Since $A = \bigcup_{n\in\N}\mb A(n)$, (i) implies that $\mu(\mb A(n))\to \mu(A)$ in $\mf A$ as $n\to \infty$. As $\mu(B) = \mu(\mb B(1))-\mu(A)$ and $\mu(\mb B(n)) = \mu(\mb B(1))-\mu(\mb A(n))$ for every $n\in\N$, the statement follows. 	
\end{proof}

The next example shows that the assumption that $\mf A$ be an HTG cannot be dropped in Lemma~\ref{ll7}(ii).

\begin{example}
  Let $\mathcal{Q}$ be the set of all subsets of $\N$ and $\mu\colon \mathcal{Q}\to\mf R$ be such that $\mu(A) = \infty$ for infinite $A\in \mathcal{Q}$ and $\mu(A)$ is equal to the number of elements of $A$ for finite $A\in\mathcal{Q}$. Then $\mu$ is an $\mf R$-valued measure. Let $\mb A\colon\N\to\mathcal{Q}$ be such that $\mb A(n) = \set{k}{k\in \N\mbox{ and }k\geq n}$ for every $n\in\N$. Then $\mb A$ is a decreasing sequence of sets and $\mu(\mb A(n)) = \infty$ for every $n\in\N$. At the same time, we have $\bigcap_{n\in\N}\mb A(n) = \varnothing$.  
\end{example}

\begin{lemma}\label{l_semiprem}
  Let $\mf A$ be an HTG, $X$ be a sequentially closed subset of $\mf A$, and $\mu$ be an $\mf A$-valued premeasure. Then $\mu^{-1}(X)$  is a semi-monotone class.
\end{lemma}
\begin{proof}
  The statement follows immediately from Definition~\ref{dd2} and Lemmas~\ref{ll1}(i) and~\ref{ll7}. 
\end{proof}

\begin{proposition}\label{p_Im}
  Let $\mf A$ be an HTG and $X$ be a sequentially closed subset of $\mf A$. Let $\mathcal{K}$ be a ring and $\mu$ be an $\mf A$-valued $\sigma$-additive function such that $\delta(\mathcal K)\subset D_\mu\subset \sigma(\mathcal K)$ and $\im\mu|_{\mathcal{K}}\subset X$. Then $\im\mu \subset X$. 
\end{proposition}
\begin{proof}
  Let $\mathcal{R} = \delta(\mathcal{K})$ and $\nu = \mu|_{\mathcal{R}}$. Lemma~\ref{l_semiprem} implies that $\mathcal{Q} = \nu^{-1}(X)$ is a semi-monotone class.   Since $\mathcal{K}\subset\mathcal{Q}$, it follows from Lemma~\ref{ll4} that $\mathcal{Q} = \mathcal{R}$ and, hence, $\im\nu \subset X$. Since $\sigma(\mathcal{K}) = \sigma(\mathcal{R})$, Proposition~\ref{p_Im0delta} ensures that $\im\mu\subset X$.
\end{proof}

\begin{theorem}\label{tt4}
  Let $\mf A$ be an HTG and $\mf B$ be a sequentially closed submonoid of $\mf A$, $\mathcal K$ be a $\cap$-closed set, and $\mu$ be an $\mf A$-valued $\sigma$-additive function such that $\delta(\mathcal K)\subset D_\mu\subset \sigma(\mathcal K)$ and $\im\mu|_{\mathcal{K}}\subset\mf B$. Suppose that either $\mathcal K$ is a semi-ring or $\mf B$ is a subgroup of $\mf A$. Then $\im \mu \subset \mf B$. 
\end{theorem}
\begin{proof}
   Let $\mathcal{R} = \delta(\mathcal{K})$ and $\nu = \mu|_{\mathcal{R}}$. Lemma~\ref{l_semiprem} implies that $\mathcal{Q} = \nu^{-1}(\mf B)$ is a semi-monotone class. Clearly, $\mathcal{K}\subset\mathcal{Q}\subset\sigma(\mathcal{K})$. If $A,B\in\mathcal Q$ and $A\cap B=\varnothing$, then $A\cup B\in\mathcal Q$ because $\nu$ is additive and $\mf B$ is stable under addition. As $0\in\mf B$, we have $\varnothing\in\mathcal Q$ and, hence, $\mathcal Q$ is closed under finite disjoint unions. We claim that $\mathcal Q$ is a $\delta$-ring. If $\mathcal K$ is a semi-ring, then this is true by Lemma~\ref{l_semimonotone}. Suppose that $\mf B$ is a subgroup of $\mf A$. If $A,B\in\mathcal Q$ are such that $B\subset A$, then $A\setminus B\in \mathcal Q$ because $\nu(A\setminus B)=\nu(A)-\nu(B)$ and $\nu(A),\nu(B)\in \mf B$. This means that $\mathcal Q$ is an $\alpha$-class and, therefore, our claim follows from Lemma~\ref{l_alpha}. As $\mathcal{Q}$ is a $\delta$-ring, we have $\mathcal{Q} = \mathcal{R}$ and, hence, $\im\nu \subset \mf B$. Since $\sigma(\mathcal{K}) = \sigma(\mathcal{R})$, Proposition~\ref{p_Im0delta} ensures that $\im\mu\subset \mf B$.
\end{proof}

\begin{proposition}\label{p_groupunique}
  Let $\mathcal K$ be a $\cap$-closed set, $\mf A$ be an HTG, and $\mu_1$ and $\mu_2$ be $\mf A$-valued $\sigma$-additive functions such that $\delta(\mathcal K)\subset D_{\mu_1}\cap D_{\mu_2}\subset \sigma(\mathcal K)$ and $\mu_1|_{\mathcal K}=\mu_2|_{\mathcal K}$. Then $\mu_1(A) = \mu_2(A)$ for every $A\in D_{\mu_1}\cap D_{\mu_2}$.
\end{proposition}
\begin{proof}
  Let $\mathcal Q = D_{\mu_1}\cap D_{\mu_2}$ and the map $\mu\colon\mathcal Q\to\mf A$ be such that $\mu(A) = \mu_1(A) - \mu_2(A)$ for every $A\in \mathcal Q$. Then $\mu$ is an $\mf A$-valued $\sigma$-additive function and $\mu(A)=0$ for every $A\in\mathcal K$. Applying Theorem~\ref{tt4} to $\mf B = \{0\}$, we conclude that $\mu(A)=0$ for every $A\in D_\mu$, i.e., $\mu_1$ and $\mu_2$ coincide on $\mathcal Q$.
\end{proof}

\begin{remark}
  If $\mf A$ is a Hausdorff topological vector space, $\mathcal{K}$ is a ring, and $D_\mu$ and $D_\nu$ are both equal either to $\delta(\mathcal{K})$ or to $\sigma(\mathcal{K})$, then Proposition~\ref{p_groupunique} coincides with Proposition~6 in Sec.~I.2.8 of~\cite{Dinculeanu}.
\end{remark}

\begin{theorem}\label{tt1}
  Let $\mathcal K$ be a $\cap$-closed set, $\mf A$ be an HTG, and $\mu_1$ and $\mu_2$ be $\mf A$-valued premeasures such that $\mathcal K\subset D_{\mu_1}\subset \sigma(\mathcal K)$, $\mathcal K\subset D_{\mu_2}\subset\sigma(\mathcal K)$, and $\mu_1|_{\mathcal K}=\mu_2|_{\mathcal K}$. Then $\mu_1 \asymp \mu_2$. If $\mu_2$ is a measure, then it is an extension of $\mu_1$. If $\mu_1$ and $\mu_2$ are both measures, then $\mu_1=\mu_2$.
\end{theorem}
\begin{proof}
  As $D_{\mu_1}\cap D_{\mu_2}$ is a $\delta$-ring, we have $\delta(\mathcal{K})\subset D_{\mu_1}\cap D_{\mu_2}$. Since $\q{\mu_1} = \sigma(\mathcal K) = \q{\mu_2}$, it follows from Proposition~\ref{p_groupunique} that $\mu_1 \asymp \mu_2$. If $\mu_2$ is a measure, then it is an extension of $\mu_1$ by Definition~\ref{dd6}. If $\mu_1$ and $\mu_2$ are both measures, then they are extensions of each other and, hence, $\mu_1=\mu_2$. 
\end{proof}

\section{\texorpdfstring{Extension of $\sigma$-contents}{Extension of σ-contents}}
\label{s_ext}

\begin{definition}\label{d_sf}
  Let $\mf A$ be an HTM.   An $\mf A$-valued map $\mu$ is called exhaustive if $\mu(\mb A(n))\to 0$ in $\mf A$ as $n\to\infty$ for every disjoint family $\mb A\colon\N\to D_\mu$.  
\end{definition}

\begin{proposition}\label{p_sion}
  Let $\mf A$ be a sequentially complete HTG and $\mu$ be an exhaustive $\mf A$-valued $\sigma$-content. Then there is a unique $\mf A$-valued $\sigma$-additive function $\nu$ such that $D_\nu = \q\mu$ and $\mu = \nu|_{D_\mu}$.
\end{proposition}

\begin{proof}
  For a complete $\mf A$, the existence assertion follows from Sion's extension theorem; see~\cite[Theorem~3.3]{Sion1969} and~\cite[Theorem~I.6.2]{Sion1973}. Let $\hat{\mf A}$ be the completion of $\mf A$. Applying the complete case to $\mu$, regarded as an $\hat{\mf A}$-valued $\sigma$-content, we obtain an $\hat{\mf A}$-valued $\sigma$-additive function $\nu$ such that $D_\nu=\q\mu$ and $\mu=\nu|_{D_\mu}$. The sequential completeness of $\mf A$ implies that $\mf A$ is sequentially closed in $\hat{\mf A}$. By Proposition~\ref{p_Im}, we conclude that $\im\nu\subset \mf A$ and, hence, $\nu$ is an $\mf A$-valued $\sigma$-additive function. The uniqueness of $\nu$ is ensured by Theorem~\ref{tt1}.
\end{proof}

\begin{remark}
  For a complete $\mf A$, Proposition~\ref{p_sion} was first obtained by Sion~\cite{Sion1969}, who adapted the well-known Carath\'eodory construction to the group-valued case. Alternative proofs based on topological extension arguments can be found in~\cite{Drewnowski1972,Herer1976,Weber1976, Weber2002}. It was pointed out by Drewnowski (see Remark~1 after Theorem~9.2 in~\cite{Drewnowski1972}) that it suffices to require the sequential completeness of $\mf A$. The above proof adapts Drewnowski's argument in a form that is also applicable to $\delta$-rings (see Remark~\ref{r_weber} below). 
 This may be useful because some groups important for applications (in particular, the space of all bounded linear operators on an infinite-dimensional Hilbert space endowed with the strong operator topology) are sequentially complete but not complete.
\end{remark}

\begin{definition}
  Let $\mf A$ be an HTM and $\mu$ be an $\mf A$-valued map. Given a set $A$, we say that $\mu$ is $A$-exhaustive and write $\ex{\mf A}{A}\mu$ if $\mu(\mb A(n))\to 0$ in $\mf A$ as $n\to\infty$ whenever $\mb A\colon\N\to D_\mu$ is a disjoint family such that $\mb A(n)\subset A$ for every $n\in\N$. We call $\mu$ weakly exhaustive if $\ex{\mf A}{A}\mu$ for every $A\in D_\mu$.
\end{definition}

Clearly, an $\mf A$-valued map is exhaustive if and only if it is $A$-exhaustive for every $A$. Every exhaustive $\mf A$-valued map is weakly exhaustive.

Let $\mf A$ be an HTM and $\mu$ be an $\mf A$-valued map. We define
\begin{equation}
  \label{eq:Smu}
  \mathcal{S}_{\mf A}[\mu] = \set{B\in \q\mu}{B\subset A\mbox{ for some }A\in \sigma_0(D_\mu)\mbox{ such that }\ex{\mf A}{A}\mu}.
\end{equation}
 
\begin{lemma}\label{l_Smuaux}
  Let $\mf A$ be an HTM, $\mu$ be an $\mf A$-valued $\sigma$-content, $C\in D_\mu$, and $V$ be a neighbourhood of $\mu(C)$ in $\mf A$. Let $A,B\in \sigma_0(D_\mu)$ be such that $C\subset A\cup B$. Then there are $A',B'\in D_\mu$ such that $A'\subset A\cap C$, $B'\subset B\cap C$, and $\mu(A') + \mu(B') \in V$.  
\end{lemma}
\begin{proof}
  Let $\mb A\colon\N\to D_\mu$ and $\mb B\colon\N\to D_\mu$ be increasing families of sets such that $A = \bigcup_{n\in\N}\mb A(n)$ and $B = \bigcup_{n\in\N}\mb B(n)$. Lemma~\ref{ll7}(i) implies that
  \[
    \mu(C\cap (\mb A(n)\cup\mb B(n))) \in V
  \]
  for some $n\in\N$. Clearly, the sets $A' = C\cap \mb A(n)$ and $B' = (C\cap\mb B(n))\setminus A'$ have the required properties.
\end{proof} 

\begin{proposition}\label{p_Smu}
  Let $\mf A$ be a regular HTM and $\mu$ be an $\mf A$-valued $\sigma$-content. Then $\mathcal{S}_{\mf A}[\mu]$ is a $\delta$-ring.
\end{proposition}
\begin{proof}
  Let $\mathcal{Q} = \mathcal{S}_{\mf A}[\mu]$. It is clear from~(\ref{eq:Smu}) that, together with each of its elements, $\mathcal{Q}$ contains all its subsets belonging to $\q\mu$. In particular, $\mathcal{Q}$ is closed under countable intersections and $A\setminus B\in \mathcal{Q}$ for every $A,B\in \mathcal{Q}$. It remains to show that $A\cup B\in \mathcal{Q}$ for every $A,B\in \mathcal{Q}$. Let $A,B\in \mathcal{Q}$. Let $E,F\in\sigma_0(D_\mu)$ be such that $A\subset E$, $B\subset F$, $\ex{\mf A}{E}\mu$, and $\ex{\mf A}{F}\mu$. Let $G = E\cup F$. Since $G\in \sigma_0(D_\mu)$, $A\cup B\in\q\mu$, and $A\cup B\subset G$, it suffices to verify that $\ex{\mf A}{G}\mu$. Suppose this is not the case. Then there are a disjoint family $\mb C\colon \N\to D_\mu$ and a closed neighbourhood of the origin $U$ in $\mf A$ such that $\mb C(n)\subset G$ and $\mu(\mb C(n))\notin U$ for every $n\in \N$. By Lemma~\ref{l_Smuaux} applied to $V = \mf A\setminus U$, there are $\mb A\colon \N\to D_\mu$ and $\mb B\colon \N\to D_\mu$ such that $\mb A(n)\subset E\cap\mb C(n)$, $\mb B(n)\subset F\cap\mb C(n)$, and
  \begin{equation}
    \label{eq:notin}
    \mu(\mb A(n)) + \mu(\mb B(n))\notin U
  \end{equation}
  for every $n\in \N$. As $\mb C$ is disjoint, the families $\mb A$ and $\mb B$ are disjoint and, therefore, $\lim_{n\to\infty}\mu(\mb A(n)) = \lim_{n\to\infty}\mu(\mb B(n)) = 0$. This implies that $\lim_{n\to\infty}(\mu(\mb A(n)) + \mu(\mb B(n))) = 0$, in contradiction to~(\ref{eq:notin}).
\end{proof}

The next theorem is the main result of this section.

\begin{theorem}\label{t_ext}
  Let $\mf A$ be a sequentially complete HTG and $\mu$ be a weakly exhaustive $\mf A$-valued $\sigma$-content. Then there is a unique $\mf A$-valued measure $\nu$ such that $D_\mu\subset D_\nu\subset \q\mu$ and $\mu = \nu|_{D_\mu}$. We have $D_\mu\subset\mathcal{S}_{\mf A}[\mu]\subset D_\nu$.
\end{theorem}

If $\mu$ is exhaustive, then $\mathcal{S}_{\mf A}[\mu] = \q\mu$ by Lemma~\ref{l_sigma0} and Theorem~\ref{t_ext} reduces to Proposition~\ref{p_sion}. In general, it may happen that $D_\nu \neq \mathcal{S}_{\mf A}[\mu]$ under the conditions of Theorem~\ref{t_ext} (see the example in Appendix~\ref{s_example2}). We expect that the equality $D_\nu = \mathcal{S}_{\mf A}[\mu]$ holds for metrizable~$\mf A$.

\begin{corollary}\label{c_weber}
  Under the conditions of Theorem~\ref{t_ext}, there is a unique $\mf A$-valued $\sigma$-additive function $\nu$ such that $D_\nu = \delta(D_\mu)$ and $\mu = \nu|_{D_\mu}$.
\end{corollary}
\begin{proof}
  The existence and uniqueness of $\nu$ are ensured by Theorem~\ref{t_ext} and Proposition~\ref{p_groupunique} respectively.
\end{proof}

\begin{remark}\label{r_weber}
  Corollary~\ref{c_weber} was obtained by Weber (see Folgerung~4.6 in~\cite{Weber1976}) for complete $\mf A$. The present version follows from Weber's result by the same completion argument as in the proof of Proposition~\ref{p_sion}. The existence of $\nu$ in Theorem~\ref{t_ext} follows immediately from  Corollary~\ref{c_weber} (just extend $\mu$ to $\delta(D_\mu)$ and take the associated measure) and, hence, is an easy consequence of Weber's result. At the same time, this argument does not yield the inclusion $\mathcal{S}_{\mf A}[\mu]\subset D_\nu$. Below, we derive Theorem~\ref{t_ext} (and Weber's result) from Proposition~\ref{p_sion}.
\end{remark}

\begin{corollary}\label{c_posext}
  Let $\mu$ be a positive $\sigma$-content. Then there is a unique $\R$-valued measure $\nu$ such that $D_\mu\subset D_\nu\subset \q\mu$ and $\mu = \nu|_{D_\mu}$. We have $\im\nu\subset\R_+$.  
\end{corollary}
\begin{proof}
  Let $A\in D_\mu$ and $\mb A\colon\N\to D_\mu$ be a disjoint family such that $\mb A(n)\subset A$ for every $n\in\N$. By the positivity of $\mu$, we have $\finsum_{k\in K}\mu(\mb A(k)) \leq \mu(A)$ for every $K\in\mathcal{F}_\N$ and, hence, the family $\mu\circ\mb A$ is summable in $\R$. This implies that $\lim_{n\to\infty}\mu(\mb A(n)) = 0$, i.e., $\mu$ is a weakly exhaustive $\R$-valued map. The existence and uniqueness of $\nu$ therefore follow from Theorem~\ref{t_ext}. As $D_\nu$ is a $\delta$-ring, we have $\delta(D_\mu)\subset D_\nu$ and Proposition~\ref{p_Im} applied to $\R_+$, $D_\mu$, and $\nu$ in place of $X$, $\mathcal{K}$, and $\mu$ respectively ensures that $\im \nu\subset \R_+$.
\end{proof}

Using Corollary~\ref{c_posext} and the correspondence between $\mf R$-valued measures and positive measures established in Sec.~\ref{s_pos}, it is easy to derive Halmos's extension result for $\mf R$-valued $\sigma$-contents.

\begin{proposition}[{\cite[Sec.~13, Theorem~A]{Halmos}}]
  Let $\eta$ be a $\sigma$-finite $\mf R$-valued $\sigma$-content. Then there is a unique $\mf R$-valued measure $\zeta$ such that $D_\zeta = \q\eta$ and $\eta = \zeta|_{D_\eta}$. The measure $\zeta$ is $\sigma$-finite.
\end{proposition}
\begin{proof}
  Let $\mu = \mf F\eta$. Proposition~\ref{p_Fmupremeas} implies that $\mu$ is a positive $\sigma$-content. By Corollary~\ref{c_posext}, there is a positive measure $\nu$ such that $\nu$ is an extension of $\mu$ and $\q\nu =\q\mu$. Proposition~\ref{p_barsigma} ensures that $\q\mu = \q\eta$ and, hence, $\q\nu = \q\eta$. Let $\zeta = \as[\mf R]\nu$. By Theorem~\ref{t_Fass}(i), $\zeta$ is a $\sigma$-finite $\mf R$-valued measure, $\zeta$ is an extension of $\nu$, and $D_\zeta = \q\eta$. Let $\mathcal{K} = D_\mu$. As $D_\eta\subset\sigma_0(\mathcal{K})$ by Proposition~\ref{p_barsigma} and $D_\eta\subset D_\zeta$, we have $\mathcal{K}\subset D_\eta\cap D_\zeta\subset \sigma_0(\mathcal{K})$. Since $\eta|_{\mathcal{K}} = \zeta|_{\mathcal{K}}$, Proposition~\ref{p_un0} implies that $\eta = \zeta|_{D_\eta}$. Let $\zeta'$ be another $\mf R$-valued measure such that $D_{\zeta'} = \q\eta$ and $\eta = \zeta'|_{D_\eta}$. Let $\nu' = \mf F\zeta'$. By Proposition~\ref{p_Fmu}, $\nu'$ is a positive measure. Clearly, $\nu'$ is an extension of $\mu$. As $D_{\zeta'} = \q\mu$, we have $D_{\nu'}\subset\q\mu$ and it follows from the uniqueness part of Corollary~\ref{c_posext} that $\nu = \nu'$. This means that $\zeta$ and $\zeta'$ coincide on $D_\nu$. Since $\q\nu = D_\zeta = D_{\zeta'}$, Theorem~\ref{t_unique} implies that $\zeta = \zeta'$. 
\end{proof}

The rest of this section is devoted to the proof of Theorem~\ref{t_ext}.

Given sets $\mathcal{K}$ and $A$, we put $\mathcal{K}|^A = \set{B}{B\in\mathcal{K}\mbox{ and }B\subset A}$. For every $\mathcal{K}$, $A$, and $B$, we obviously have
\begin{equation}
  \label{eq:|cap}
  \mathcal{K}|^{A\cap B} = \mathcal{K}|^A\cap\mathcal{K}|^B.
\end{equation}

\begin{lemma}\label{l_sigma|}
  Let $\mathcal{K}$ be a $\cap$-closed set and $A\in \sigma_0(\mathcal{K})$. Then $\sigma(\mathcal{K}|^A) = \sigma(\mathcal{K})|^A$.
\end{lemma}
\begin{proof}
  Let $\mathcal{Q} = \sigma(\mathcal{K}|^A)$ and $\mathcal{Q}' = \sigma(\mathcal{K})|^A$. Since $\mathcal{Q}'$ is a $\sigma$-ring and $\mathcal{K}|^A\subset\mathcal{Q}'$, we have $\mathcal{Q}\subset\mathcal{Q}'$. We show that
  \begin{equation}\label{eq:BAQ}
    A\cap B\in \mathcal{Q},\quad B\in \sigma(\mathcal{K}).
  \end{equation}
  Let $\mathcal{R} = \set{B}{B\in\sigma(\mathcal{K})\mbox{ and }A\cap B\in \mathcal{Q}}$. Let $B\in \mathcal{K}$ and $\mb A\colon I\to \mathcal{K}$ be a countable family such that $A = \bigcup_{i\in I}\mb A(i)$. Since $\mathcal{K}$ is $\cap$-closed, we have $\mb A(i)\cap B\in \mathcal{K}|^A$ for every $i\in I$ and the equality $A\cap B = \bigcup_{i\in I}(\mb A(i)\cap B)$ implies that $A\cap B\in \mathcal{Q}$ and, hence, $B\in \mathcal{R}$. Thus, $\mathcal{K}\subset \mathcal{R}$. As $\mathcal{R}$ is obviously a $\sigma$-ring, we conclude that $\mathcal{R} = \sigma(\mathcal{K})$ and (\ref{eq:BAQ}) is proved. Let $B\in \mathcal{Q}'$. Since $B\in \sigma(\mathcal{K})$ and $B = A\cap B$, it follows from~(\ref{eq:BAQ}) that $B\in \mathcal{Q}$. This means that $\mathcal{Q}'\subset\mathcal{Q}$ and, therefore, $\mathcal{Q} = \mathcal{Q}'$.
\end{proof}

\begin{proof}[Proof of Theorem~\ref{t_ext}]
  Let $\mathcal{S} = \mathcal{S}_{\mf A}[\mu]$ and $\Lambda = \set{\lambda\in \sigma_0(D_\mu)}{\ex{\mf A}{\lambda}{\mu}}$. Clearly, $\Lambda$ is $\cap$-closed. Let $\fm{D}{\Lambda}$ and $\fm{\Theta}{\Lambda}$ be such that $D(\lambda) = D_\mu|^\lambda$ and $\Theta(\lambda) = \q\mu|^\lambda$ for every $\lambda\in \Lambda$.
  Then
  \begin{equation}
    \label{eq:DS}
    D_\mu = \bigcup_{\lambda\in\Lambda}D(\lambda),\quad \mathcal{S} = \bigcup_{\lambda\in\Lambda}\Theta(\lambda).
  \end{equation}
  Indeed, as $\mu$ is weakly exhaustive, we have $D_\mu\subset\Lambda$. This implies the first equality in~(\ref{eq:DS}). The second equality in~(\ref{eq:DS}) follows immediately from~(\ref{eq:Smu}). By Lemma~\ref{l_sigma|}, we conclude that
  \begin{equation}
    \label{eq:sigmaD}
    \Theta(\lambda) = \sigma(D(\lambda)),\quad \lambda\in\Lambda.
  \end{equation}
  Let $\fm{\rho}{\Lambda}$ be such that $\rho(\lambda) = \mu|_{D(\lambda)}$ for every $\lambda\in \Lambda$. The definition of $\Lambda$ implies that $\rho(\lambda)$ is an exhaustive $\mf A$-valued $\sigma$-content for every $\lambda\in \Lambda$. By Proposition~\ref{p_sion} and~(\ref{eq:sigmaD}), there is $\fm{\tau}{\Lambda}$ such that $\tau(\lambda)$ is an $\mf A$-valued $\sigma$-additive function, $D_{\tau(\lambda)} = \Theta(\lambda)$, and $\tau(\lambda)|_{D(\lambda)} = \rho(\lambda)$ for every $\lambda\in \Lambda$. We show that
  \begin{equation}
    \label{eq:lsubl'}
    \tau(\lambda)|_{\Theta(\lambda')} = \tau(\lambda')\mbox{ for every }\lambda,\lambda'\in\Lambda\mbox{ such that }\lambda'\subset\lambda.
  \end{equation}
  Let $\lambda,\lambda'\in \Lambda$ and $\lambda'\subset\lambda$. Let $\mathcal{K} = D(\lambda')$ and $\varphi = \tau(\lambda)|_{\Theta(\lambda')}$. Since $\rho(\lambda)|_{\mathcal{K}} = (\mu|_{D(\lambda)})|_{\mathcal{K}} = \mu|_{\mathcal{K}} = \rho(\lambda')$, we obtain
  \[
    \varphi|_{\mathcal{K}} = \tau(\lambda)|_{\mathcal{K}} = (\tau(\lambda)|_{D(\lambda)})|_{\mathcal{K}} = \rho(\lambda)|_{\mathcal{K}} = \rho(\lambda') = \tau(\lambda')|_{\mathcal{K}}.
  \]
  By Proposition~\ref{p_groupunique} and~(\ref{eq:sigmaD}), we conclude that $\varphi = \tau(\lambda')$ and (\ref{eq:lsubl'}) is proved. Since $\Lambda$ is $\cap$-closed, it follows from~(\ref{eq:|cap}) that $\Theta(\lambda)\cap\Theta(\lambda') = \Theta(\lambda\cap\lambda')$ for every $\lambda,\lambda'\in\Lambda$. By~(\ref{eq:lsubl'}), we conclude that $\tau(\lambda)|_{\Theta(\lambda)\cap\Theta(\lambda')} = \tau(\lambda\cap\lambda') = \tau(\lambda')|_{\Theta(\lambda)\cap\Theta(\lambda')}$ for every $\lambda,\lambda'\in\Lambda$. In view of the second equality in~(\ref{eq:DS}), this implies the existence of $\fm{\eta}{\mathcal{S}}$ such that $\eta|_{\Theta(\lambda)} = \tau(\lambda)$ for every $\lambda\in \Lambda$. If $\lambda,\lambda'\in\Lambda$, $A\in \Theta(\lambda)$, and $B\in\Theta(\lambda')$, then we obviously have $A\cap B\in \Theta(\lambda)\cap\Theta(\lambda')$. By Propositions~\ref{ll6a}(iii) and~\ref{p_Smu}, we conclude that $\eta$ is an $\mf A$-valued premeasure. By~(\ref{eq:DS}), we have $D_\mu\subset\mathcal{S}$. Since
  \[
    (\eta|_{D_\mu})|_{D(\lambda)} = \eta|_{D(\lambda)} = (\eta|_{\Theta(\lambda)})|_{D(\lambda)} = \tau(\lambda)|_{D(\lambda)} = \rho(\lambda)
  \]
  for every $\lambda\in\Lambda$, the first equality in~(\ref{eq:DS}) implies that $\eta|_{D_\mu} = \mu$. Let $\nu = \as[\mf A]\eta$. Then $\nu$ is an extension of $\eta$ and it follows from~(\ref{Qmu}) that $\q\nu = \sigma(\mathcal{S})$. We therefore have $D_\mu\subset \mathcal{S}\subset D_\nu\subset\q\mu$ and $\nu|_{D_\mu} = \mu$. The uniqueness of $\nu$ is ensured by Theorem~\ref{tt1}.
\end{proof}

\appendix

\section{Summation in monoids}
\label{app_B}

The notion of summability used in Sec.~\ref{s_add} and several of the results below are standard for HTGs; see, e.g., \cite[Sec.~III.5]{BourbakiTopologie1_4}. We give a self-contained treatment because we need corresponding statements for arbitrary HTMs and, in particular, must keep track of where regularity is required.

First of all, we give the precise definition of a finite sum. Given a monoid $\mf A$, there is a unique map $\psi$ from the class of all finite $\mf A$-valued families to $\mf A$ such that 
\begin{enumerate}
\item[(a)] $\psi(\varnothing) = 0$,
\item[(b)] if $f\colon K \to \mf A$ is a finite family and $k\in K$, then $\psi(f) = f(k) + \psi\left(f|_{K\setminus\{k\}}\right)$.
\end{enumerate}
The existence and uniqueness of $\psi$ follow by induction on the cardinality of $K$. If $f$ is a finite $\mf A$-valued family, then we put $\finsum f = \psi(f)$. If $K$ is a finite set and $f$ is such that $f(k)\in\mf A$ for every $k\in K$, then we define $\finsum_{k\in K}f(k) = \finsum\{f(k)\}_{k\in K}$.

\begin{lemma}\label{l_finsumdef}
  Let $\mf A$ be a monoid and $g\colon K\to \mf A$ be a finite family. If $K = \varnothing$, then $\finsum g = 0$. If $k_0\in K$, then $\finsum g = g(k_0) + \finsum g|_{K\setminus\{k_0\}}$. 
\end{lemma}
\begin{proof}
  The statement follows immediately from the definition of $\finsum g$.
\end{proof}

\begin{proposition}\label{p_finsumdef1}
  Let $\mf A$ be a monoid, $K$ be a finite set, and $f(k)\in \mf A$ for all $k\in K$. If $K = \varnothing$, then $\finsum_{k\in K} f(k) = 0$. If $k_0\in K$, then $\finsum_{k\in K} f(k) = f(k_0) + \finsum_{k\in K\setminus\{k_0\}}f(k)$. 
\end{proposition}
\begin{proof}
  Let $g = \{f(k)\}_{k\in K}$. By definition, we have $\finsum_{k\in K}f(k) = \finsum g$. If $K = \varnothing$, then $\finsum g = 0$ by Lemma~\ref{l_finsumdef} and, hence, $\finsum_{k\in K} f(k) = 0$. Let $k_0\in K$. Since $\{f(k)\}_{k\in K\setminus \{k_0\}} = g|_{K\setminus\{k_0\}}$, we have $\finsum_{k\in K\setminus\{k_0\}}f(k) = \finsum g|_{K\setminus\{k_0\}}$. As $g(k_0) = f(k_0)$, it follows from Lemma~\ref{l_finsumdef} that $\finsum_{k\in K} f(k) = f(k_0) + \finsum_{k\in K\setminus\{k_0\}}f(k)$.
\end{proof}

\begin{proposition}\label{p_finsum}
  Let $\mf A$ be a monoid, $K$ be a finite set, $f(k)\in \mf A$ for all $k\in K$, and $x = \finsum_{k\in K}f(k)$.
  \begin{enumerate}
  \item[(i)] Let $\tau\colon J\to K$ be a bijection. Then $x = \finsum_{j\in J}f(\tau(j))$.
  \item[(ii)] Let $g(k)\in \mf A$ for all $k\in K$. Then $\finsum_{k\in K}(f(k) + g(k)) = x + \finsum_{k\in K}g(k)$.
  \item[(iii)] Let $\mf B$ be a monoid and $\varphi\colon\mf A\to\mf B$ be a monoid homomorphism. Then $\finsum_{k\in K} \varphi(f(k)) = \varphi(x)$.  
  \item[(iv)] Let $K = I\cup J$, where $I\cap J = \varnothing$. Then $x = \finsum_{k\in I}f(k) + \finsum_{k\in J}f(k)$.
  \item[(v)] Let $f(k) = 0$ for all $k\in K$. Then $x = 0$.
  \end{enumerate}
\end{proposition}
\begin{proof}
  All statements are proved by induction on the cardinality of $K$. By Proposition~\ref{p_finsumdef1}, they hold for $K = \varnothing$. Let $n\in\N$ and suppose the statements are true for all $K$ with $n-1$ elements. Let $K$ have $n$ elements. Choose $k_0\in K$ and let $K' = K\setminus\{k_0\}$. The induction step for each statement is completed as follows.
  \par\medskip\noindent
  (i) Let $j_0 = \tau^{-1}(k_0)$, $J' = J\setminus\{j_0\}$, and $\tau' = \tau|_{J'}$. Clearly, $\tau'$ is a bijection between $J'$ and $K'$. By the  induction hypothesis, we have $\finsum_{k\in K'}f(k) = \finsum_{j\in J'}f(\tau'(j)) = \finsum_{j\in J'} f(\tau(j))$. By Proposition~\ref{p_finsumdef1}, it follows that
  \[
    x = f(k_0) + \finsum_{k\in K'}f(k) = f(\tau(j_0)) + \finsum_{j\in J'}f(\tau(j)) = \finsum_{j\in J}f(\tau(j)). 
  \]  
  \par\noindent
  (ii) The induction hypothesis and Proposition~\ref{p_finsumdef1} yield 
  \begin{multline}\nonumber
    \finsum_{k\in K}(f(k) + g(k)) = f(k_0) + g(k_0) + \finsum_{k\in K'}(f(k) + g(k)) \\= f(k_0) + g(k_0) +  \finsum_{k\in K'}f(k) + \finsum_{k\in K'}g(k) = x + \finsum_{k\in K}g(k).
  \end{multline}
  \par\noindent
  (iii) The induction hypothesis and Proposition~\ref{p_finsumdef1} imply that
  \[
    \finsum_{k\in K} \varphi(f(k)) = \varphi(f(k_0)) + \finsum_{k\in K'}\varphi(f(k)) = \varphi(f(k_0)) + \varphi\left(\finsum_{k\in K'} f(k)\right) = \varphi(x).
  \]
  (iv) We have either $k_0\in I$ or $k_0\in J$. For certainty, we assume $k_0\in I$. Let $I' = I\setminus\{k_0\}$. Then $K' = I'\cup J$ and $I'\cap J = \varnothing$. By the induction hypothesis and Proposition~\ref{p_finsumdef1}, we obtain
  \[
    x = f(k_0) + \finsum_{k\in K'}f(k) = f(k_0) + \finsum_{k\in I'}f(k) + \finsum_{k\in J} f(k) = \finsum_{k\in I}f(k) + \finsum_{k\in J}f(k).
  \]
  (v) The induction hypothesis and Proposition~\ref{p_finsumdef1} yield
  $x = f(k_0) + \finsum_{k\in K'} f(k) = 0 + 0 = 0$.
\end{proof}

We now turn to unconditional summation in HTMs. Some additional results concerning finite summation will later be derived from the analogous topological results.

\begin{proposition}\label{p_sumfin}
  Let $\mf A$ be an HTM, $K$ be a finite set, and $f(k)\in \mf A$ for all $k\in K$. Then $\sum_{k\in K}f(k) = \finsum_{k\in K} f(k)$.
\end{proposition}
\begin{proof}
  Since $K$ is the greatest element of $\mathcal{F}_K$, the net $\theta = \bigS_{k\in K} f(k)$ converges in $\mf A$ over $\mathcal{F}_K$ to $\theta(K) = \finsum_{k\in K} f(k)$.
\end{proof}

\begin{corollary}\label{c_emptysum}
  Let $\mf A$ be an HTM. For every $f$, we have $\sm{\mf A}{k\in \varnothing}{f(k)}$ and $\sum_{k\in \varnothing}f(k) = 0$.
\end{corollary}
\begin{proof}
  The statement follows from Propositions~\ref{p_finsumdef1} and~\ref{p_sumfin}.
\end{proof}

\begin{proposition}\label{p_sumchange}
  Let $\mf A$ be an HTM and $\tau\colon J\to K$ be a bijection. Then we have $\sm{\mf A}{k\in K}{f(k)}$ if and only if $\sm{\mf A}{j\in J}{f(\tau(j))}$. If $\sm{\mf A}{k\in K}{f(k)}$, then $\sum_{k\in K}f(k) = \sum_{j\in J}f(\tau(j))$.
\end{proposition}
\begin{proof}
    We can assume that $f(k)\in \mf A$ for all $k\in K$ because this condition is implied by both $\sm{\mf A}{k\in K}{f(k)}$ and $\sm{\mf A}{j\in J}{f(\tau(j))}$.
  Let  $\theta = \bigS_{k\in K} f(k)$ and $\rho = \bigS_{j\in J}f(\tau(j))$. Let $T\colon \mathcal{F}_J \to \mathcal{F}_K$ be such that $T(J') = \im \tau|_{J'}$ for every $J'\in \mathcal{F}_J$. Let $J'\in \mathcal{F}_J$, $\tau' = \tau|_{J'}$, and $K' = T(J')$.  Since $\tau'$ is a bijection between $J'$ and $K'$, Proposition~\ref{p_finsum}(i) implies that
  \[
    \theta(K') = \finsum_{k\in K'}f(k) = \finsum_{j\in J'}f(\tau'(j)) = \finsum_{j\in J'}f(\tau(j)) = \rho(J').
  \]
  Thus, $\rho = \theta\circ T$. As $T$ is an isomorphism of ordered sets between $\mathcal{F}_J$ and $\mathcal{F}_K$, it follows that $\theta$ converges in $\mf A$ over $\mathcal{F}_K$ if and only if $\rho$ converges in $\mf A$ over $\mathcal{F}_J$. This means that $\sm{\mf A}{k\in K}{f(k)}$ if and only if $\sm{\mf A}{j\in J}{f(\tau(j))}$. If $\sm{\mf A}{k\in K}{f(k)}$, then we have  $\sum_{k\in K} f(k) = \lim_{K'\in \mathcal{F}_K}\theta(K') = \lim_{J'\in \mathcal{F}_J}\rho(J') = \sum_{j\in J}f(\tau(j))$.
\end{proof}

\begin{proposition}\label{p_sum}
  Let $\mf A$ be an HTM, $\sm{\mf A}{k\in K}{f(k)}$, and $x = \sum_{k\in K}f(k)$.
  \begin{enumerate}
  \item[(i)] Let $\sm{\mf A}{k\in K}{g(k)}$. Then $\sum_{k\in K}(f(k) + g(k)) = x + \sum_{k\in K}g(k)$.
  \item[(ii)] Let $\mf B$ be an HTM and $\varphi\colon\mf A\to\mf B$ be a continuous homomorphism. Then $\sum_{k\in K} \varphi(f(k)) = \varphi(x)$.   
  \item [(iii)] Suppose $\mf A$ is an HTG. Then $\sum_{k\in K}(-f(k)) = -x$.
  \item [(iv)] Suppose $\mf A$ is an HTG and $\sm{\mf A}{k\in K}{g(k)}$. Then $\sum_{k\in K}(f(k) - g(k)) = x - \sum_{k\in K}g(k)$.
  \end{enumerate}
\end{proposition}
\begin{proof}
  (i) Let  $\theta = \bigS_{k\in K} f(k)$,  $\rho = \bigS_{k\in K}g(k)$ , and $\sigma = \bigS_{k\in K}(f(k)+g(k))$. By Proposition~\ref{p_finsum}(ii), we have $\sigma(J) = \theta(J) + \rho(J)$ for every $J\in \mathcal{F}_K$. Hence, 
  \[
    \sum_{k\in K}(f(k)+g(k)) = \lim_{J\in \mathcal{F}_K} \sigma(J) = \lim_{J\in \mathcal{F}_K}\theta(J)  + \lim_{J\in \mathcal{F}_K} \rho(J)  = x + \sum_{k\in K}g(k).
  \]
  (ii) Let  $\theta = \bigS_{k\in K} f(k)$ and $\rho = \bigS_{k\in K}\varphi (f(k))$. By Proposition~\ref{p_finsum}(iii), we have $\rho(J) = \varphi(\theta(J))$ for every $J\in \mathcal F_K$. Since $\varphi$ is continuous, it follows that
  \[
    \sum_{k\in K}\varphi(f(k)) = \lim_{J\in \mathcal{F}_K}\rho(J) = \varphi\left(\lim_{J\in \mathcal{F}_K}\theta(J)\right) =\varphi(x).
  \]
  (iii) It suffices to apply~(ii) to $\varphi = \{-y\}_{y\in \mf A}$.
  \par\medskip\noindent
  (iv) The statement follows immediately from~(i) and~(iii).
\end{proof}

\begin{lemma}\label{l_zeroext}
  Let $\mf A$ be an HTM, $I\subset K$, and $f(k) = 0$ for every $k\in K\setminus I$. Then the relations $\sm{\mf A}{k\in I}{f(k)}$ and $\sm{\mf A}{k\in K}{f(k)}$ are equivalent. If $\sm{\mf A}{k\in I}{f(k)}$, then $\sum_{k\in I}f(k) = \sum_{k\in K} f(k)$.  
\end{lemma}
\begin{proof}
  We can assume that $f(k)\in \mf A$ for all $k\in K$ because this condition is implied by both $\sm{\mf A}{k\in I}{f(k)}$ and $\sm{\mf A}{k\in K}{f(k)}$. Let  $\theta = \bigS_{k\in K} f(k)$ and $\rho = \bigS_{k\in I} f(k)$. Let $J = K\setminus I$ and $\Lambda = \mathcal{F}_I \times \mathcal{F}_J$. We endow $\Lambda$ with the product order: for $I', I''\in \mathcal{F}_I$ and $J',J''\in \mathcal{F}_J$,  we have $(I',J')\leq (I'',J'')$ if and only if $I'\subset I''$ and $J'\subset J''$.  Let $\fm{T}{\Lambda}$ be such that $T(I',J') = I'\cup J'$ for every $I'\in \mathcal{F}_I$ and $J'\in \mathcal{F}_J$. Clearly, $T$ is an isomorphism of ordered sets between $\Lambda$ and $\mathcal{F}_K$. Let $\tilde\rho\colon \Lambda\to \mf A$ be such that $\tilde\rho(I',J') = \rho(I')$ for every $I'\in \mathcal{F}_I$ and $J'\in \mathcal{F}_J$. Since $\tilde\rho = \theta\circ T$ by Proposition~\ref{p_finsum}(iv,v), we obtain
  \begin{multline}\nonumber
    \sm{\mf A}{k\in I}{f(k)} \Leftrightarrow \rho\mbox{ converges in $\mf A$ over $\mathcal{F}_I$} \Leftrightarrow \tilde\rho\mbox{ converges in $\mf A$ over $\Lambda$}\\ \Leftrightarrow \theta\mbox{ converges in $\mf A$ over $\mathcal{F}_K$}\Leftrightarrow \sm{\mf A}{k\in K}{f(k)}.
  \end{multline}
  If $\sm{\mf A}{k\in I}{f(k)}$, then we have
  \[
    \sum_{k\in I}f(k) = \lim_{I'\in \mathcal{F}_I}\rho(I') = \lim_{\lambda\in \Lambda}\tilde\rho(\lambda) = \lim_{K'\in \mathcal{F}_K}\theta(K') = \sum_{k\in K} f(k).
  \]
\end{proof}
 
\begin{proposition}
  Let $\mf A$ be an HTM, $K$ be a set, and $f(k) = 0$ for every $k\in K$. Then $\sm{\mf A}{k\in K}{f(k)}$ and $\sum_{k\in K}f(k) = 0$.
\end{proposition}
\begin{proof}
  The statement follows from Corollary~\ref{c_emptysum} and Lemma~\ref{l_zeroext}.
\end{proof}

\begin{proposition}\label{p_sum2union}
  Let $\mf A$ be an HTM and $K = I\cup J$, where $I\cap J = \varnothing$. Let $\sm{\mf A}{k\in I}{f(k)}$. If $\sm{\mf A}{k\in J}{f(k)}$, then $\sum_{k\in K} f(k) = \sum_{k\in I}f(k) + \sum_{k\in J}f(k)$. If $\mf A$ is an HTG and $\sm{\mf A}{k\in K}{f(k)}$, then $\sm{\mf A}{k\in J}{f(k)}$. 
\end{proposition}
\begin{proof}
  We can assume that $f(k)\in \mf A$ for all $k\in K$ (as $\sm{\mf A}{k\in I}{f(k)}$, this condition is implied by both $\sm{\mf A}{k\in J}{f(k)}$ and $\sm{\mf A}{k\in K}{f(k)}$).
  Let $g\colon K\to \mf A$ be such that $g(k) = f(k)$ for every $k\in I$ and $g(k) = 0$ for every $k\in J$. Let $h\colon K\to \mf A$ be such that $h(k) = f(k)$ for every $k\in J$ and $h(k) = 0$ for every $k\in I$. By Lemma~\ref{l_zeroext}, we have $\sum_{k\in I}f(k) = \sum_{k\in K}g(k)$. Suppose $\sm{\mf A}{k\in J}{f(k)}$. Then $\sum_{k\in J}f(k) = \sum_{k\in K}h(k)$ by Lemma~\ref{l_zeroext}. Since $f(k) = g(k) + h(k)$ for every $k\in K$, it follows from Proposition~\ref{p_sum}(i) that $\sum_{k\in K} f(k) = \sum_{k\in I}f(k) + \sum_{k\in J}f(k)$. Now suppose that $\mf A$ is an HTG and $\sm{\mf A}{k\in K}{f(k)}$. Since $h(k) = f(k) - g(k)$ for every $k\in K$, Proposition~\ref{p_sum}(iv) implies that $\sm{\mf A}{k\in K}{h(k)}$ and, hence, $\sm{\mf A}{k\in J}{f(k)}$ by Lemma~\ref{l_zeroext}.
\end{proof}

\begin{corollary}\label{c_sumelement}
  Let $\mf A$ be an HTM, $K$ be a set, $k_0\in K$, and $f(k_0)\in \mf A$. Let $\sm{\mf A}{k\in K\setminus\{k_0\}}{f(k)}$. Then $\sum_{k\in K}f(k) = f(k_0) + \sum_{k\in K\setminus\{k_0\}}f(k)$.
\end{corollary}
\begin{proof}
  By Propositions~\ref{p_finsumdef1} and~\ref{p_sumfin}, we have $\sum_{k\in \{k_0\}}f(k) = f(k_0)$. It remains to apply Proposition~\ref{p_sum2union} to $I = \{k_0\}$ and $J = K\setminus\{k_0\}$. 
\end{proof}

\begin{proposition}\label{p_sumfinunion}
  Let $\mf A$ be an HTM, $K$ be a set, and $\fm{\mb J}{I}$ be a finite partition of $K$ such that $\sm{\mf A}{k\in \mb J(i)}{f(k)}$ for every $i\in I$. Then $\sum_{k\in K}f(k) = \sum_{i\in I}\sum_{k\in\mb J(i)}f(k)$.
\end{proposition}
\begin{proof}
  The proof is by induction on the cardinality of $I$. If $I = \varnothing$, then $K = \varnothing$ and the statement is obviously true. Let $n\in\N$ and suppose the statement is true for all $I$ with $n-1$ elements. Let $I$ have $n$ elements. Let $\fm{g}{I}$ be such that $g(i) = \sum_{k\in\mb J(i)} f(k)$. Choose $i_0\in I$ and put $I' = I\setminus\{i_0\}$. Let $K' =\bigcup_{i\in I'}\mb J(i)$. As $I'$ is finite, it follows from Proposition~\ref{p_sumfin} that $\sm{\mf A}{i\in I'}{g(i)}$. By the induction hypothesis, we have $\sum_{k\in K'}f(k) = \sum_{i\in I'} g(i)$. Proposition~\ref{p_sum2union} applied to the representation $K = \mb J(i_0)\cup K'$ ensures that $\sum_{k\in K}f(k) = g(i_0) + \sum_{i\in I'}g(i)$. By Corollary~\ref{c_sumelement}, we conclude that $\sum_{k\in K}f(k) = \sum_{i\in I}g(i)$.
\end{proof}

\begin{proposition}\label{p_sumunionfin}
  Let $\mf A$ be an HTM, $K$ be a set, and $\sm{\mf A}{k\in K}{f(k)}$. Let  $\fm{\mb J}{I}$ be a partition of $K$ such that $\mb J(i)$ is finite for every $i\in I$. Then $\sum_{k\in K}f(k) = \sum_{i\in I}\sum_{k\in\mb J(i)}f(k)$.
\end{proposition}
\begin{proof}
  Let $\theta = \bigS_{k\in K} f(k)$ and $x = \sum_{k\in K}f(k)$. Let $U$ be a neighbourhood of $x$ in $\mf A$ and $K_0\in \mathcal{F}_K$ be such that $\theta(K') \in U$ whenever $K'\in \mathcal{F}_K$ and $K_0\subset K'$. By Proposition~\ref{p_sumfin}, we have $\sm{\mf A}{k\in\mb J(i)}{f(k)}$ for every $i\in I$. Let $\fm{g}{I}$ be such that $g(i) = \sum_{k\in \mb J(i)}f(k)$ for every $i\in I$ and let $I_0 = \set{i\in I}{K_0\cap\mb J(i)\neq\varnothing}$.  Since $K_0$ is finite, we have $I_0\in \mathcal{F}_I$. Let $I'\in \mathcal{F}_I$ be such that $I_0\subset I'$. Since $\mb J(i)$ is finite for every $i\in I$, the set $L = \bigcup_{i\in I'}\mb J(i)$ belongs to $\mathcal{F}_K$. It follows from Proposition~\ref{p_sumfinunion} applied to $\mb J|_{I'}$, $I'$, and $L$ in place of $\mb J$, $I$, and $K$ respectively that $\sum_{i\in I'}g(i) = \sum_{k\in L} f(k)$. By Proposition~\ref{p_sumfin}, this implies that $\finsum_{i\in I'}g(i)$ is equal to $\theta(L)$ and, hence, belongs to $U$ because $K_0\subset L$. This means that $\sum_{i\in I} g(i) = x$.
\end{proof}

\begin{proposition}\label{p_fubiniHTM}
  Let $\mf A$ be an HTM, $I$ and $J$ be sets, and $K = I\times  J$. Let one of the following conditions be satisfied:
  \begin{enumerate}
  \item [(a)] $I$ is finite and $\sm{\mf A}{j\in J}{f(i,j)}$ for every $i\in I$,
  \item [(b)] $J$ is finite and $\sm{\mf A}{i\in  I}{f(i,j)}$ for every $j\in J$.
  \end{enumerate}
  Then
  \begin{equation}
    \label{eq:fub}
      \sum_{i\in I}\sum_{j\in J}f(i,j) = \sum_{k\in K} f(k) = \sum_{j\in J}\sum_{i\in I}f(i,j).
    \end{equation}
  \end{proposition}
\begin{proof}
  Suppose (a) holds. Let $\fm{\mb J}{I}$ be such that $\mb J(i) = \{i\}\times J$ for every $i\in I$. Since $\{(i,j)\}_{j\in J}$ is a bijection between $J$ and $\mb J(i)$, Proposition~\ref{p_sumchange} ensures that $\sum_{k\in\mb J(i)}f(k) = \sum_{j\in J}f(i,j)$ for every $i\in I$. Proposition~\ref{p_sumfinunion} therefore yields the first equality in~(\ref{eq:fub}). 
  Let $\fm{\mb J'}{J}$ be such that $\mb J'(j) = I\times\{j\}$ for every $j\in J$. By Proposition~\ref{p_sumfin}, we have $\sm{\mf A}{i\in  I}{f(i,j)}$ for every $j\in J$. Since $\{(i,j)\}_{i\in I}$ is a bijection between $I$ and $\mb J'(j)$, Proposition~\ref{p_sumchange} ensures that $\sum_{k\in\mb J'(j)}f(k) = \sum_{i\in I}f(i,j)$ for every $j\in J$. The second equality in~(\ref{eq:fub}) therefore follows from Proposition~\ref{p_sumunionfin} applied to $\mb J'$ and $J$ in place of $\mb J$ and $I$ respectively. The proof for the case when (b) holds is analogous. 
\end{proof}

\begin{proposition}\label{p_disjointsum}
  Let $\mf A$ be an HTM, $\fm{\mb J}{I}$ be a family of finite sets, and $K = \bigsqcup_{i\in I}\mb J(i)$. Let $f$ be such that $\sm{\mf A}{k\in K}{f(k)}$. Then
    \[
      \sum_{k\in K} f(k) = \sum_{i\in I}\sum_{j\in\mb J(i)}f(i,j).
    \]
\end{proposition}
\begin{proof}
  Let $\fm{\mb J'}{I}$ be such that $\mb J'(i) = \{i\}\times \mb J(i)$ for every $i\in I$. Then $\mb J'$ is a partition of $K$. By Proposition~\ref{p_sumfin}, we have $\sm{\mf A}{j\in \mb J(i)}{f(i,j)}$ for every $i\in I$. As $\{(i,j)\}_{j\in \mb J(i)}$ is a bijection between $\mb J(i)$ and $\mb J'(i)$, Proposition~\ref{p_sumchange} ensures that $\sum_{k\in\mb J'(i)}f(k) = \sum_{j\in \mb J(i)}f(i,j)$ for every $i\in I$. Hence, the statement follows from Proposition~\ref{p_sumunionfin} applied to $\mb J'$ in place of~$\mb J$. 
\end{proof}

\begin{proposition}\label{p_sumunion}
  Let $\mf A$ be a regular HTM, $K$ be a set, and $\sm{\mf A}{k\in K}{f(k)}$. Let  $\fm{\mb J}{I}$ be a partition of $K$ such that $\sm{\mf A}{k\in \mb J(i)}{f(k)}$ for every $i\in I$. Then $\sum_{k\in K}f(k) = \sum_{i\in I}\sum_{k\in\mb J(i)}f(k)$.
\end{proposition}
\begin{proof}
  Let $x = \sum_{k\in K}f(k)$ and $U$ be a closed neighbourhood of $x$ in $\mf A$. Let $K_0\in \mathcal{F}_K$ be such that $\finsum_{k\in K'}f(k) \in U$ whenever $K'\in \mathcal{F}_K$ and $K_0\subset K'$. Let $\fm{g}{I}$ be such that $g(i) = \sum_{k\in \mb J(i)}f(k)$ for every $i\in I$ and let $I_0 = \set{i\in I}{K_0\cap\mb J(i)\neq\varnothing}$.  Since $K_0$ is finite, we have $I_0\in \mathcal{F}_I$. Let $I'\in \mathcal{F}_I$ be such that $I_0\subset I'$. Let $L = \bigcup_{i\in I'}\mb J(i)$. It follows from Proposition~\ref{p_sumfinunion} applied to $\mb J|_{I'}$, $I'$, and $L$ in place of $\mb J$, $I$, and $K$ respectively that $\sm{\mf A}{k\in L}{f(k)}$ and $\sum_{i\in I'}g(i) = \sum_{k\in L} f(k)$. Since $K_0\subset L$ and $U$ is closed, we have $\sum_{k\in L} f(k)\in U$. By Proposition~\ref{p_sumfin}, this implies that $\finsum_{i\in I'}g(i)$ belongs to $U$. In view of the regularity of $\mf A$, this means that $\sum_{i\in I} g(i) = x$.  
\end{proof}

\begin{proposition}\label{p_fubiniRTS}
  Let $\mf A$ be a regular HTM, $I$ and $J$ be sets, and $K = I\times  J$. Let $\sm{\mf A}{k\in K}{f(k)}$ and $x = \sum_{k\in K} f(k)$.
  \begin{enumerate}
  \item [(i)] Let $\sm{\mf A}{j\in  J}{f(i,j)}$ for every $i\in I$. Then $x = \sum_{i\in I}\sum_{j\in J}f(i,j)$.
  \item [(ii)] Let $\sm{\mf A}{i\in  I}{f(i,j)}$ for every $j\in J$. Then $x = \sum_{j\in J}\sum_{i\in I}f(i,j)$.
  \end{enumerate}
\end{proposition}
\begin{proof}
  (i) Let $\fm{\mb J}{I}$ be such that $\mb J(i) = \{i\}\times J$ for every $i\in I$. Since $\{(i,j)\}_{j\in J}$ is a bijection between $J$ and $\mb J(i)$, Proposition~\ref{p_sumchange} ensures that $\sum_{k\in\mb J(i)}f(k) = \sum_{j\in J}f(i,j)$ for every $i\in I$. The statement hence follows from Proposition~\ref{p_sumunion}. 
  \par\medskip\noindent
  (ii) Let $\fm{\mb J}{J}$ be such that $\mb J(j) = I\times\{j\}$ for every $j\in J$. Since $\{(i,j)\}_{i\in I}$ is a bijection between $I$ and $\mb J(j)$, Proposition~\ref{p_sumchange} ensures that $\sum_{k\in\mb J(j)}f(k) = \sum_{i\in I}f(i,j)$ for every $j\in J$. The statement therefore follows from Proposition~\ref{p_sumunion} applied to $J$ in place of $I$.
\end{proof}

\begin{remark}
  When $\mf A$ is a complete uniform HTM, Proposition~\ref{p_fubiniRTS} coincides with Lemma~2.5 in~\cite{FoxMorales1983}. 
\end{remark}

The next example shows that the regularity assumption cannot be dropped in Proposition~\ref{p_fubiniRTS}.  

\begin{example}\label{e_nonreg}
  Let $\mathcal{T}$ be the set of all subsets $O$ of $\R^2$ such that
  \[
    \mbox{if $(x,y)\in O$, then $(a,x]\times (b,y)\subset O$ for some $a<x$ and $b<y$.} 
  \]
  It is easy to see that $\mathcal{T}$ is a Hausdorff topology on $\R^2$ that makes the ordinary addition in $\R^2$ continuous. We let $\mf S$ denote the HTM obtained by equipping $\R^2$ with $\mathcal{T}$. Let $K = \N\times \N$ and $f\colon K\to\R^2$ be such that $f(1,1) = (1/2,1/2)$, $f(i,1) = (2^{-i},0)$ for $i>1$, $f(1,j) = (0,2^{-j})$ for $j>1$, and $f(i,j) = 0$ whenever $i>1$ and $j>1$. Then $\sum^{\mf S}_{k\in K} f(k) = (1,1)$. At the same time, the sequence $\sum^{\mf S}_{i \leq n} \sum^{\mf S}_{j\in\N} f(i,j) = (1-2^{-n},1)$ does not converge in $\mf S$ as $n\to\infty$. Thus, statement~(i) of Proposition~\ref{p_fubiniRTS} breaks down in this case. It follows from Proposition~\ref{p_fubiniRTS} (and can be easily verified directly) that $\mf S$ is not regular.
\end{example}

\begin{remark}\label{r_paratop}
  A group that is a topological monoid is called a paratopological group. The HTM $\mf S$ in Example~\ref{e_nonreg} is actually a paratopological group. It is a slight modification of a nonregular paratopological group constructed by Ravsky~\cite[Example~1.8]{Ravsky2001}.
\end{remark}

\begin{proposition}\label{p_sumsubsgr}
  Let $\mf A$ be an HTM, $\mf B$ be a topological submonoid of $\mf A$. Suppose $f(k)\in\mf B$ for all $k\in K$. If $\sm{\mf B}{k\in K}{f(k)}$, then $\sm{\mf A}{k\in K}{f(k)}$ and $\sum^{\mf A}_{k\in K} f(k) = \sum^{\mf B}_{k\in K}f(k)$. If $\sm{\mf A}{k\in K}{f(k)}$ and $\sum^{\mf A}_{k\in K} f(k)$ belongs to $\mf B$, then $\sm{\mf B}{k\in K}{f(k)}$. 
\end{proposition}
\begin{proof}
  Let $\sm{\mf B}{k\in K}{f(k)}$. Then  $\sm{\mf A}{k\in K}{f(k)}$ and $\sum^{\mf A}_{k\in K} f(k) = \sum^{\mf B}_{k\in K}f(k)$ by Proposition~\ref{p_sum}(ii) because the identity map of $\mf B$ is a continuous homomorphism from $\mf B$ to $\mf A$. Suppose now that $\sm{\mf A}{k\in K}{f(k)}$ and $x = \sum^{\mf A}_{k\in K} f(k)$ belongs to $\mf B$. Let $\theta = \bigS^{\mf A}_{k\in K}f(k)$. Then $\theta$ converges to $x$ in $\mf A$ over $\mathcal F_K$. Let $U$ be a neighbourhood of $x$ in $\mf B$. Then there is a neighbourhood $O$ of $x$ in $\mf A$ such that $U = O\cap \mf B$. Let $J\in \mathcal F_K$ be such that $\theta(K')\in O$ for every $K'\in\mathcal F_K$ such that $K'\supset J$. As $\theta(K')\in \mf B$ for every $K'\in\mathcal F_K$, it follows that $\theta(K')\in U$ for every $K'\in\mathcal F_K$ such that $K'\supset J$. Thus, $\theta$ converges in $\mf B$ over $\mathcal F_K$. Since $\theta = \bigS^{\mf B}_{k\in K}f(k)$, this means that $\sm{\mf B}{k\in K}{f(k)}$.
\end{proof}

\begin{proposition}\label{p_seqclosedsum}
  Let $\mf A$ be an HTM, $\mf B$ be a sequentially closed submonoid of $\mf A$, $K$ be a countable set, $f(k)\in \mf B$ for all $k\in K$, and $\sm{\mf A}{k\in K}{f(k)}$. Then $\sum^{\mf A}_{k\in K}f(k)$ belongs to $\mf B$. If, in addition, $\mf B$ is a topological submonoid of $\mf A$, then $\sm{\mf B}{k\in K}{f(k)}$. 
\end{proposition}
\begin{proof}
  Let $\theta = \bigS^{\mf A}_{k\in K} f(k)$ and $x = \sum^{\mf A}_{k\in K}f(k)$. Then $\im\theta\subset\mf B$ and $\theta$ converges to $x$ in $\mf A$ over $\mathcal F_K$. Since $\mf B$ is sequentially closed in $\mf A$ and $\mathcal F_K$ is countable, this implies that $x\in\mf B$. If $\mf B$ is a topological submonoid of $\mf A$, then Proposition~\ref{p_sumsubsgr} implies that $\sm{\mf B}{k\in K}{f(k)}$.
\end{proof}

\begin{proposition}\label{p_seqcompl}
  Let $\mf A$ be a sequentially complete HTG and $\sm{\mf A}{k\in K}{f(k)}$. Then $\sm{\mf A}{k\in J}{f(k)}$ for every countable $J\subset K$.
\end{proposition}
\begin{proof}
  Let $\theta = \bigS_{k\in K} f(k)$ and $\rho = \bigS_{k\in J} f(k)$. Let $U$ be a neighbourhood of the origin in $\mf A$. Since $\theta$ is a Cauchy net in $\mf A$ over $\mathcal{F}_K$, there is $K_0\in \mathcal{F}_K$ such that $\theta(K') - \theta(K'')\in U$ for every $K',K''\in \mathcal{F}_K$ such that $K_0\subset K'\cap K''$. Let $J_0 = J\cap K_0$ and $L = K_0\setminus J$. Let $J',J''\in \mathcal{F}_J$ be such that $J_0\subset J'\cap J''$. Let $K' = J'\cup L$ and $K'' = J''\cup L$. Then $K_0\subset K'\cap K''$. As $J'\cap L = J''\cap L = \varnothing$, it follows from Proposition~\ref{p_finsum}(iv) that $\rho(J') - \rho(J'') = \theta(K') - \theta(K'')$ and, hence, $\rho(J')-\rho(J'')\in U$. Thus, $\rho$ is a Cauchy net in $\mf A$ over $\mathcal{F}_J$. Since $\mf A$ is sequentially complete and $\mathcal{F}_J$ is countable, we conclude that $\sm{\mf A}{k\in J}{f(k)}$.
\end{proof}

Given a monoid $\mf A$, we let $\mf A_d$ denote the HTM obtained by endowing $\mf A$ with the discrete topology.

\begin{lemma}\label{l_discretesum}
  Let $\mf A$ be a monoid, $K$ be a finite set, and $f(k)\in \mf A$ for all $k\in K$. Then $\sum_{k\in K}^{\mf A_d} f(k) = \finsum_{k\in K}^{\mf A}f(k)$.    
\end{lemma}
\begin{proof}
  Since $\mf A$ and $\mf A_d$ coincide as monoids, we have $\finsum_{k\in K}^{\mf A}f(k) = \finsum_{k\in K}^{\mf A_d}f(k)$. The statement therefore follows from Proposition~\ref{p_sumfin}.
\end{proof}

\begin{proposition}\label{p_fubini2}
  Let $\mf A$ be a monoid, $I$ and $J$ be finite sets, $K = I\times J$, and $ f(k)\in \mf A$ for all $k\in K$. Then
  \[
    \finsum_{k\in K} f(k) = \finsum_{i\in I}\finsum_{j\in J}f(i,j) = \finsum_{j\in J}\finsum_{i\in I}f(i,j).
  \]  
\end{proposition}
\begin{proof}
  Since $I$ is finite, Proposition~\ref{p_sumfin} implies that $\sm{\mf A_d}{i\in I}{f(i,j)}$ for every $j\in J$. As $J$ is finite, Proposition~\ref{p_fubiniHTM} ensures that
  \[
    {\sum_{k\in K}}^{\mf A_d}f(k) = {\sum_{i\in I}}^{\mf A_d}{\sum_{j\in J}}^{\mf A_d}f(i,j) = {\sum_{j\in J}}^{\mf A_d}{\sum_{i\in I}}^{\mf A_d}f(i,j).
  \]
  The statement now follows from Lemma~\ref{l_discretesum}.
\end{proof}

\begin{proposition}\label{p_finsumunion}
  Let $\mf A$ be a monoid, $K$ be a finite set, and $f(k)\in \mf A$ for all $k\in K$. Let $\fm{\mb J}{I}$ be a finite partition of $K$.  Then
  \[
    \finsum_{k\in K}f(k) = \finsum_{i\in I}\finsum_{k\in\mb J(i)}f(k).
  \]
\end{proposition}
\begin{proof}
  By Proposition~\ref{p_sumfin}, we have  $\sm{\mf A_d}{k\in \mb J(i)}{f(k)}$ for every $i\in I$. Since
  \[
    {\sum_{k\in K}}^{\mf A_d}f(k) = {\sum_{i\in I}}^{\mf A_d}{\sum_{k\in\mb J(i)}}^{\mf A_d}f(k)
  \]
   by Proposition~\ref{p_sumfinunion}, the statement follows from Lemma~\ref{l_discretesum}.
\end{proof}

\begin{proposition}\label{p_findisjointsum}
  Let $\mf A$ be a monoid, $\fm{\mb J}{I}$ be a finite family of finite sets, $K = \bigsqcup_{i\in I}\mb J(i)$, and $f(k)\in \mf A$ for all $k\in K$. Then
    \[
      \finsum_{k\in K} f(k) = \finsum_{i\in I}\finsum_{j\in\mb J(i)}f(i,j).
    \]
\end{proposition}
\begin{proof}
  By Proposition~\ref{p_sumfin}, we have $\sm{\mf A_d}{k\in K}{f(k)}$. By Proposition~\ref{p_disjointsum}, we conclude that
   \[
      {\sum_{k\in K}}^{\mf A_d} f(k) = {\sum_{i\in I}}^{\mf A_d}{\sum_{j\in\mb J(i)}}^{\mf A_d}f(i,j).
    \]
    The statement now follows from Lemma~\ref{l_discretesum}.
\end{proof}

\section{Proof of Proposition~\ref{l_semi}}
\label{app_C}

\begin{lemma}\label{l_sigmaadd}
  Let $\mf A$ be an HTM, $\mathcal Q$ be a semi-ring, and $\nu$ be an $\mf A$-valued additive function such that $D_\nu=\kappa(\mathcal Q)$. Then $\nu$ is $\sigma$-additive if and only if $\nu|_{\mathcal Q}$ is $\sigma$-additive.
\end{lemma}
\begin{proof}
  Let $\mu = \nu|_{\mathcal Q}$. Obviously, if $\nu$ is $\sigma$-additive, then so is $\mu$. Let $\mu$ be $\sigma$-additive, $A\in D_\nu$, and $\mb A\colon I\to D_\nu$ be a countable partition of $A$. By Proposition~\ref{p_semiring}, there are $\fm{\mb J}{I}$ and $\fm{\mb B}{I}$ such that $\mb B(i)\colon\mb J(i)\to \mathcal Q$ is a finite partition of $\mb A(i)$ for every $i\in I$. Let $K = \bigsqcup_{i\in I}\mb J(i)$ and $\mb C\colon K\to\mathcal Q$ be such that $\mb C(i,j) = \mb B(i|j)$ for every $i\in I$ and $j\in \mb J(i)$. It follows from Lemma~\ref{p_disjpart} that $\mb C$ is a partition of $A$. By Proposition~\ref{p_semiring}, there is a finite partition $\mb E\colon \Lambda\to\mathcal Q$ of $A$. Let $\fm{\mb F}{K\times\Lambda}$ be such that $\mb F(k,\lambda) = \mb C(k)\cap\mb E(\lambda)$ for every $k\in K$ and $\lambda\in \Lambda$. The family $\{\mb F(k,\lambda)\}_{k\in K}$ is a countable partition of $\mb E(\lambda)$ for every $\lambda\in\Lambda$. Hence, $\sum_{k\in K} \mu(\mb F(k,\lambda)) = \mu(\mb E(\lambda))$ for every $\lambda\in \Lambda$ by the $\sigma$-additivity of $\mu$. Since $\Lambda$ is finite, the additivity of $\nu$ and Proposition~\ref{p_fubiniHTM} imply that
  \[
    \nu(A) = \sum_{\lambda\in\Lambda}\mu(\mb E(\lambda)) = \sum_{\lambda\in\Lambda}\sum_{k\in K} \mu(\mb F(k,\lambda)) = \sum_{k\in K}\sum_{\lambda\in\Lambda} \mu(\mb F(k,\lambda)).
  \]
  The family $\{\mb F(k,\lambda)\}_{\lambda\in\Lambda}$ is a finite partition of $\mb C(k)$ for every $k\in K$. By the additivity of $\mu$, it follows that $\nu(A) = \sum_{k\in K} \mu(\mb C(k))$.
  Since $\mb J(i)$ is finite for every $i\in I$, we have $\nu(\mb A(i)) = \sum_{j\in \mb J(i)}\mu(\mb B(i|j))$ by the additivity of $\nu$ and, hence, 
  \[
    \nu(A) = \sum_{i\in I}\sum_{j\in \mb J(i)}\mu(\mb B(i|j)) = \sum_{i\in I}\nu(\mb A(i))
  \]
  by Proposition~\ref{p_disjointsum}. This means that $\nu$ is $\sigma$-additive. 
\end{proof}

\begin{proof}[Proof of Proposition~\ref{l_semi}]
  Let $A\in \kappa(D_\mu)$ and let $\fm{\mb A}{I}$ and $\fm{\mb A'}{J}$ be finite partitions of $A$ into elements of $D_\mu$. Let $\fm{\mb B}{I\times J}$ be such that $\mb B(i,j)=\mb A(i)\cap \mb A'(j)$ for every $i\in I$ and $j\in J$. The family $\{\mb B(i,j)\}_{j\in J}$ is a finite partition of $\mb A(i)$ for every $i\in I$ and the family $\{\mb B(i,j)\}_{i\in I}$ is a finite partition of $\mb A'(j)$ for every $j\in J$. Hence, Proposition~\ref{p_fubini2} and the additivity of $\mu$ imply that
  \[
    \finsum_{i\in I} \mu(\mb A(i)) = \finsum_{i\in I} \finsum_{j\in J} \mu(\mb B(i,j))= \finsum_{j\in J}\finsum_{i\in I} \mu(\mb B(i,j))=\finsum_{j\in J}\mu(\mb A'(j)).
  \]
  Let $\nu$ be a map such that $D_\nu = \kappa(D_\mu)$ and
  $\nu(A)=\finsum_{i\in I}\mu(\mb A(i))$
  for every $A\in\kappa(D_\mu)$ and every finite partition $\fm{\mb A}{I}$ of $A$ into elements of $D_\mu$ (such partitions exist by Proposition~\ref{p_semiring}). As shown above, the sum defining $\nu(A)$ does not depend on the choice of a partition and, therefore, such $\nu$ exists. We obviously have $\nu(A) = \mu(A)$ for every $A\in D_\mu$.

  Let $A\in D_\nu$ and $\mb A\colon I\to D_\nu$ be a finite partition of $A$. By Proposition~\ref{p_semiring}, there are $\fm{\mb J}{I}$ and $\fm{\mb B}{I}$ such that $\mb B(i)\colon \mb J(i)\to D_\mu$ is a finite partition of $\mb A(i)$ for every $i\in I$. Let $K=\bigsqcup_{i\in I} \mb J(i)$ and $\mb C\colon K\to D_\mu$ be such that $\mb C(i,j)=\mb B(i|j)$ for every $i\in I$ and $j\in \mb J(i)$. Since  $\mb C$ is a partition of $A$ by Lemma~\ref{p_disjpart}, it follows from Proposition~\ref{p_findisjointsum} that
  \[
    \nu(A) = \finsum_{k\in K} \mu(\mb C(k)) = \finsum_{i\in I}\finsum_{j\in \mb J(i)}\mu(\mb B(i|j)) = \finsum_{i\in I}\nu(\mb A(i)).
  \]
  This means that $\nu$ is an $\mf A$-valued additive function, and the existence of $\nu$ is proved. The uniqueness of $\nu$ is obvious.
  If $\mf A$ is an HTM and $\mu$ is $\sigma$-additive, then $\nu$ is $\sigma$-additive by Lemma~\ref{l_sigmaadd}.
\end{proof}

\section{Some counterexamples}
\label{app_counter}

Given a finite set $A$, we let $\card A$ denote the number of its elements.

\subsection{\texorpdfstring{An $\mf A$-valued premeasure $\mu$ such that $\Sigma_{\mf A}(\mu) \neq \mathcal{L}_{\mf A}(\mu)$}{An 𝔄-valued premeasure μ such that Σ\_𝔄(μ) ≠ ℒ\_𝔄(μ)}} \label{s_example1}
Here, we construct an example showing that the regularity assumption in Theorem~\ref{ll16} cannot be dropped.

Let $I$ be an infinite countable set. We define the $\R$-vector space $\mf A$ by the equality $\mf A = \R^I$. Let $\mathcal{P}$ be the set of all subsets of $I$ and $\chi\colon \mathcal{P}\to \mf A$ be defined as follows: for every $A\in \mathcal{P}$, $\chi(A)$ is a map from $I$ to $\R$ that is equal to unity on $A$ and vanishes on $I\setminus A$. We define the map $\delta\colon I\to \mf A$ by setting $\delta(i) = \chi(\{i\})$ for every $i\in I$.

Let $J$ be an infinite set and $\fm{\mb A}{J}$ be a partition of $I$ such that $\mb A(j)$ is an infinite set for every $j\in J$. Let $S_1 = \im \delta$ and $S_2 = \im (\chi\circ\mb A)$. Clearly, $S_1$ and $S_2$ are disjoint subsets of $\mf A$. Let $S = S_1\cup S_2$.

\begin{lemma}\label{l_linind}
  $S$ is a linearly independent subset of $\mf A$.
\end{lemma}
\begin{proof}
  Let $T\subset S$ be a finite set and $\alpha\colon T\to \R$. Suppose
  \begin{equation}
    \label{eq:lincomb}
    \finsum_{t\in T} \alpha(t)t = 0.  
  \end{equation}
  Let $T_1 = T\cap S_1$ and $T_2 = T\cap S_2$. Let $I_1 = \delta^{-1}(T) = \delta^{-1}(T_1)$. Then $I_1$ is a finite subset of $I$. Let $t_0\in T_2$ and let $j_0\in J$ be such that $t_0 =\chi(\mb A(j_0))$. Since $\mb A(j_0)$ is infinite, there is $i_0\in \mb A(j_0)$ such that $i_0\notin I_1$. Let $t\in T_2$ be such that $t\neq t_0$. Let $j\in J$ be such that $t =\chi(\mb A(j))$. As $j\neq j_0$, we have $\mb A(j)\cap \mb A(j_0) = \varnothing$ and, hence, $i_0\notin \mb A(j)$. It follows that $t(i_0) = 0$. Let $t\in T_1$ and let $i\in I_1$ be such that $t  = \delta(i)$. Since $i_0\notin I_1$, we have $i\neq i_0$ and, therefore, $t(i_0) = 0$. Thus, $t(i_0) = 0$ for all $t\in T$ such that $t\neq t_0$. As $t_0(i_0) = 1$, we have $\finsum_{t\in T}\alpha(t)t(i_0) = \alpha(t_0)$ and it follows from~(\ref{eq:lincomb}) that $\alpha(t_0) = 0$. This proves that $\alpha(t) = 0$ for every $t\in T_2$. By~(\ref{eq:lincomb}), we conclude that $\finsum_{t\in T_1}\alpha(t)t = 0$. Let $t_1\in T_1$ and $i_1\in I_1$ be such that $t_1 = \delta(i_1)$. Let $t\in T_1$ be such that $t\neq t_1$ and let $i\in I_1$ be such that $t = \delta(i)$. As $i\neq i_1$, we have $t(i_1) = 0$. This implies that $\finsum_{t\in T_1}\alpha(t)t(i_1) = \alpha(t_1)t_1(i_1) = \alpha(t_1)$ and, hence, $\alpha(t_1) = 0$. Thus, $\alpha(t) = 0$ for all $t\in T_1$ and, therefore, for all $t\in T$. This means that $S$ is linearly independent.     
\end{proof}

\begin{lemma}\label{l_phi}
  There is a linear functional $\phi\colon\mf A\to\R$ such that $\phi(\delta(i)) = 1$ for every $i\in I$ and $\phi(\chi(\mb A(j))) = -1$ for every $j\in J$.  
\end{lemma}
\begin{proof}
  In view of Lemma~\ref{l_linind}, Zorn's lemma ensures that there is a maximal linearly independent subset $S'$ of $\mf A$ such that $S\subset S'$. For every $f\colon S'\to \R$, there is a unique linear functional $\phi\colon\mf A\to \R$ such that $\phi|_{S'} = f$. If we choose $f$ so that $f(s) = 1$ for all $s\in S_1$ and $f(s) = -1$ for all $s\in S_2$, then $\phi$ has the required properties. 
\end{proof}

Let $\phi$ be a fixed linear functional satisfying the conditions of Lemma~\ref{l_phi}. Let $\fm{N}{\mf A}$ be such that $N(x) = \set{y}{y\in\mf A\mbox{ and }\phi(y-x)\geq 0}$ for every $x\in\mf A$. Clearly, $N(0)$ is a submonoid of $\mf A$ and $N(x) = x + N(0)$ for every $x\in\mf A$. We obviously have
\begin{equation}
  \label{eq:NsubsetN}
  N(y)\subset N(x)\mbox{ for every $x\in\mf A$ and $y\in N(x)$}.
\end{equation}
Let $\tau_0$ denote the product topology on $\mf A$. We define the topology $\tau$ on $\mf A$ by the equality
\[
  \tau = \set{O}{\mbox{for every $x\in O$, there is $U\in\tau_0$ such that $x\in U$ and $U\cap N(x)\subset O$}}.
\]
It follows from~\ref{eq:NsubsetN} that the set $U\cap N(x)$ belongs to $\tau$ for every $U\in\tau_0$ and $x\in\mf A$. In particular, $N(x)\in\tau$ for every $x\in\mf A$. Since $\tau$ is stronger than $\tau_0$, it is Hausdorff. 

\begin{lemma}\label{l_addcont}
  The addition on $\mf A$ is continuous with respect to $\tau$.
\end{lemma}
\begin{proof}
  Let $x_1,x_2\in \mf A$, $x = x_1 + x_2$, and $O\in \tau$ be such that $x\in O$. Let $U\in\tau_0$ be such that $x\in U$ and $U\cap N(x)\subset O$. Since the addition on $\mf A$ is continuous with respect to $\tau_0$, there are $U_1,U_2\in\tau_0$ such that $x_1\in U_1$, $x_2\in U_2$, and $y_1+y_2\in U$ for every $y_1\in U_1$ and $y_2\in U_2$. Let $O_1 = U_1\cap N(x_1)$ and $O_2 = U_2\cap N(x_2)$. Then $O_1,O_2\in\tau$, $x_1\in O_1$, and $x_2\in O_2$. Let $y_1\in O_1$, $y_2\in O_2$, and $y = y_1 + y_2$. Clearly, $y\in U$. As $\phi(y-x) = \phi(y_1-x_1) + \phi(y_2-x_2)\geq 0$, we have $y\in N(x)$ and, hence, $y\in O$. This means that the addition on $\mf A$ is continuous with respect to $\tau$.
\end{proof}

We endow $\mf A$ with the topology $\tau$. By Lemma~\ref{l_addcont}, this makes $\mf A$ into an HTM (in fact, $\mf A$ is a Hausdorff paratopological group, see Remark~\ref{r_paratop}). For every $x\in\mf A$, the sets of the form $U\cap N(x)$, where $U\in\tau_0$ and $x\in U$, constitute a basis of neighbourhoods of $x$ in $\mf A$.

\begin{lemma}\label{l_sumdelta}
  Let $I'\subset I$. Then $\sm{\mf A}{i\in I'}{\delta(i)}$ and $\sum_{i\in I'}\delta(i) = \chi(I')$.
\end{lemma}
\begin{proof}
  Let $\mf A_0$ denote the topological vector space obtained by endowing $\mf A$ with the topology $\tau_0$. Let $x = \chi(I')$. It is straightforward to verify that $\sm{\mf A_0}{i\in I'}{\delta(i)}$ and $\sum^{\mf A_0}_{i\in I'}\delta(i) = x$. Let $O$ be a neighbourhood of $x$ in $\mf A$. Then $U\cap N(x)\subset O$ for some $U\in\tau_0$ such that $x\in U$. Let $K\in\mathcal F_{I'}$ be such that $\card K \geq \phi(x)$ and $\finsum_{i\in K'}\delta(i) \in U$ for every $K'\in\mathcal F_{I'}$ such that $K\subset K'$. Let $K'\in\mathcal F_{I'}$ be such that $K\subset K'$ and let $y = \finsum_{i\in K'}\delta(i)$. By the definition of $\phi$, we have $\phi(y) = \finsum_{i\in K'}\phi(\delta(i)) = \card{K'}$ and, hence, $\phi(y-x) = \card{K'}-\phi(x)\geq 0$. Thus, $y\in N(x)$ and, therefore, $y\in O$. This means that $\sm{\mf A}{i\in I'}{\delta(i)}$ and $\sum_{i\in I'}\delta(i) = x$.
\end{proof}

Let $\mu = \bigS^{\mf A} \delta$. By Proposition~\ref{p_Spremeasure}, $\mu$ is an $\mf A$-valued premeasure. It follows from~(\ref{eq:LAf}), Lemma~\ref{l_LL}, and Lemma~\ref{l_sumdelta} that $\mathcal{L}_{\mf A}(\mu) = \mathcal{P}$. We claim that $\Sigma_{\mf A}(\mu) \neq \mathcal{L}_{\mf A}(\mu)$. Suppose the contrary and let $\nu = \as[\mf A]\mu$. By Proposition~\ref{p_assS} and Lemma~\ref{l_sumdelta}, we have $\nu(I') = \chi(I')$ for every $I'\subset I$. Let $x = \nu(I)$. The $\sigma$-additivity of $\nu$ implies that $x = \sum_{j\in J} \chi(\mb A(j))$. Let $K\in \mathcal{F}_J$ be such that $\finsum_{j\in K'} \chi(\mb A(j))\in N(x)$ for every $K'\in \mathcal{F}_J$ such that $K\subset K'$. Since $J$ is infinite, there is $K'\in \mathcal{F}_J$ such that $K\subset K'$ and $\card{K'} + \phi(x) > 0$. Let $y = \finsum_{j\in K'} \chi(\mb A(j))$. By the definition of $\phi$, we have $\phi(y) = -\card{K'}$. This implies that $\phi(y - x) < 0$  and, therefore, $y\notin N(x)$. We thus arrive at a contradiction and our claim is proved.

\begin{remark}
  It follows from Theorem~\ref{ll16} that $\mf A$ is not regular. It can be shown directly that, for every $x\in\mf A$,  there are no disjoint neighbourhoods of $x$ and the closed set $\mf A\setminus N(x)$.
\end{remark}

\subsection{\texorpdfstring{Inclusion $\mathcal{S}_{\mf A}[\mu]\subset D_\nu$ in Theorem~\ref{t_ext} can be proper}{Inclusion 𝒮\_𝔄[μ] ⊂ D\_ν in Theorem~\ref{t_ext} can be proper}}\label{s_example2}
Let $\mathcal{B}$ denote the Borel $\sigma$-ring of
$\R$ and $X$ be a countable subset of $\R$ that is unbounded from above. Let $\mathcal{Q} = \set{A}{A\in \mathcal{B}\mbox{ and }A\setminus X\mbox{ is bounded from above in }\R}$. Clearly, $\mathcal{Q}$ is a $\delta$-ring and $\sigma(\mathcal{Q}) = \mathcal{B}$. Let
\[
  I = \set{T}{T\subset\R\mbox{ and $T\cap A$ is a finite set for every $A\in \mathcal{Q}$}}.
\]
Obviously, elements of $I$ are countable sets.

\begin{lemma}\label{l_infinite}
  Let $A\in \mathcal{B}\setminus\mathcal{Q}$. Then there is $T\in I$ such that $T\subset A$ and $T$ is infinite.
\end{lemma}
\begin{proof}
  The definition of $\mathcal{Q}$ implies that $A\setminus X$ is unbounded from above in $\R$. Hence, there is $\alpha\colon\N\to A\setminus X$ such that $\alpha(n)\geq n$ for every $n\in\N$. Let $T=\im\alpha$. Clearly, $T\subset A$ and $T$ is infinite. If $B\in\mathcal Q$, then $B\setminus X$ is bounded from above, and the inequality $\alpha(n)\geq n$ implies that $T\cap B$ is finite. Thus $T\in I$.
\end{proof}

Let $\mf A$ denote the $\R$-vector space $\R^I$ endowed with the product topology. Then $\mf A$ is a complete Hausdorff topological vector space over $\R$ and, in particular, a complete HTG. Let $\nu\colon\mathcal{Q}\to \mf A$ be such that $\nu(A|T) = \card(T\cap A)$ for every $A\in\mathcal{Q}$ and $T\in I$. 

\begin{lemma}
  $\nu$ is an $\mf A$-valued measure.
\end{lemma}
\begin{proof}
  Let $A\in\mathcal{Q}$ and $\mb A\colon K\to \mathcal{Q}$ be a countable partition of $A$. We show that
  \begin{equation}
    \label{eq:sumnuAT}
    \nu(A|T) = {\sum_{k\in K}}^\R \nu(\mb A(k)|T),\quad T\in I.
  \end{equation}
  Let $T\in I$ and $f\colon T\to \R$ be such that $f(t) = 1$ for every $t\in T$. Then $\nu(A|T) = (\bigS^\R f)(T\cap A)$ for every $A\in \mathcal{Q}$. Equality~(\ref{eq:sumnuAT}) is therefore ensured by Proposition~\ref{p_Spremeasure}.
  It follows from~(\ref{eq:sumnuAT}) and the definition of product topology that $\nu(A) = \sum_{k\in K}^{\mf A} \nu(\mb A(k))$. This means that $\nu$ is $\sigma$-additive and, hence,  is an $\mf A$-valued premeasure.

  Let $\nu'$ be an $\mf A$-valued premeasure such that $\nu'\asymp \nu$. Suppose there is $A\in D_{\nu'}$ such that $A\notin\mathcal{Q}$. Since $\q{\nu'} = \q\nu = \mathcal{B}$, we have $A\in\mathcal{B}$. By Lemma~\ref{l_infinite}, there is $T\in I$ such that $T\subset A$ and $T$ is infinite. As $T$ is countable, we have $T\in\mathcal{B}$ and, hence, $T\in D_{\nu'}$ by Lemma~\ref{ll1}(iii). Since $\{t\}\in\mathcal{B}$ for every $t\in \R$, it follows from Lemma~\ref{ll1}(iii) that $\{t\}\in D_{\nu'}$ for every $t\in T$. The $\sigma$-additivity of $\nu'$ therefore implies that $\sm{\mf A}{t\in T}{\nu'(\{t\})}$. As $\{x(T)\}_{x\in\mf A}$ is a continuous homomorphism from $\mf A$ to $\R$, Proposition~\ref{p_sum}(ii) ensures that
  \begin{equation}
    \label{eq:sumnu'}
    \sm{\R}{t\in T}{\nu'(\{t\}|T)}.
  \end{equation}
  On the other hand, we have $\{t\}\in \mathcal{Q}$ for every $t\in \R$ and, hence, $\nu'(\{t\}|T) = \nu(\{t\}|T)  = \card\{t\} = 1$ for every $t\in T$. Since $T$ is infinite, it follows that (\ref{eq:sumnu'}) is false. This contradiction proves that $D_{\nu'}\subset\mathcal{Q}$ and, therefore, $\nu$ is an extension of $\nu'$. By Definition~\ref{dd6}, we conclude that $\nu$ is a measure.
\end{proof}

Let $\mathcal{R}$ be the semi-ring consisting of all semi-open intervals of the form $(a,b]$, where $a,b\in\R$ and $a\leq b$, and let $\mathcal{K} = \kappa(\mathcal{R})$. We have $\mathcal{K}\subset\mathcal{Q}$ and $\sigma(\mathcal{K}) = \mathcal{B}$. Let $\mu = \nu|_{\mathcal{K}}$. The $\sigma$-additivity of $\nu$ implies that $\mu$ is a weakly exhaustive $\mf A$-valued $\sigma$-content. Thus, $\mf A$, $\mu$, and $\nu$ satisfy the conditions of Theorem~\ref{t_ext}. We claim that $\mathcal{S}_{\mf A}[\mu]\neq D_\nu$. Since $X\in\mathcal{Q}$, it suffices to show that $X\notin \mathcal{S}_{\mf A}[\mu]$. Suppose the contrary. By~(\ref{eq:Smu}), there is $A\in \sigma_0(\mathcal{K})$ such that $X\subset A$ and $\ex{\mf A}{A}\mu$. Since $A\subset A$, it follows from~(\ref{eq:Smu}) that $A\in \mathcal{S}_{\mf A}[\mu]$ and, hence, $A\in \mathcal{Q}$ by Theorem~\ref{t_ext}. At the same time, we have $A\notin\mathcal{Q}$ by Lemma~\ref{l_notinQ} below. We thus arrive at a contradiction and our claim is proved.

\begin{lemma}\label{l_notinQ}
  Let $A\in\sigma_0(\mathcal{K})$ and let $A$ be unbounded from above in $\R$. Then $A\notin\mathcal{Q}$.
\end{lemma}
\begin{proof}
  Suppose $A\in\mathcal{Q}$. Then the set $B = A\setminus X$ is bounded from above in $\R$. Let $s$ be an upper bound of $B$ in $\R$. Since $A$ is not bounded from above, there is $n\in A$ such that $n > s$. As $A\in\sigma_0(\mathcal{K})$, it follows from Proposition~\ref{p_semiring} that $n\in C$ for some $C\in\mathcal{R}$ such that $C\subset A$. Let $C' = C\cap (s,\infty)$. Then $n\in C'$ and, hence, $C'$ is a nonempty semi-open interval in $\R$ and, in particular, is an uncountable set. On the other hand, $C'$ is countable because $C'\cap B = \varnothing$ and $C'\subset A$ and, therefore,  $C'\subset X$. We thus obtain a contradiction and the statement is proved.  
\end{proof}

\end{document}